\documentclass[final,a4paper,11pt]{article}
\usepackage{graphicx}
\usepackage{mathptmx}
\usepackage{amssymb,amsmath, amsfonts,amsthm}
\usepackage{a4wide}
\usepackage{color}

\usepackage{hyperref}

\newcommand{\email}[1]{{\tt #1}}
\newcommand{\R}{\mathbb{R}}
\newcommand{\oR}{\overline{\R}}
\newcommand{\N}{\mathbb{N}}
\newcommand{\norm}[1]{\|#1\|}

\newcommand{\dist}[1]{{\rm dist}(#1)}

\newcommand{\mv}{\,\mid\,}
\newcommand{\bmv}{\,\big\vert\,}
\newcommand{\Bmv}{\,\Big\vert\,}

\newcommand{\A}{{\cal A}}
\newcommand{\B}{{\cal B}}

\newcommand{\M}{{\cal M}}
\newcommand{\New}{{\cal N}}

\newcommand{\Sp}{{\mathcal S}}
\newcommand{\F}{{\cal F}}
\newcommand{\Z}{{\cal Z}}

\newcommand{\longsetto}[1]{\mathop{\longrightarrow}\limits^{#1}}
\newcommand{\skalp}[1]{\langle #1\rangle}

\newcommand{\xb}{\bar x}
\newcommand{\yb}{\bar y}

\newcommand{\lb}{\bar\lambda}

\newcommand{\AT}[2]{{\textstyle{#1\atop#2}}}
\newcommand{\xba}{{\bar x^\ast}}

\newcommand{\oo}{o}
\newcommand{\OO}{{\cal O}}
\newcommand{\argmin}{\mathop{\rm arg\,min}}

\newcommand{\cl}{{\rm cl\,}}

\newcommand{\inn}{{\rm int\,}}
\newcommand{\bd}{{\rm bd\,}}
\newcommand{\co}{{\rm conv\,}}
\newcommand{\gph}{\mathrm{gph}\,}
\newcommand{\dom}{\mathrm{dom}\,}

\newcommand{\tto}{\rightrightarrows}

\newcommand{\myvec}[1]{\begin{pmatrix}#1\end{pmatrix}}
\newcommand{\SCD}{SCD\ }

\newcommand{\scdreg}{{\rm scd\,reg\;}}
\newcommand{\reg}{{\rm reg\,}}
\newcommand{\ssstar}{semismooth$^{*}$ }
\newcommand{\ee}[2]{{#1}^{(#2)}}
\newcommand{\ZnP}{\Z_n^{P,W}}

\newcommand{\phiFB}[1]{\varphi_{\ee \lambda{#1}}^{\rm FB}(\ee x{#1})}
\newcommand{\psiFBp}[2]{\psi_{\ee \lambda{#1}}^{\rm FB}\big(\ee x{#1},#2\big)}

\newcommand{\rge}{{\rm rge\;}}

\newlength{\myAlgBox}
\newtheorem{theorem}{Theorem}[section]
\newtheorem{proposition}[theorem]{Proposition}
\newtheorem{remark}[theorem]{Remark}
\newtheorem{lemma}[theorem]{Lemma}
\newtheorem{corollary}[theorem]{Corollary}
\newtheorem{definition}[theorem]{Definition}
\newtheorem{example}[theorem]{Example}
\newtheorem{algorithm}{Algorithm}
\newtheorem{assumption}{Assumption}

\title{On a globally convergent semismooth$^{*}$ Newton method in nonsmooth nonconvex optimization}
\author{Helmut Gfrerer\thanks{Institute of Computational Mathematics, Johannes Kepler University
Linz and Johann Radon Institute for Computational and Applied Mathematics (RICAM), Austrian Academy of Sciences, A-4040 Linz, Austria; \email{helmut.gfrerer@ricam.oeaw.ac.at}; ORCID 0000-0001-6860-102X}
\if{ \and   Ji\v{r}\'{i} V. Outrata\thanks{Institute of Information Theory and Automation, Czech Academy of
 Sciences, 18208 Prague, Czech Republic, and Centre for
              Informatics and Applied Optimization, Federation University of Australia, POB 663,
              Ballarat,  Vic 3350, Australia,  \email{outrata@utia.cas.cz}}}\fi
}
\date{}

\begin{document}
\maketitle

{\footnotesize
{\bf Abstract.} In this paper we present GSSN, a globalized SCD semismooth$^{*}$ Newton method for solving  nonsmooth nonconvex optimization problems. The global convergence properties of the method are ensured by the proximal gradient method, whereas locally superlinear convergence is established  via the SCD semismooth$^{*}$ Newton method under quite weak assumptions. The Newton direction is based on the SC (subspace containing) derivative of the subdifferential mapping and can be computed  by the (approximate) solution  of an equality-constrained quadratic program. Special attention is given to the efficient numerical implementation of the overall method.
 }

\section{Introduction}
This paper deals with nonsmooth and nonconvex optimization problems of the form
\begin{equation}\label{EqOptProbl}
  \min \varphi(x):=f(x)+g(x),
\end{equation}
where $f:\R^n\to\R$ is a continuously differentiable function with strictly continuous gradient $\nabla f$ and $g:\R^n\to\oR:=\R\cup\{\infty\}$ is a proper lsc prox-bounded function. Throughout the paper we will assume these properties on $f$ and $g$ without further mention.

For solving such problems, first-order methods are very popular, because they are easy to implement, have low computational cost per iterate and global convergence properties can be shown under very mild assumptions. However, the rate of convergence is at most linear and is very sensitive to some condition number of the problem. Thus, in some cases first-order methods can produce only solutions with low or medium accuracy in reasonable time.

On the other hand, second-order methods are locally superlinearly convergent under suitable smoothness and regularity assumptions at the solution. Having at hand a starting point sufficiently close to the solution, these second-order methods produce in few iterations a solution with high accuracy. However, the computational cost per iterate is  higher and it may be a difficult task to globalize the second-order methods.

Such a globalization can be done by a hybrid of first- and second-order methods. In a series of recent papers \cite{StThPa17,ThAhPa19,ThStPa18,AhThPa21} this was done by using {\em forward-backward splitting (FBS)}  (also called {\em proximal gradient method} (PGM)) as the first-order method. One step of FBS yields a sufficient decrease in the forward-backward envelope (FBE), which can be used to make a line search along a direction obtained by the second-order method. In the mentioned papers, this is done by using a quasi-Newton method for computing a fixed point of the forward-backward operator. This quasi-Newton technique avoids the computation of derivatives for the forward-backward operator. A crucial assumption for guaranteeing superlinear convergence is that the forward-backward operator is strictly differentiable at the solution.

Another approach uses the fact that under suitable assumptions on $g$ (convexity or uniform prox-regularity) the FBE is continuously differentiable and  a descent method combined with some semi\-smooth Newton method can be applied in order to minimize the FBE, cf. \cite{PaBe13, ThAhPa19,KhMoPhTr22b}. Here, the difficult part is the computation of one element of the Clarke generalized derivative of the forward-backward operator and the limiting coderivative of the subdifferential $\partial g$, respectively.

In the  Truncated Nonsmooth Newton Multigrid (TNNMG) method \cite{GraSa19}, global convergence is ensured by the convergence properties of a nonlinear pre-smoother, whereas the subsequent linear correction step is crucial for fast convergence of the method.

Concerning second-order methods, a lot of active investigations have been done in direction of proximal Newton methods, see, e.g., \cite{KaLe21, LeSuSa14, MoYuZeZh22} and \cite{PoeSchiJa22, PoeSchiJa23} for the infinite dimensional case. Here, in each iteration one has to solve a nonsmooth problem of the form
\[\min_d f(\ee xk)+\skalp{\nabla f(\ee xk),d}+\frac 12 d^T\ee Hk d +g(\ee xk +d),\]
where $\ee Hk$ is some approximation of the Hessian $\nabla^2 f(\ee xk)$.

Other second-order methods rely on the observation, that in many cases $\gph \partial \varphi$, the graph of the subdifferential mapping, is a smooth manifold locally near the solution, see, e.g., \cite{BaIuMa22, LeWy20}. To this class  also belongs the theory of $\cal VU$ decomposition as introduced in \cite{LeSa97}. We refer the reader also to the survey \cite{Sa18} and the references therein.

In this paper we want to globalize the \SCD semismooth$^{*}$ Newton method for solving the first-order optimality conditions $0\in\partial \varphi(x)$. The semismooth$^{*}$ Newton method aims at solving inclusions $0\in F(x)$ with set-valued  mappings $F:\R^n\tto\R^n$. It was introduced in \cite{GfrOut21} and is mainly motivated by the semismooth$^{*}$ property which was defined by means of limiting coderivatives. The semismooth$^{*}$ property can be considered as a generalization of the classical semismooth property for single-valued functions \cite{QiSun93}. Let us mention that the class of mappings having the semismoothness$^{*}$ property  is rather broad. As it follows from a result by Jourani \cite{Jou07}, every mapping, whose graph is a closed subanalytic set, is semismooth$^{*}$.

In the very recent paper \cite{GfrOut22a} the \SCD variant (Subspace Containing Derivative) of the \ssstar Newton method has been developed which avoids the use of limiting coderivatives and the SC derivative is comparatively easy to compute. As the name already indicates, the elements of the SC derivative are subspaces. This seems to be quite natural, since for single-valued smooth mappings the Fr\'echet derivative is a linear mapping, whose graph is a subspace. In fact, for single-valued Lipschitzian mappings $F$ the elements of the SC derivative are exactly the graphs of the linear mappings induced by the elements of the so-called B-Jacobian. However, the SC derivative of a set-valued mapping may contain subspaces which are not the graph of a linear mapping. Using the SC derivative, we can then define  the \SCD semismoothness$^{*}$ property which is weaker than the above mentioned one of semismoothness$^{*}$.

The \SCD semismooth$^{*}$ Newton direction is  calculated as the solution of a linear system which is formed by means of an arbitrary basis of a subspace belonging to the SC derivative. Note that we need not to know the whole SC derivative but only one element for computing the Newton direction.

In order that the SC derivative of $\partial \varphi$ exists we have to impose a very mild assumption on $g$, namely that $g$ is prox-regular and subdifferentially continuous at $x$ for $x^*$ for $(x,x^*)$ belonging to a dense subset of $\gph \partial g$. Although there exist examples of proper lsc functions which does not fulfill this assumption, we claim that nearly all functions used in applications have this property.

To ensure superlinear convergence of the \SCD semismooth$^{*}$ Newton method, we have to impose the semismooth$^{*}$ property together with a certain regularity condition on the subdifferential $\partial \varphi$ at the solution. This regularity condition is called \SCD regularity and is weaker than metric regularity. On the other hand, together with the semismooth$^{*}$ property, \SCD regularity implies strong metric subregularity. \SCD regularity also guarantees the unique solvability of the linear system defining the Newton direction. Let us mention that the assumed semismooth${}^*$ property is weaker than the smoothness assumptions used in other papers for ensuring superlinear convergence like strict differentiability of the forward-backward operator \cite{ThStPa18} or the graph of the subdifferential mapping being a smooth manifold around the solution \cite{BaIuMa22,LeWy20}. The \SCD semismooth$^{*}$ Newton method has been already successfully applied to some challenging applications \cite{GfrManOutVal22, GfrOutVal22b}.

The aim of this paper is to globalize the SCD semismooth$^{*}$ Newton method for solving the inclusion $0\in\partial \varphi(x)$ and the obtained algorithm we will call GSSN (Globalized SCD semismooth$^{*}$ Newton method). For globalization we use a hybrid with FBS. In particular, our algorithm is a modification of PANOC${}^+$ developed by De~Marchi and Themelis \cite{MaTh22} and does not require that $\nabla f$ is globally Lipschitz. We are able to show that every accumulation point of the produced sequence is a fixed point of the forward-backward operator and therefore stationary for $\varphi$. Further, if one iterate is sufficiently close to a local minimizer where $\partial\varphi$ is both \SCD semismooth$^{*}$ and \SCD regular, then the whole sequence converges superlinearly to that minimizer.

For the purpose of solving large scale problems we also treat the situation that the linear system defining the Newton direction is only  solved approximately by some iterative method. We will show that the \SCD Newton direction is a stationary solution of a quadratic function and therefore some CG method can be applied.

The paper is organized in the following way. In the preliminary Section 2 we recall first some basic notions from variational analysis followed by some properties of the forward-backward envelope. In addition we recall some basic results on SCD mappings and the SCD semismooth$^{*}$ Newton method for solving inclusions $0\in F(x)$. In Section 3 we analyze the SCD semismooth$^{*}$ Newton method more in detail for the case when $F$ is the subdifferential $\partial \varphi$ of an lsc function $\varphi$. It appears that the subspaces of the SC derivative of $\partial\varphi$ have a certain basis representation which can be efficiently exploited for computation of the Newton direction. In Section 4 we present a globally convergent algorithm which acts as a template for GSSN. In Section 5 we establish superlinear convergence results. The analysis is also done for large-scale problems when the Newton direction cannot be computed exactly but is only approximated by some iterative method. Finally, in Section 6 we present some numerical experiments for large-scale problems which demonstrate the effectiveness of our method.

Our notation is basically standard. For an element $u \in \R^n$, $\norm{u}$ denotes its Euclidean norm,
$\B_\epsilon(u)$ denotes the closed ball around $u$ with radius $\epsilon$ and $\B$ stands for the closed unit ball.  For a matrix $A$, $\rge A$ signifies its range and $\norm{A}$ denotes its spectral norm.  Given a set $\Omega \subset \R^s$, we define the distance
from a point $x$ to $\Omega$ by
$\dist{x, \Omega} := \inf\{\norm{y-x}\mv y\in\Omega\}$, the respective indicator function is denoted by
$\delta_{\Omega}$ and $\stackrel{\Omega}{x \rightarrow \bar{x}}$ means convergence within $\Omega$.
When a mapping $F : \R^n \rightarrow \R^m$ is differentiable at $x$, we denote by $\nabla F(x)$ its Jacobian. We denote by $\R_+$ ($\R_{++}$) the set of all nonnegative (positive) real numbers and by $\R_-$ ($\R_{--}$) the set of all nonpositive (negative) real numbers.

\section{\label{SecPrel}Preliminaries}
\subsection{Variational analysis}
Given a set $A\subset\R^n$ and a point $\xb\in A$, the {\em tangent cone} to $A$ at $\xb$ is defined as
\[T_A(\xb):=\limsup_{t\downarrow 0}\frac {A-\xb}t.\]
We call $A$ {\em geometrically derivable} at $\xb$, if $T_A(\xb)=\lim_{t\downarrow 0}(A-\xb)/t$, i.e., for every $u\in T_A(\xb)$ and every sequence $t_k\downarrow 0$ there exits a sequence $u_k\to u$ such that $\xb+t_ku_k\in A$.

Consider now a function $q:\R^n\to\oR$ and a point $\xb\in \dom q:=\{x\in\R^n\mv q(x)<\infty\}$.
The {\em regular subdifferential} of $q$ at $\xb$ is given by
\[\widehat\partial q(\xb):=\Big\{x^*\in\R^n\mv\liminf_{x\to\xb}\frac{q(x)-q(\xb)-\skalp{x^*,x-\xb}}{\norm{x-\xb}}\geq 0\Big\}\]
while the {\em (limiting) subdifferential} is defined by
\[\partial q(\xb):=\{x^*\mv \exists (x_k, x_k^*)\to(\xb,x^*) \mbox{ with }x_k^*\in\widehat \partial q(x_k)\ \forall k\mbox{ and }\lim_{k\to\infty} q(x_k)=q(\xb)\}.\]
Given $\xba\in\partial q(\xb)$, the function $q$ is called {\em prox-regular} at $\xb$ for $\xba$ if there
exist $\epsilon> 0$ and $\rho\geq 0$ such that for all $x',x\in \B_{\epsilon}(\xb)$ with $\vert q(x)-q(\xb)\vert\leq \epsilon$ one has
\[q(x')\geq q(x)+\skalp{x^*,x'-x}-\frac \rho2 \norm{x'-x}^2\quad\mbox{whenever}\quad x^*\in \partial q(x)\cap \B_\epsilon(\xba).\]
When this holds for all $\xba \in \partial q(\xb)$, $q$ is said to be prox-regular at $\xb$. Further,  $q$ is called {\em subdifferentially continuous} at $\xb$ for $\xba$ if for any sequence $(x_k,x_k^*)\longsetto{{\gph \partial q}}(\xb,\xba)$ we have $\lim_{k\to\infty}q(x_k)=q(\xb)$. When this holds for all $\xba \in \partial q(\xb)$, $q$ is said to be subdifferentially continuous  at $\xb$.

The function $q$ is said to be {\em prox-bounded} if there exists some $\lambda>0$ such that $\inf_z \frac1{2\lambda}\norm{z-x}^2+q(z)>-\infty$ for some $x\in\R^n$. The supremum of the set of all such $\lambda$ is the threshold $\lambda_q$ for the prox-boundedness of $q$.
Further, the {\em proximal mapping} ${\rm prox}_{\lambda q}:\R^n\tto\R^n$ is defined by
\[{\rm prox}_{\lambda q}(x):=\argmin_z \frac1{2\lambda}\norm{z-x}^2+q(z).\]

Given a mapping $F:D\to\R^m$ with $D\subset\R^n$, we say that $F$ is {\em strictly continuous} at $\xb\in D$ if
\[{\rm Lip\,} F(\xb):=\limsup_{\AT{x,x'\longsetto{D}\xb}{x\not=x'}}\frac{\norm{F(x)-F(x')}}{\norm{x-x'}}<\infty.\]
The mapping $F$ is called strictly continuous if it is strictly continuous at every point $x\in D$.

\subsection{On the forward-backward envelope}

Consider the problem \eqref{EqOptProbl}. Given $\lambda>0$, we define the function $\psi_\lambda^{\rm FB}:\R^n\times\R^n\to \oR$ by
\[\psi_\lambda^{\rm FB}(x,z):=\ell_f(x,z)+\frac1{2\lambda}\norm{z-x}^2+g(z)\mbox{  with } \ell_f(x,z):=f(x)+\skalp{\nabla f(x),z-x}.\]
Further we define the mapping $T_\lambda:\R^n\tto\R^n$ and the function $\varphi_\lambda^{\rm FB}: \R^n\to\R$ by
\begin{align*}T_\lambda(x)&:=\argmin_z\psi_\lambda^{\rm FB}(x,z),\\
\varphi_\lambda^{\rm FB}(x)&:=\min_z \psi_\lambda^{\rm FB}(x,z).
\end{align*}
Clearly, for any $z\in T_\lambda(x)$ there holds the first-order optimality condition
\begin{equation}\label{EqFO}
  0\in\widehat\partial_z \psi_\lambda^{\rm FB}(x,z)=\nabla f(x)+\frac 1\lambda(z-x)+\widehat\partial g(z).
\end{equation}
In particular, whenever $x\in T_\lambda(x)$, we have $0\in\nabla f(x)+\widehat\partial g(x)=\widehat\partial\varphi(x)$, i.e., $x$ is a stationary point for $\varphi$.

The function $\varphi_\lambda^{\rm FB}$ is the so-called {\em forward-backward envelope}, which has been introduced in \cite{PaBe13} under the name {\em composite Moreau envelope}, and has been successfully used in a series of papers for the solution of nonsmooth optimization problems, cf. \cite{StThPa17, ThStPa18, StThPa19}.

In what follows we  will also  use of the well-known descent lemma:
\begin{lemma}[{see, e.g., \cite[Lemma 5.7]{Be17}}]\label{LemDesc}Let $\nabla f$ be Lipschitz continuous with constant $L$ on the convex set $D$. Then for any $x,y\in D$ one has
\[f(y)\leq \ell_f(x,y)+\frac L2\norm{y-x}^2.\]
\end{lemma}
Note that by the proof of \cite[Lemma 5.7]{Be17} we even have $\vert f(y)- \ell_f(x,y)\vert \leq \frac L2\norm{y-x}^2$, $x,y\in D$.

Let us state some basic properties of $T_\lambda$ and $\varphi_\lambda^{\rm FB}$. First at all note that
\[T_\lambda(x)={\rm prox}_{\lambda g}(x-\lambda\nabla f(x)).\]
\begin{lemma}[{cf. \cite[Lemma 2.2]{MaTh22}}]\label{LemBasFBE1}Let $\lambda\in(0,\lambda_g)$, where $\lambda_g$ denotes the threshold of prox-boundedness of $g$.\begin{enumerate}
  \item[(i)] $\varphi_\lambda^{\rm FB}$ is everywhere finite and strictly continuous.
  \item[(ii)] For every $x\in\R^n$ one has $\varphi_\lambda^{\rm FB}(x)\leq \varphi(x)$, where equality holds if and only if $x\in T_\lambda(x)$.
  \item[(iii)] If $\tilde x\in T_\lambda(x)$ and $f(\tilde x)\leq \ell(x,\tilde x)+\frac L2\norm{\tilde x-x}^2$ then
  \begin{align*}
    \varphi(\tilde x)\leq \varphi_\lambda^{\rm FB}(x)-\frac{1-\lambda L}{2\lambda}\norm{\tilde x-x}^2.
  \end{align*}
\end{enumerate}
\end{lemma}
\begin{lemma}\label{LemBasicFBE}
  \begin{enumerate}
    \item[(i)] For every bounded set $M$ and every $\alpha>0$ there exists some $\hat\lambda_{M,\alpha}>0$ such that for every  $\lambda\in(0,\hat\lambda_{M,\alpha}]$ and every pair $(x,z)\in M\times M$ one has
    \[f(z)\leq \ell_f(x,z)+\frac{\alpha}{2\lambda}\norm{z-x}^2.\]
    \item[(ii)] For every bounded set $\bar M$ and every $\bar\lambda\in(0,\lambda_g)$ the set
    \begin{equation}\label{EqS} S:=\bigcup_{\AT{x\in \bar M}{\lambda\in(0,\bar\lambda]}}T_\lambda(x)\end{equation}
    is bounded.
    \item[(iii)]If $\lambda_1<\lambda_2$ then for all $x\in\R^n$ there holds $\varphi^{\rm FB}_{\lambda_2}(x)\leq \varphi^{\rm FB}_{\lambda_1}(x)$.
  \end{enumerate}
\end{lemma}
\begin{proof}
  ad (i): Since $\nabla f$ is assumed to be strictly continuous, $\nabla f$ is Lipschitz continuous on the convex compact set $\cl\co M$ with some constant $L$, cf. \cite[Theorem 9.2]{RoWe98}. By Lemma \ref{LemDesc}  there holds $f(z)\leq \ell(x,z)+\frac L2\norm{z-x}^2$ for all $x,z\in \cl\co M$ and the assertion follows with $\hat\lambda_{M,\alpha} = \alpha/L$.

  ad (ii): Assume on the contrary that the set $S$ is unbounded, i.e., there exist sequences $(x_k,\lambda_k)\in \bar M\times (0,\bar\lambda]$ and $z_k\in T_{\lambda_k}(x_k)$ with $\norm{z_k}\to \infty$ as $k\to\infty$. Taking any $\tilde z\in \dom g$ we have
  \[f(x_k)+\skalp{\nabla f(x_k),z_k-x_k}+\frac 1{2\lambda_k}\norm{z_k-x_k}^2+g(z_k)\leq f(x_k)+\skalp{\nabla f(x_k),\tilde z-x_k}+\frac 1{2\lambda_k}\norm{\tilde z-x_k}^2+g(\tilde z).\]
  A rearranging and the Cauchy-Schwarz inequality yield
  \begin{align*}g(\tilde z)&\geq \skalp{\nabla f(x_k)-\frac {x_k}{\lambda_k},z_k-\tilde z}+\frac 1{2\lambda_k}(\norm{z_k}^2-\norm{\tilde z}^2)+g(z_k)\\
  &\geq -\frac1{\lambda_k}\norm{\lambda_k\nabla f(x_k)-x_k}\norm{z_k-\tilde z}+\frac 1{2\lambda_k}(\norm{z_k}^2-\norm{\tilde z}^2)+g(z_k).\end{align*}
  Since the sequences $x_k$ and $\lambda_k$ are bounded and $\nabla f$ is continuous, the sequence $\norm{\lambda_k\nabla f(x_k)-x_k}$ is bounded as well. By taking into account $\lim_{k\to\infty}\norm{z_k}=\infty$, for arbitrary $\epsilon>0$ we can find some $k_\epsilon$ such that for any $k\geq k_\epsilon$ one has
  \[\norm{\lambda_k\nabla f(x_k)-x_k}\norm{z_k-\tilde z}\leq \frac\epsilon2\norm{z_k}^2\mbox{ and } \norm{\tilde z}^2\leq\epsilon\norm{z_k}^2,\]
  resulting in
  \[g(\tilde z)\geq \frac{1-2\epsilon}{2\lambda_k}\norm{z_k}^2+g(z_k)\geq \frac{1-2\epsilon}{2\bar\lambda}\norm{z_k}^2+g(z_k).\]
  Taking $\epsilon>0$ small enough such that $\bar\lambda/(1-2\epsilon)<\lambda_g$ we obtain $\liminf_{k\to\infty}g(z_k)/\norm{z_k}^2<-1/(2\lambda_g)$. However,
  by \cite[Exercise 1.24]{RoWe98} there holds $\liminf_{k\to\infty}g(z_k)/\norm{z_k}^2\geq -1/(2\lambda_g)$, a contradiction.

  ad (iii): This is an easy consequence of $\psi^{\rm FB}_{\lambda_2}(x,\cdot)- \psi^{\rm FB}_{\lambda_1}(x,\cdot)=\big(\frac 1{\lambda_2}-\frac 1{\lambda_1}\big)\norm{\cdot-x}^2\leq 0$, $x\in\R^n$.
\end{proof}

\subsection{On \SCD mappings}

In this subsection we want to recall the basic definitions and features  of the \SCD property  introduced in the recent paper \cite{GfrOut22a}.

In what follows we denote by $\Z_n$ the metric space of all $n$-dimensional subspaces of $\R^{2n}$ equipped with the metric
\[d_\Z(L_1,L_2):=\norm{P_{L_1}-P_{L_2}}\]
where $P_{L_i}$ is the symmetric $2n\times 2n$ matrix representing the orthogonal projection on $L_i$, $i=1,2$. $\Z_n$ is a compact metric space, cf. \cite[Lemma 3.1]{GfrOut22a}

A sequence $L_k\in \Z_n$ converges in $\Z_n$ to some $L\in \Z_n$, i.e. $\lim_{k\to\infty}d_\Z(L_k,L)=0$, if and only if $\lim_{k\to\infty}L_k=L$ in the sense of Painlev\'e-Kuratowski convergence, cf. \cite[Lemma 3.1]{GfrOut22a}.

We treat every element of $\R^{2n}$ as a column vector. In order to keep notation simple we write $(u,v)$ instead of $\myvec{u\\v}\in\R^{2n}$ when this does not lead to confusion.

We will also work with bases for the subspaces $L\in\Z_n$. Let $\M_n$ denote the collection of all $2n\times n$ matrices with full column rank $n$ and for $L\in \Z_n$ we define
\[\M(L):=\{Z\in \M_n\mv \rge Z =L\},\]
i.e., the columns of $Z\in\M(L)$ are a basis for $L$.
We can partition every matrix $Z\in\M(L)$ into two $n\times n$ matrices $A$ and $B$ and we will write $Z=(A,B)$ instead of $Z=\myvec{A\\B}$. It follows that $\rge(A,B):=\{(Au,Bu)\mv u\in\R^n\}\doteq \{\myvec{Au\\Bu}\mv u\in\R^n\}=L$.

Further, for every $L\in \Z_n$ we can define the adjoint subspace
\begin{align*}
  L^*&:=\{(-v^*,u^*)\mv (u^*,v^*)\in L^\perp\},
\end{align*}
where $L^\perp$ denotes as usual the orthogonal complement of $L$. Then it can be shown that $(L^*)^*=L$ and $d_\Z(L_1,L_2)=d_\Z(L_1^*,L_2^*)$. Thus the mapping $L\to L^*$ defines an isometry on $\Z_n$.

\begin{definition}\label{DefSCDProperty}
  Consider a mapping  $F:\R^n\tto\R^n$.
  \begin{enumerate}
    \item We call $F$ {\em graphically smooth of dimension $n$} at $(x,y)\in \gph F$, if
    \[T_{\gph F}(x,y)\in \Z_n.\] Further we denote by $\OO_F$ the set of all points belonging to $\gph F$ where $F$ is graphically smooth of dimension $n$.
    \item We associate with $F$ the {\em SC (subspace containing) derivative} $\Sp F:\gph F\tto \Z_n$ and the {\em adjoint SC derivative} $\Sp^* F:\gph F\tto \Z_n$, given by
    \begin{align*}
    \Sp F(x,y)&:=\{L\in \Z_n\mv \exists (x_k,y_k)\longsetto{{\OO_F}}(x,y):\ \lim_{k\to\infty} d_\Z(L,T_{\gph F}(x_k,y_k))=0\},\\
    \Sp^* F(x,y)&:=\{L^*\mv L\in\Sp F(x,y)\}.
    \end{align*}
    \item We say that $F$ has the {\em\SCD (subspace containing derivative) property at} $(x,y)\in\gph F$, if $\Sp F(x,y)\not=\emptyset$. We say that $F$ has the \SCD property {\em around} $(x,y)\in\gph F$, if there is a neighborhood $W$ of $(x,y)$ such that $F$ has the \SCD property at every $(x',y')\in\gph F\cap W$.
     Finally, we call $F$ an {\em \SCD mapping} if
    $F$ has the \SCD property at every point of its graph.
  \end{enumerate}
\end{definition}

The (adjoint) SC derivative was introduced in \cite{GfrOut22a}. Note that the definition of $\Sp^*F$ in \cite{GfrOut22a} is different but equivalent to the definition above, cf. \cite[Remark 3.5]{GfrOut22a}.

In this paper we are concerned with the \SCD property  of subdifferential mappings.
\begin{proposition}[{cf. \cite[Proposition 3.26]{GfrOut22a}}]\label{PropProxRegularQ}
      Suppose that $q:\R^n\to\oR$ is prox-regular and subdifferentially continuous at $\xb$ for $\xb^*\in\partial q(\xb)$. Then $\partial q$ has the \SCD property around $(\xb,\xba)$.  Further, there exists some $\lambda>0$ such that for every $(x,x^*)\in\gph\partial q$ sufficiently close to $(\xb,\xba)$ there holds $\Sp\partial q(x,x^*)=\Sp^*\partial q(x,x^*)$ and for every $L\in\Sp^*\partial q(x,x^*)$  there is a symmetric positive semidefinite $n\times n$ matrix $B$ such that $L=L^*=\rge(B,\frac 1\lambda(I- B))$.
\end{proposition}

Next we turn to the notion of \SCD regularity.
\begin{definition}
\begin{enumerate}
\item We denote by $\Z_n^{\rm reg}$ the collection of all subspaces $L\in\Z_n$ such that
  \begin{equation*}
    (y^*,0)\in L\ \Rightarrow\ y^*=0.
  \end{equation*}
  \item  A mapping $F:\R^n\tto\R^n$ is called {\em \SCD regular} at $(x,y)\in\gph F$, if $F$ has the \SCD property at $(x,y)$ and
  \begin{equation}\label{EqSCDReg}
    (y^*,0)\in L^* \Rightarrow\ y^*=0\ \forall L^*\in \Sp^*F(x,y),
  \end{equation}
  i.e., $L^*\in \Z_n^{\rm reg}$ for all $L^*\in \Sp^* F(x,y)$. Further, we will denote by
  \[{\rm scd\,reg\;}F(x,y):=\sup\{\norm{y^*}\mv (y^*,x^*)\in L^*, L^*\in \Sp^*F(x,y), \norm{x^*}\leq 1\}\] 
  the {\em modulus of \SCD regularity} of $F$ at $(x,y)$.
\end{enumerate}
\end{definition}
Since the elements of $\Sp^*F(x,y)$ are contained in $\gph D^*F(x,y)$, cf. \cite[Lemma 3.7]{GfrOut22a}, where $D^*F(x,y)$ denotes the limiting coderivative of $F$ at $(x,y)$,  it follows from the Mordukhovich criterion, see, e.g., \cite[Theorem 3.3]{Mo18}, that \SCD regularity is weaker than metric regularity. We refer to \cite{GfrOut22a} for an example demonstrating that SCD regularity is strictly weaker than metric regularity.

In the following proposition we state some basic properties of subspaces $L\in\Z_n^{\rm reg}$.
\begin{proposition}[cf.{\cite[Proposition 4.2]{GfrOut22a}}] \label{PropC_L}
    Given a $2n\times n$ matrix $Z=(A,B)$, there holds $\rge Z\in \Z_n^{\rm reg}$ if and only if the $n\times n$ matrix $B$ is nonsingular. Thus,
    for every $L^*\in \Z_n^{\rm reg}$   there is a  unique $n\times n$ matrix $C_{L^*}$  such that $L^*=\rge(C_{L^*},I)$. Further, $L=\rge(C_{L^*}^T,I)\in\Z_n^{\rm reg}$, $C_L=C_{L^*}^T$,
    \begin{equation*}
    \skalp{x^*,C_{L^*}^Tv}=\skalp{y^*,v}\ \forall (y^*,x^*)\in L^*\ \forall v\in\R^n
    \end{equation*}
    and
    \begin{equation*}
    \norm{y^*}\leq \norm{C_{L^*}}\norm{x^*}\ \forall (y^*,x^*)\in L^*.
  \end{equation*}
\end{proposition}
Note that for every $L^*\in\Z_n^{\reg}$ and every $(A,B)\in \M(L^*)$ the matrix $B$ is nonsingular and $C_{L^*}=AB^{-1}$.

The next lemma follows  from \cite[Lemma 4.7, Proposition 4.8]{GfrOut22a}.
\begin{lemma}\label{Lem_scdreg}
   Assume that $F:\R^n\tto\R^n$ is \SCD regular at $(\xb,\yb)\in\gph F$. Then
   \[{\rm scd\,reg\;}F(\xb,\yb)=\sup\{\norm{C_{L^*}}\mv L^*\in\Sp^*F(\xb,\yb)\}<\infty.\]
   Moreover, if $F$ has the \SCD property around $(\xb,\yb)$, then $F$ is \SCD regular at every $(x,y)\in\gph F$ sufficiently close to $(\xb,\yb)$ and
  \[\limsup_{(x,y)\longsetto{\gph F}(\xb,\yb)}{\rm scd\,reg\;}F(x,y)\leq{\rm scd\,reg\;}F(\xb,\yb).\]
\end{lemma}

As an example of \SCD regularity we refer to the subdifferential mapping around tilt-stable local minimizers.
\begin{definition} [\bf tilt-stable minimizers]\label{DefTiltStab} Let $q:\R^n\to\oR$, and let $\xb\in\dom q$. Then $\xb$ is a {\em tilt-stable local minimizer} of $q$ if there is a number $\gamma>0$ such that the mapping
\begin{equation}\label{EqM_gamma}
M_\gamma(x^\ast):=\argmin\big\{q(x)-\skalp{x^\ast,x}\mv x\in\B_\gamma(\xb)\big\},\quad x^*\in\R^n,
\end{equation}
is single-valued and Lipschitz continuous in some neighborhood of $\bar x^*=0\in\R^n$ with $M_\gamma(0)=\{\xb\}$.
\end{definition}

\begin{theorem}[{\cite[Theorem 7.9]{GfrOut22a}}]\label{ThTiltStab}
  For a function $q:\R^n\to \oR$ having $0\in\partial q(\xb)$ and such that $q$ is both prox-regular and subdifferentially continuous at  $\xb$ for $\xba=0$, the following statements are equivalent:
  \begin{enumerate}
  \item[(i)] $\xb$ is a tilt-stable local minimizer of $q$.
  \item[(ii)] $\partial q$ is \SCD regular at $(\xb,0)$ and $C_L$ is positive semidefinite for every $L\in\Sp \partial q(\xb,0)=\Sp^* \partial q(\xb,0)$.
  \end{enumerate}
\end{theorem}
Note that the subdifferential mapping $\partial q$ can be \SCD regular also at  local minimizers which are not tilt-stable, cf. \cite[Example 6.5]{GfrOut22a}.

\subsection{On the \SCD semismooth$^{*}$ Newton method}

In this subsection we outline  the \SCD semismooth$^{*}$ Newton method presented in  the recent paper \cite{GfrOut22a} for solving the inclusion
\[0\in F(x),\]
where $F:\R^n\tto\R^n$ is a set-valued mapping.

\begin{definition}\label{DefSCDssstar}
We say that $F:\R^n\tto\R^n$ is {\em\SCD \ssstar at} $(\xb,\yb)\in\gph F$ if $F$ has the \SCD property around $(\xb,\yb)$ and
for every $\epsilon>0$ there is some $\delta>0$ such that
\begin{align}\label{EqDefSCDSemiSmooth}
\vert \skalp{x^*,x-\xb}-\skalp{y^*,y-\yb}\vert&\leq \epsilon
\norm{(x,y)-(\xb,\yb)}\norm{(x^*,y^*)}
\end{align}
holds for all $(x,y)\in \gph F\cap \B_\delta(\xb,\yb)$ and all $(y^*,x^*)\in L^*$, $L^* \in \Sp^*F(x,y)$.
\end{definition}

The SCD \ssstar property has also a useful primal interpretation:

\begin{proposition}
  Let the  mapping $F:\R^n\tto\R^m$ and the pair $(\xb,\yb)\in\gph F$ be given. Then the following two statements are equivalent:
  \begin{enumerate}
    \item $F$ is SCD \ssstar at $(\xb,\yb)$.
    \item For every $\epsilon>0$ there is some $\delta>0$ such that for every $(x,y)\in \gph F\cap\B_\delta(\xb,\yb)$ and every $L\in \Sp F(x,y)$ there holds
    \[\dist{(x-\xb,y-\yb),L}\leq\epsilon\norm{(x-\xb,y-\yb)}.\]
  \end{enumerate}
\end{proposition}
  \begin{proof}
    Consider $(x,y)\in\gph F$ and $L\in \Sp F(x,y)$. By the well-known distance formula we have
    \begin{align*}\dist{(x-\xb,y-\yb),L}&=\sup\{\skalp{x^*,x-\xb}+\skalp{y^*,y-\yb}\mv  (x^*,y^*)\in L^\perp, \norm{(x^*,y^*)}=1\}\\
    &=\sup\{\vert \skalp{x^*,x-\xb}-\skalp{y^*,y-\yb}\vert\mv (y^*,x^*)\in L^*, \norm{(y^*,x^*)}=1\}
    \end{align*}
    and the equivalence follows now from Definition \ref{EqDefSCDSemiSmooth}.
  \end{proof}

The notion of \SCD semismoothness$^{*}$ is weaker than the one of semismoothness$^{*}$ introduced in \cite{GfrOut21}. Thus, we obtain from \cite[Proposition 2.10]{GfrOut22a} the following
result.
\begin{proposition}
  Let $F:\R^n\tto\R^n$ be an \SCD mapping. Each of the following two conditions is sufficient for $F$ to be \SCD \ssstar at every point of its graph:
  \begin{enumerate}
    \item[(i)]$\gph $F is the union of finitely many convex sets.
    \item[(ii)]$\gph F$ is a closed subanalytic set.
  \end{enumerate}
\end{proposition}
\begin{example}
  Consider $g(x)=\norm{x}$ on $\R^n$. Since
  \[\gph \partial g=\left\{(x,x^*)\in\R^n\times\R^n\mv \sum_{i=1}^n {x_i^*}^2\leq 1, \Big(\sum_{i=1}^nx_ix_i^*\Big)^2=\sum_{i=1}^nx_i^2, \sum_{i=1}^nx_ix_i^*\geq 0\right\}\]
  is a closed subanalytic set and the subdifferential of any proper lsc convex function is an \SCD mapping by Proposition \ref{PropProxRegularQ}, $\partial g$ is \SCD \ssstar at every point of its graph.
\end{example}
The basic idea behind the \SCD semismooth$^{*}$ Newton method is as follows. Assume that the mapping $F$ is \SCD \ssstar at $(\xb,0)\in\gph F$ and consider a point $(x,y)\in\gph F$ close to $(\xb,0)$ together with a subspace $L^*\in \Sp^*F(x,y)$ and $n$ pairs $(y_i^*,x_i^*)\in\R^n\times\R^n$, $i=1,\ldots,n$, forming a basis of $L^*$. Then, by \eqref{EqDefSCDSemiSmooth}, there holds
\[\skalp{x_i^*,\xb-x}=-\skalp{y_i^*,y}+\oo(\norm{(x,y)-(\xb,0)})\norm{(y_i^*,x_i^*)},\ i=1,\ldots,n\]
and we expect that, under suitable assumptions, a unique solution of the system of linear equations
\[\skalp{x_i^*,\Delta x}=-\skalp{y_i^*,y},\ i=1,\ldots,n\]
exists and satisfies $\norm{x+\Delta x-\xb}=\oo(\norm{(x,y)-(\xb,0)})$.

Let us formalize this idea. Consider an iterate $\ee xk$. Since we are dealing with multifunctions, we cannot expect that $F(\ee xk)\not=\emptyset$ or that $0$ is close to $F(\ee xk)$, even if $\ee xk$ is very close to the solution $\xb$. Hence, we first perform a so-called {\em approximation step} by finding a pair $(\ee {\hat x}k,\ee{\hat y}k)$ satisfying
\begin{equation}
  \label{EqApprStepF}\norm{(\ee{\hat x}k,\ee{\hat y}k)-(\xb,0)}\leq\eta\norm{\ee xk-\xb}
\end{equation}
for some parameter $\eta>0$. This bound involves the unknown solution $\xb$. However, if we compute $(\ee {\hat x}k,\ee{\hat y}k)$ as some approximate projection of $(\ee xk,0)$ on $\gph F$, i.e.,
\[\norm{(\ee {\hat x}k,\ee{\hat y}k)-(\ee xk,0)}\leq \beta\dist{(\ee xk,0),\gph F}\]
for some $\beta\geq 1$, then
\begin{align*}\norm{(\ee{\hat x}k,\ee{\hat y}k)-(\xb,0)}&\leq \norm{(\ee {\hat x}k,\ee{\hat y}k)-(\ee xk,0)}+\norm{(\ee xk,0)-(\xb,0)}\\
&\leq\beta\dist{(\ee xk,0),\gph F}+\norm{(\ee xk,0)-(\xb,0)}\leq(\beta+1)\norm{(\ee xk,0)-(\xb,0)}\end{align*}
and therefore \eqref{EqApprStepF} holds with $\eta=\beta+1$. Then we perform the so-called {\em Newton step}. We choose some subspace $\ee {L^*}k\in \Sp^*F(\ee{\hat x}k,\ee{\hat y}k)$ together with some basis $(\ee Bk,\ee Ak)\in\M(\ee {L^*}k)$ and compute the Newton direction $\Delta\ee xk$ as solution of the linear system
\begin{equation}\label{EqNewtonStep}{\ee Ak}^T\Delta x=-{\ee Bk}^T\ee {\hat y}k.\end{equation}
By Proposition \ref{PropC_L}, the matrix $\ee Ak$ is nonsingular if and only if $\ee {L^*}k\in \Z_n^{\rm reg}$ and in this case we have
\begin{equation*}\Delta \ee xk=-{\ee Ak}^{-T}{\ee Bk}^T\ee {\hat y}k=-C_{\ee {L^*}k}^T\ee {\hat y}k.\end{equation*}
An equivalent approach for computing the Newton direction is based on the relation $C_L=C_{L^*}^T$. Choosing a subspace $\ee Lk\in\Sp F(\ee{\hat x}k,\ee{\hat y}k)$ and a basis $(\ee Xk,\ee Yk)\in \M(\ee Lk)$, the Newton direction is given by
\[\ee {\Delta x}k=\ee Xk p\quad\mbox{where $p$ solves}\quad \ee Yk p=-\ee {\hat y}k.\]
\\
Finally we get the new iterate as $\ee x{k+1}=\ee{\hat x}k+\Delta\ee xk$.

The following proposition provides the key estimate for the \ssstar Newton method for \SCD mappings.
\begin{proposition}[cf. {\cite[Proposition 5.3]{GfrOut22a}}]\label{PropConvNewton}
  Assume that $F:\R^n\tto\R^n$ is \SCD \ssstar at $(\xb,\yb)\in\gph F$. Then for every $\epsilon>0$ there is some $\delta>0$ such that the estimate
  \begin{equation*}
  \norm{x-C_{L^*}^T(y-\yb)-\xb}\leq \epsilon\sqrt{n(1+\norm{C_{L^*}}^2)}\norm{(x,y)-(\xb,\yb)}
  \end{equation*}
  holds for every $(x,y)\in\gph F\cap \B_\delta(\xb,\yb)$ and every $L^*\in\Sp^*F(x,y)\cap\Z_n^{\rm reg}$.
\end{proposition}

In order to state a convergence result, we introduce the following set-valued mappings
\begin{gather*}
  \A_{\eta,\xb}(x):=\{(\hat x,\hat y)\in\gph F\mv \norm{(\hat x,\hat y)-(\xb,0)}\leq \eta\norm{x-\xb}\},\\
  \New_{\eta,\xb}(x):=\{\hat x-C_{L^*}^T\hat y\mv (\hat x,\hat y)\in \A_{\eta,\xb}(x), L^*\in\Sp^*F(\hat x,\hat y)\cap \Z_n^{\rm reg}\}.
\end{gather*}
Note that the Newton iteration described above can be shortly written as $\ee x{k+1}\in\New_{\eta,\xb}(\ee xk)$.
\begin{proposition}[{\cite[Proposition 4.3]{GfrOutVal22b}}]\label{PropSingleStep}
  Assume that $F$ is \SCD \ssstar  and \SCD regular at $(\xb,0)\in\gph F$ and let $\eta>0$. Then there is some $\bar\delta>0$  such that for every $x\in \B_{\bar\delta}(\xb)$ the mapping $F$ is \SCD regular at every point $(\hat x,\hat y)\in \A_{\eta,\xb}(x)$. Moreover, for every $\epsilon>0$ there is some $\delta\in(0,\bar\delta]$ such that
  \[\norm{z-\xb}\leq\epsilon\norm{x-\xb}\ \forall x\in \B_\delta(\xb)\ \forall z\in \New_{\eta,\xb}(x).\]
\end{proposition}
It now easily follows that, under the assumptions of Proposition \ref{PropSingleStep}, for every starting point $\ee x0$ sufficiently close to the solution $\xb$ the \SCD \ssstar Newton method is well-defined and any sequence $\ee xk$ produced by this method converges superlinearly to $\xb$.

We also need the following proposition which easily follows from \cite[Theorem 6.2]{GfrOut22a}.
\begin{proposition}\label{Prop_sSR}
  Assume that $F$ is \SCD \ssstar at $(\xb,0) \in\gph F$ and \SCD regular at $(\xb,0)$. Then for every $\kappa>\scdreg F(\xb,0)$ there is some neighborhood $U$ of $\xb$ such that
  \[\norm{x-\xb}\leq \kappa \dist{0,F(x)}\ \forall x\in U,\]
  i.e., $F$ is strongly metrically subregular at $(\xb,0)$, cf. \cite[Chapter 3.9]{DoRo14}.
\end{proposition}

\section{On the SCD semismooth$^{*}$ Newton step in optimization}

As explained in the previous section, an iteration of the SCD \ssstar Newton method consists of an approximation step and a Newton step. In this section we analyze the Newton step for solving the inclusion $0\in \partial\varphi(x)$ for rather general lsc functions $\varphi:\R^n\to\oR$ more in detail. The discussion of the approximation step is shifted to the next sections, where we will show that the approximation step can be realized by our globalization procedure, i.e., by essentially one step of PGM.

Proposition \ref{PropProxRegularQ} states that the SC derivative of prox-regular and subdifferentially continuous functions $\varphi$ have the favourable property that every subspace $L$ belonging to the SC derivative is self-adjoint, i.e., $L=L^*$.  Since we want to work with such subspaces, we make the  following assumption, which we assume from now on for the rest of this section.
\begin{assumption}\label{AssProx}The set
\[{\rm proxreg\,}\varphi:=\{(x,x^*)\in\gph \partial\varphi\mv \mbox{ $\varphi$ is prox-regular and subdifferentially continuous at $x$ for $x^*$}\}\]
is dense in $\gph\partial\varphi$.
\end{assumption}
\begin{lemma}$\partial \varphi$ is an \SCD mapping.
\end{lemma}
\begin{proof}
  Consider an arbitrary pair $(\xb,\xba)\in\gph \partial\varphi$. By Assumption \ref{AssProx},  there exists a sequence $(x_k,x_k^*)$ in ${\rm proxreg\,}\varphi$ converging to $(\xb,\xba)$  and, by Proposition \ref{PropProxRegularQ}, for every $k$ there is some $L_k\in\Sp\partial\varphi(x_k,x_k^*)$. The metric space $\Z_n$ is compact and therefore the sequence $L_k$ has a limit point $\bar L\in\Z_n$. Since $(x_k,x_k^*)\to(\xb,\xba)$, it follows that $\bar L\in\Sp\partial \varphi(\xb,\xba)\not=\emptyset$ by \cite[Lemma 3.13]{GfrOut22a}. Since this holds for arbitrary $(\xb,\xba)\in\gph \partial\varphi$, the subdifferential $\partial \varphi$ is an \SCD mapping.
\end{proof}
In order to implement the SCD \ssstar Newton step, we need a basis for the selected subspace $\ee Lk$. Instead of the basis representation provided by Proposition \ref{PropProxRegularQ} we will use another one, which seems to be more appropriate.
\begin{definition}\label{DefZnP}
  We denote by $\ZnP$ the collection of all subspaces $L\in\Z_n$ such that there are symmetric $n\times n$ matrices $P$ and $W$ with the following properties:
  \begin{enumerate}
  \item[(i)] $L=\rge(P,W)$,
  \item[(ii)] $P^2=P$, i.e.. $P$ represents the orthogonal projection onto some subspace of $\R^n$,
  \item[iii)] $W(I-P)=I-P$.
  \end{enumerate}
\end{definition}
Note that for symmetric matrices $W,P$ fulfilling conditions (ii) and (iii) of the above definition  we have
\begin{gather}\label{EqPW1}(I-P)W=\big(W(I-P)\big)^T=(I-P)^T=I-P=W(I-P),\\
\label{EqPW2}PW=WP=PWP,\\
\label{EqPW3}W=WP+W(I-P)=PWP+(I-P).
\end{gather}
\begin{lemma}\label{LemSelfadj}
  Every subspace $L\in \ZnP$ is self-adjoint, i.e., $L=L^*$.
\end{lemma}
\begin{proof}
  Let $(P,W)\in \M(L)$ be according to Definition \ref{DefZnP}. For every $(u,v)\in L$ there is some $p\in\R^n$ with $(u,v)=(Pp,Wp)$ implying
  \[\skalp{Ws,u}+\skalp{-Ps,v}= \skalp{Ws,Pp}-\skalp{Ps,Wp}=\skalp{s,(WP-PW)p}=0\ \forall s\in\R^n\]
  by symmetry of $P$, $W$ and \eqref{EqPW2}. Hence, $\rge(W,-P)\subset L^\perp$. Further, since $L=\rge(P,W)$ is an $n$-dimensional subspace, so are $L^\perp$ and $\rge(W,-P)$ as well and we conclude $L^\perp=\rge(W,-P)$. Now the equality $L=L^*$ follows from the definition of $L^*$.
\end{proof}
\begin{lemma}
  For every $L\in\ZnP$ there exists exactly one pair $(P,W)$ of symmetric $n\times n$ matrices fulfilling properties (i), (ii) and (iii) of Definition \ref{DefZnP}.
\end{lemma}
\begin{proof}Consider two pairs $(P_i,W_i)$, $i=1,2$, fulfilling (i)--(iii). The matrices $P_1$ and $P_2$ represent the orthogonal projections onto the same subspace and thus $P_1=P_2=:P$. Since $\rge(P,W_1)=\rge(P,W_2)$, for every $z_1$ there is some $z_2$ with $Pz_1=Pz_2$ and $W_1z_1=W_2z_2$. Hence, by using \eqref{EqPW1}, we have
\[(I-P)z_1=(I-P)W_1z_1=(I-P)W_2z_2=(I-P)z_2\]
and together with $Pz_1=Pz_2$ the equality $z_1=z_2$ follows. Hence $W_1=W_2$.
\end{proof}
\begin{lemma}
  $\ZnP$ is closed in $\Z_n$.
\end{lemma}
\begin{proof}
  Consider a sequence $L_k$ in $\ZnP$ and assume that $L_k$ converges to some $L\in \Z_n$. In order to prove the lemma, we have to show $L\in\ZnP$. Let $(P_k,W_k)$ be the symmetric $n\times n$ matrices associated with $L_k$ according to Definition \ref{DefZnP}. By \eqref{EqPW2}, for each $k$ the symmetric matrices $P_k$ and $W_k$ commute and can thus be diagonalized with the same orthogonal matrix $Q_k$, i.e., there are diagonal matrices $D^k={\rm diag}(d^k_1,\ldots, d^k_n)$ and $E^k={\rm diag}(e^k_1,\ldots, e^k_n)$ such that $P_k=Q_kD^kQ_k^T$ and $W_k=Q_kE^kQ_k^T$. Since $P_k$ represents an orthogonal projection, we have $d^k_j\in\{0,1\}$, $j=1,\ldots,n$ and from
  \[W_k(I-P_k)=Q_kE^kQ_k^T( I -Q_kD^kQ_k^T)= Q_k E^k(I-D^k)Q_k^T=I-P_k=Q_k( I-D^k)Q_k^T\]
   we deduce that $E^k(I-D^k)=I-D^k$ implying  $e^k_j(1-d^k_j)=1-d^k_j$, $j=1,\ldots,n$. By passing to a subsequence we may assume  that the orthogonal matrices $Q_k$ converge to some orthogonal matrix $Q$, the diagonal matrices $D^k$ are constant $D^k\equiv D$ for all $k$ and the diagonal matrices $E^k$ converge to some diagonal matrix $E$ with entries $e_j\in\R\cup\{\pm \infty\}$. Let $I:=\{j\in\{1,\ldots,n\}\mv e_j\in\R\}$ and for each $k$ sufficiently large let $S^k$ denote the diagonal matrix with elements
   \[s_j^k:=\begin{cases}1&\mbox{if $j\in I$,}\\\frac 1{e_j^k}&\mbox{if $j\not\in I$.}\end{cases}\]
    Then
    \begin{equation}\label{EqL_k}L_k=\rge(P_k,W_k)=\rge(Q_kD^kQ_k^T,Q_kE^kQ_k^T)=\rge(Q_kD^k,Q_kE^k)=\rge(Q_kD^kS^k,Q_kE^kS^k)\end{equation}
    and the diagonal matrices $D^kS^k$ and $E^kS^k$ converge to diagonal matrices $\tilde D={\rm diag}(\tilde d_1,\ldots,\tilde d_n\}$ and $\tilde E={\rm diag}(\tilde e_1,\ldots,\tilde e_n\}$, respectively, satisfying
    \[\tilde d_j=\begin{cases}
      d_j\in\{0,1\}&\mbox{if $j\in I$,}\\0&\mbox{if $j\not\in I$},
    \end{cases}\qquad
    \tilde e_j=\begin{cases}e_j&\mbox{if $j\in I$,}\\1&\mbox{if $j\not\in I$.}
    \end{cases}\]
    Now consider the symmetric matrices $P:=Q\tilde D Q^T$ and $W:=Q\tilde E Q^T$. Since $\tilde d_j\in\{0,1\}$, $j=1,\ldots,n$, it follows that $P^2=Q\tilde D^2Q^T=Q\tilde D Q^T=P$.
    For every $j\in I$ we have $\tilde e_j(1-\tilde d_j)=e_j(1-d_j)=\lim_{k\to\infty}e_j^k(1-d_j^k)=\lim_{k\to\infty}1-d_j^k=1-\tilde d_j$ and for every $j\not\in I$ there holds $\tilde e_j(1-\tilde d_j)=(1-\tilde d_j)=1$. Thus $\tilde E(I-\tilde D)=I-\tilde D$ implying
    \[W(I-P)=Q\tilde E Q^TQ(I-\tilde D)Q^T=Q\tilde E(I-\tilde D)Q^T=Q(I-\tilde D)Q^T=I-P.\]
    By \cite[Lemma 3.1]{GfrOut22a}, we obtain from \eqref{EqL_k} together with $\lim_{k\to \infty}L_k=L$  that
    \[L=\rge(Q\tilde D, Q \tilde E)=\rge(Q\tilde DQ^T,Q\tilde E Q^T)=\rge(P,W).\]
    This shows that $L\in\Z_n^{P,W}$ and the lemma is proved.
\end{proof}
\begin{proposition}\label{PropPW_Basis}For every $(x,x^*)\in{\rm proxreg\,}\varphi$ there holds $\emptyset\not=\Sp\partial\varphi(x,x^*)\subset \ZnP$.
  Further, for every $(x,x^*)\in\gph \varphi$ one has
  \begin{equation}\label{EqSpPW}\Sp_{P,W}\partial \varphi(x,x^*):=\Sp\partial\varphi(x,x^*)\cap \ZnP=\Sp^*\partial\varphi(x,x^*)\cap \ZnP=:\Sp^*_{P,W}\partial \varphi(x,x^*)\not=\emptyset.\end{equation}
\end{proposition}
\begin{proof}
  Consider $(x,x^*)\in{\rm proxreg\,}\varphi$, $L\in\Sp\partial\varphi(x,x^*)\not=\emptyset$, $\lambda>0$ and the symmetric positive semidefinite matrix $B$ according to Proposition \ref{PropProxRegularQ} such that $L=\rge\big(B,\frac 1\lambda(I-B)\big)$. Let $P=BB^\dag$ be the orthogonal projection onto $\rge B$,  where $B^\dag$ denotes the Moore-Penrose inverse of $B$, see, e.g. \cite{Pen55}. Since $BB^\dag=B^\dag B=P$ and $B(I-P)=(I-P)B^\dag=0$, we obtain that
  \[\big(B+\frac1\lambda(I-P)\big)\big(B^\dag+\lambda(I-P)\big)=BB^\dag+(I-P)=I\]
  showing that $B^\dag+\lambda(I-P)$ is nonsingular. Hence
  \[L=\rge\big( B,\frac 1\lambda(I-B)\big)=\rge \big(B(B^\dag+\lambda(I-P)),\frac 1\lambda(I-B)(B^\dag+\lambda(I-P))\big)=\rge\big(P,\frac 1\lambda(I-B)B^\dag+I-P\big)\]
  Since $B^\dag$ is symmetric, the matrix $W:=\frac 1\lambda(I-B)B^\dag+I-P$ is symmetric as well and together with $\big(\frac 1\lambda(I-B)B^\dag+I-P\big)(I-P)=I-P$ we conclude that $L\in\ZnP$. Thus $\Sp\partial\varphi(x,x^*)\subset \ZnP$. In order to show that $\Sp_{P,W}\partial\varphi(x,x^*)\not=\emptyset$ for every $(x,x^*)\in\gph\partial\varphi$, consider $(x,x^*)\in\gph\partial\varphi$ together with  sequences $(x_k,x_k^*)$ in ${\rm proxreg\,}\varphi$ converging to $(x,x^*)$ and $L_k\in\Sp\partial\varphi(x_k,x_k^*)\subset\Z_n^{P,W}$. Since $\ZnP$ is compact as a closed subset of the compact metric space $\Z_n$, the sequence $L_k$ has a limit point $\bar L\in \ZnP$ and, by \cite[Lemma 3.13]{GfrOut22a}, $\bar L$ also belongs to $\Sp\partial \varphi(x,x^*)$ proving $\Sp_{P,W}\partial \varphi(x,x^*)\not=\emptyset$. The equality $\Sp_{P,W}\partial \varphi(x,x^*)=\Sp_{P,W}^*\partial \varphi(x,x^*)$ is an immediate consequence of the definition of $\Sp^*\partial\varphi$ and Lemma \ref{LemSelfadj}.
\end{proof}
\begin{remark}
  In many cases we can prove $\Sp_{P,W}\partial \varphi=\Sp\partial\varphi$, e.g., for prox-regular and subdifferentially continuous functions. In fact, we do not know any proper lsc function $\varphi$ fulfilling Assumption \ref{AssProx} such that $\Sp_{P,W}\partial \varphi(x,x^\ast)\not=\Sp\partial\varphi(x,x^*)$ for some $(x,x^*)\in\gph\partial\varphi$. For the purpose of applying the semismooth$^{*}$ Newton method this issue is of no relevance.
\end{remark}
If $\varphi$ is twice Fr\'echet differentiable at $x$, then
\[T_{\gph \nabla\varphi}(x,\nabla \varphi(x))=\rge(I,\nabla^2\varphi(x))\]
implying  $\rge(I,\nabla^2 \varphi(x))\in \Sp_{P,W} \nabla\varphi(x,\nabla \varphi(x))$ with $P=I$ and $W=\nabla^2 \varphi(x)$.

\begin{example}
  Consider the $\ell_1$ norm $\norm{x}_1=\sum_{i=1}^n\vert x_i\vert$ on $\R^n$. Then
  \[\partial \norm{x}_1=\partial \vert x_1\vert\times\partial \vert x_2\vert\times\ldots\times\partial \vert x_n\vert\]
  and
  \[\Sp\partial \norm{\cdot}_1(x,x^*)=\Big\{\big\{(u,u^*)\in\R^n\times\R^n\mv (u_i,u_i^*)\in L_i,\ i=1,\ldots,n\big\}\Bmv L_i\in\Sp \partial\vert\cdot\vert(x_i,x_i^*),\ i=1,\ldots,n\Big\}\]
  by \cite[Lemma 3.8]{GfrOutVal22b}. Since $\gph \partial \vert\cdot\vert=\big(\R_-\times \{-1\}\big)\cup\big(\{0\}\times[-1,1]\big)\cup\big(\R_+\times\{+1\}\big)$, we obtain
  \[\OO_{\partial\vert\cdot\vert}=\big(\R_{--}\times \{-1\}\big)\cup\big(\{0\}\times(-1,1)\big)\cup\big(\R_{++}\times\{1\}\big)\]
  and
    \[\Sp \partial\vert\cdot\vert(t,t^*)=\begin{cases}\{\R\times \{0\}\}&\mbox{if $(t,t^*)\in \big(\R_{--}\times \{-1\}\big)\cup \big(\R_{++}\times\{+1\}\big)$,}\\
    \{\{0\}\times\R\}&\mbox{if $(t,t^*)\in\{0\}\times(-1,1)$.}\end{cases}\]
  For the remaining two points  we must take limits of subspaces, which is an easy task in this case, and we obtain
  \[ \Sp \partial\vert\cdot\vert(0, \pm 1)=\{\R\times \{0\}, \{0\}\times\R\}.\]
  It follows that every $L\in \Sp \partial\norm{\cdot}_1(x,x^*)$ has the representation $L=\rge(P,I-P)$, where $P$ is a diagonal matrix with diagonal entries belonging to $\{0,1\}$.
\end{example}

\begin{example}\label{ExEll_q}
  Consider the  function $\norm{x}_q^q: =\sum_{i=1}^n\vert x_i\vert^q$ for $q\in(0,1)$. Then
  \[\partial \norm{x}_q^q=\partial \vert \cdot\vert^q(x_1)\times\partial \vert \cdot\vert^q(x_2)\times\ldots\times\partial \vert \cdot\vert^q(x_n)\]
  and
  \[\partial \vert \cdot\vert^q(t)=\begin{cases}q\vert t\vert ^{q-1}{\rm sign\,}(t)&\mbox{if $t\not=0$,}\\ \R&\mbox{if $t=0$.}\end{cases}\]
  Hence
  \[\Sp \partial\vert\cdot\vert^q(t,t^*)=\begin{cases}\{\rge(1,q(q-1)\vert t\vert^{q-2})\}&\mbox{if $t\not=0$,}\\
  \{\rge(0,1)\}&\mbox{if $t=0$.}\end{cases}\]
   Again, by \cite[Lemma 3.8]{GfrOutVal22b} we have
  \[\Sp\partial \norm{\cdot}_q^q(x,x^*)=\Big\{\big\{(u,u^*)\in\R^n\times\R^n\mv (u_i,u_i^*)\in L_i,\ i=1,\ldots,n\big\}\Bmv L_i\in\Sp \partial\vert\cdot\vert^q(x_i,x_i^*),\ i=1,\ldots,n\Big\}\]
  and thus $\Sp\partial \norm{\cdot}_q^q(x,x^*)$ consists of the single subspace $\rge(P,W)$ with diagonal matrices $P,W$ satisfying
  \[P_{ii}=\begin{cases}1&\mbox{if $x_i\not=0$,}\\
  0&\mbox{if $x_i=0$,}\end{cases}\quad W_{ii}=\begin{cases}q(q-1)\vert x_i\vert^{q-2}&\mbox{if $x_i\not=0$,}\\
  1&\mbox{if $x_i=0$.}\end{cases}\]
\end{example}

\begin{example}
  Consider the $\ell_0$ pseudo-norm $\norm{x}_0=\sum_{i=1}^n \xi(x_i)$ on $\R^n$, where
  \[\xi(t):=\begin{cases}1&\mbox{if $t\not=0$},\\0&\mbox{if $t=0$.}\end{cases}\]
  By \cite[Proposition 10.5]{RoWe98} we have
  \[\partial \norm{x}_0=\partial\xi(x_1)\times\partial\xi(x_2)\times\ldots\times\partial\xi(x_n).\]
  Since $\gph \partial\xi =\big(\R\times\{0\}\big)\cup \big(\{0\}\times\R\big)$, we obtain that
  \[\OO_{\partial\xi}=\big((\R\setminus\{0\})\times\{0\}\big)\cup \big(\{0\}\times(\R\setminus\{0\})\big)\]
  and
  \[\Sp\partial\xi(t,t^*)=\begin{cases}\big\{\R\times\{0\}\big\}&\mbox{if $t\in\R\setminus\{0\}$, $t^*=0$,}\\
  \big\{\{0\}\times\R\big\}&\mbox{if $t=0$, $t^*\in\R\setminus\{0\}$.}\end{cases}\]
  To obtain the SC derivative at $(0,0)$, we must again perform a limiting process, which results in
  \[\Sp\partial\xi(0,0)=\big\{\R\times\{0\},\{0\}\times\R\}.\]
  Since $\gph \partial\xi$ is geometrically derivable and, whenever $T_{\gph\partial \xi}(t,t^*)$ is a subspace, the dimension of this subspace is $1$, we conclude from \cite[Proposition 3.5]{GfrManOutVal22} that
  \[\Sp\partial \norm{\cdot}_0(x,x^*)=\Big\{\big\{(u,u^*)\in\R^n\times\R^n\mv (u_i,u_i^*)\in L_i,\ i=1,\ldots,n\big\}\Bmv L_i\in\Sp \partial\xi(x_i,x_i^*),\ i=1,\ldots,n\Big\}.\]
  Note that $\xi$ is not subdifferentially continuous at $0$ for $0$ and therefore $\norm{\cdot}_{0}$ is not subdifferentially continuous at $x$ for $x^*$ whenever $x_i=x_i^*=0$ for some $i$. Nevertheless, for all subspaces $L\in \Sp\partial \norm{\cdot\nobreak}_0(x,x^*)$ we have $L=L^*=\rge(P,I-P)\in\Z_n^{P,W}$, where $P$ is a diagonal matrix with diagonal entries belonging to $\{0,1\}$, i.e., the subspaces have exactly the same form as in case of the $\ell_1$-norm.
\end{example}
\if{
\begin{example}
  Consider now the indicator function $\delta_S$ of the set $S:=\{x\in\R^n\mv \norm{x}_0\leq m\}$ for some integer $m\in\{0,\ldots,n-1\}$. Then
  \[\partial{\delta_S }(x)=\begin{cases}
    \{0\}&\mbox{if $\norm{x}_0<m$,}\\
    \{x^*\in\R^n\mv x_ix_i^*=0, i=1,\ldots,n\}&\mbox{if $\norm{x}_0=m$}
  \end{cases}\]
  and therefore
  \[\OO_{\partial \delta_S}=\{(x,x^*)\in\gph \partial\delta_S\mv  \norm{x}_0=m\}.\]
  For any index set $I\subset \{1,\ldots,n\}$ we define the subspace
  \[L_I:=\{(u,u^*)\in\R^n\times\R^n\mv u_i=0,\ i\not\in I, u_i^*=0, i\in I\}.\]
  Then it follows that for every $(x,x^*)\in\OO_{\partial S}$ one has
  \[\Sp \partial \delta_S(x,x^*)=\{L_{I(x)}\}\]
  where $I(x):=\{i\mv x_i\not=0\}$ denotes the index set of nonzero components of $x$. Now consider $x\in\R^n$ with $\norm{x}_0<m$. Taking any index set $I\subset\{1,\ldots,n\}$ with cardinality $m$ such that $I(x)\subset I$, we can find a sequence $x_k\to x$ with $I(x_k)=I$ for all $k$. Hence
  \[\Sp \partial \delta_S(x,0)=\big\{L_I\mv I(x)\subset I\subset\{1,\ldots,n\},\ \vert I\vert=m\big\}.\]
\end{example}
}\fi
\begin{example}\label{ExEuclNorm}
  In case of the Euclidean norm $\norm{\cdot}$ on $\R^n$ we have
  \[\partial \norm{x}=\begin{cases}
    \B&\mbox{if $x=0$,}\\\frac x{\norm{x}}&\mbox{if $x\not=0$}
  \end{cases}\]
  and thus
  \[\OO_{\partial\norm{\cdot}}=\big(\{0\}\times \inn \B\big)\cup \left\{\left(x,\frac x{\norm{x}}\right)\Bmv x\not=0\right\}\]
  with
  \[\Sp\partial\norm{\cdot}(x,x^*)=\begin{cases}\big\{\{0\}\times\R^n\big\}&\mbox{if $x=0$, $x^*\in \inn\B$,}\\
  \big\{\rge(I,\nabla^2\norm{x})\big\}=\Big\{\rge\Big(I,\frac 1{\norm{x}}\Big(I-\frac{xx^T}{\norm{x}^2}\Big)\Big)\Big\}&\mbox{if $x\not=0$, $x^*=\frac x{\norm{x}}$.}
  \end{cases}\]
   Now consider $x^*\in\bd\B$ and a sequence $(x_k,x_k^*)\longsetto{\OO_{\partial\norm{\cdot}}}(0,x^*)$ such that $L_k:=T_{\gph \partial\norm{\cdot}}(x_k,x_k^*)$ converges in $\Z_n$ to some subspace $L$. If $(x_k,x_k^*)\in\{0\}\times\inn \B$ for infinitely many $k$, we readily obtain $L=\{0\}\times\R^n$. Now assume that there are only finitely many $k$ with $(x_k,x_k^*)\in\{0\}\times\inn \B$ and we conclude that $x_k^*=x_k/\norm{x_k}$ and $L_k=\rge(I,\nabla^2 \norm{x_k})$ for all $k$ sufficiently large. The Hessians $\nabla^2\norm{x_k}$ are unbounded, however, if we set $B_k:=(I+\nabla^2\norm{x_k})^{-1}$ we have $L_k=\rge(B_k,\nabla^2\norm{x_k}B_k)=\rge(B_k, I-B_k)$ and
  \[B_k=\frac{\norm{x_k}}{\norm{x_k}+1}\Big(I+\frac{x_kx_k^T}{\norm{x_k}^3}\Big),\]
  cf. \cite[Example 3.7]{GfrOutVal22b}. Then $\lim_{k\to\infty}B_k=\lim_{k\to\infty}\frac{x_kx_k^T}{\norm{x_k}^2}=x^*{x^*}^T$ and $L=\rge(x^*{x^*}^T, I-x^*{x^*}^T)$ follows. Hence
  \[\Sp\norm{\cdot}(0,x^*)=\big\{\{0\}\times\R^n, \rge(x^*{x^*}^T, I-x^*{x^*}^T)\big\},\ x^*\in\bd\B.\]
\end{example}
In what follows we want to work with the bases representing the subspaces contained in $\Sp_{P,W}\partial\varphi(x,x^*)$.
\begin{definition}
  Given $(x,x^*)\in\gph\partial\varphi$, we denote by $\M_{P,W}\partial \varphi(x,x^*)$ the collection of all symmetric $n\times n$ matrices $P$ and $W$ fulfilling conditions (ii) and (iii) of Definition \ref{DefZnP} such that $\rge(P,W)\in \Sp_{P,W}\partial\varphi(x,x^*)$.
\end{definition}
Consider now the calculation of the SCD \ssstar Newton direction for the inclusion $0\in\partial\varphi(x)$ at the iterate $\ee xk$. We denote the outcome of the approximation step by $(\ee zk,\ee{z^*}k)$. Given $(\ee Pk,\ee Wk)\in\M_{P,W}\partial\varphi(\ee zk,\ee{z^*}k)$, the linear system \eqref{EqNewtonStep} defining the Newton direction reads as
\begin{equation}\label{EqSCDNewtonDir}\ee Wk\Delta x=-\ee Pk\ee{z^*}k,\end{equation}
where we have taken into account the symmetry of the matrices $\ee Wk$ and $\ee Pk$.

\begin{lemma}\label{LemStatPoint}
  Let $(z,z^*)\in\gph \partial \varphi$, $(P,W)\in \M_{P,W}\partial\varphi(z,z^*)$ and $\bar s\in\R^n$ be given. Then the following statements are equivalent:
  \begin{enumerate}
    \item[(i)] $W\bar s=-Pz^*$.
    \item[(ii)] $\bar s$ is a stationary point for the quadratic program
    \begin{equation}\label{EqQP}\min_s \frac 12s^TWs+\skalp{z^*,s}\quad\mbox{ subject to }\quad s\in\rge P,\end{equation}
    i.e.,  $0\in \partial q_{P,W}(\bar s),$ where
    \[q_{P,W}(s):=\frac 12s^TWs+\skalp{z^*,s}+\delta_{\rge P}(s)=\begin{cases}\frac 12s^TWs+\skalp{z^*,s}&\mbox{if $s\in\rge P$,}\\
    \infty&\mbox{else.}\end{cases}\]
  \end{enumerate}
\end{lemma}
\begin{proof}Note that
\[\partial q_{P,W}(s)=\begin{cases} Ws+z^*+N_{\rge P}(s)=Ws+z^*+(\rge P)^\perp&\mbox{if $s\in\rge P$,}\\\emptyset&\mbox{else.}\end{cases}\]
  ``(i)$\Rightarrow$(ii)'': Multiplying $W\bar s=-Pz^*$ with $I-P$ and taking into account \eqref{EqPW1} yields $(I-P)\bar s=(I-P)W\bar s=-(I-P)Pz^*=0$ showing $\bar s\in\rge P$. Further we have $W\bar s +z^*= (I-P)z^*\in\rge (I-P)=(\rge P)^\perp=-(\rge P)^\perp$.\\
  ``(ii)$\Rightarrow$(i)'': Since $\bar s \in\rge P$ we have $P\bar s=\bar s$. By multiplying  the inclusion $0\in W\bar s+z^*+(\rge P)^\perp$ with $P$ and taking into account \eqref{EqPW2} we obtain $0\in PW\bar s +Pz^*+ P(\rge P)^\perp=WP\bar s +Pz^*+\{0\}=W\bar s + Pz^*+\{0\}$.
  \end{proof}
  Clearly, a stationary solution $\bar s$ for the quadratic program \eqref{EqQP} is a local minimizer if and only if the matrix $PWP$ is positive semidefinite and in this case $\bar s$ is even a global minimizer. Further, a local minimizer $\bar s$ is unique if and only if
  \begin{equation}
    \label{PosDefW}s^TPWPs\geq \beta\norm{Ps}^2\ \forall s
  \end{equation}
  for some $\beta>0$.
  \begin{lemma}\label{LemWposdef}
    Consider $(P,W)\in \M_{P,W}\partial q(x,x^*)$ and assume that $W$ is positive definite. Then $L:=\rge(P,W)\in \Z_n^{\rm reg}$ and
  \begin{equation}\label{EqC_LStrongConv}\norm{C_L}=\frac1{\min\{s^TPWPs\mv \norm{Ps}=1\}}.\end{equation}
  \end{lemma}
  \begin{proof}
 Since $W$ is positive definite, $L\in \Z_n^{\rm reg}$ by Proposition \ref{PropC_L}.
  In order to prove equation \eqref{EqC_LStrongConv}, consider an $n\times m$ matrix $Z$, whose columns form an orthonormal basis of $\rge P$, where $m:=\dim(\rge P)$. Then $P=ZZ^T$ and every $p\in\R^m$ can be written as $p=Z^Ts$ with some $s\in\R^n$. By \eqref{EqPW3} we have $W=ZZ^TWZZ^T+(I-ZZ^T)$ implying $Z^TW=Z^TWZZ^T$ and $(Z^TWZ)^{-1}Z^T=Z^TW^{-1}$.
  Thus
  \[C_L=PW^{-1}=ZZ^TW^{-1}=Z(Z^TWZ)^{-1}Z^T\]
  and, by taking into account that $\norm{s}^2\geq \norm{Z^Ts}^2$ $\forall s\in\R^n$, where equality holds if and only if $s=ZZ^Ts$, we obtain that
  \begin{align*}\norm{C_L}&=\max\{s^TZ(Z^TWZ)^{-1}Z^Ts\mv \norm{s}=1\}=\max\{p^T(Z^TWZ)^{-1}p\mv \norm{p}=1\}\\
  &=\frac1{\min \{p^TZ^TWZp\mv \norm{p}=1\}}=\frac1{\min\{s^TPWPs\mv \norm{Ps}=1\}}.\end{align*}
  \end{proof}
  \begin{proposition}\label{PropTiltStab}Assume that the lsc function $\varphi:\R^n\to\oR$ is prox-regular and subdifferentially continuous at $\xb\in \R^n$ for $\xba=0$. Then the following statements are equivalent.
  \begin{enumerate}
  \item[(i)] $\xb$ is a tilt-stable local minimizer for $\varphi$
  \item[(ii)] For every pair $(P,W)\in \M_{P,W}\partial \varphi(\xb,0)$ the matrix $W$ is positive definite.
  \item[(iii)] There is a neighborhood $U\times V$ of $(\xb,0)$ such that for every $(x,x^*)\in\gph \partial\varphi\cap(U\times V)$ and every  $(P,W)\in \M_{P,W}\partial \varphi(x,x^*)$ the matrix $W$ is positive definite.
  \end{enumerate}
  Moreover, these statements entail that
  \begin{equation}\label{EqSCDReg_TiltStable}\scdreg \partial \varphi(\xb,0)=\sup\left\{\frac1{\min\{s^TPWPs\mv \norm{Ps}=1\}} \bmv (P,W)\in \M_{P,W}\partial \varphi(\xb,0)\right\}.\end{equation}
  \end{proposition}
  \begin{proof}
  ad (i)$\ \Leftrightarrow\ $(ii): By Theorem \ref{ThTiltStab}, Lemma \ref{PropPW_Basis} and \eqref{EqPW2}, $\xb$ is a tilt-stable minimizer for $\varphi$, if and only if for every $(P,W)\in\M_{P,W}\partial\varphi(\xb,0)$ the matrix $W$ is nonsingular and the matrix $PW=PWP$ is positive semidefinite. Since $z^TWz=z^TPWPz+z^T(I-P)z\geq z^TPWPz$, positive semidefiniteness of $PWP$ together with nonsingularity of $W$ is equivalent with  definiteness of $W$.

  ad (ii)$\ \Rightarrow\ $(iii): By contraposition. Assume that (iii) does not hold and there exist sequences $(x_k,x_k^*)\longsetto{\gph\partial \varphi}(\xb,0)$ and $(P_k,W_k)\in\M_{P,W}\partial\varphi(x_k,x_k^*)$ such that for every $k$ the matrix $W_k$ is not positive definite, i.e., the smallest eigenvalue $\sigma_k$ of the matrix $W_k$ is nonpositive. Since $W_k=P_kW_kP_k+(I-P_k)$, we may also conclude that $\sigma_k$ is the smallest eigenvalue of $P_kW_kP_k$. By Proposition \ref{PropProxRegularQ} there is some $\lambda>0$ such that for every $k$ sufficiently large there is some symmetric positive semidefinite matrix $B_k$ with $\rge(P_k,W_k)=\rge(B_k,\frac 1\lambda(I-B_k))$. From the proof of Proposition \ref{PropPW_Basis} we extract the relations $P_k=B_kB_k^\dag$ and $W_k=\frac 1\lambda(I-B_k)B_k^\dag +(I-P_k) =\frac 1\lambda B_k^\dag -\frac 1\lambda P_k+(I-P_k)$. Denoting by $p_k$, $\norm{p_k}=1$ a normalized eigenvector of $P_kW_kP_k$ to the eigenvalue $\sigma_k$, we obtain $p_k=P_kp_k$ and, together with the fact that the Moore-Penrose inverse of a symmetric positive semidefinite matrix is again positive semidefinite, that
  \[0\geq\sigma_k=p_k^TP_kW_kP_kp_k=\frac 1\lambda p_k^TP_kB_k^\dag P_kp_k-\frac 1\lambda\norm{P_kp_k}^2+\norm{(I-P_k)p_k}^2\geq-\frac1\lambda.\]
  Hence the sequence $W_kp_k=\sigma_kp_k$ is bounded and, since $\ZnP$ is a closed subspace of the compact metric space $\Z_n$, by possibly passing to a subsequence we may assume that $L_k:=\rge(P_k,W_k)$ converges in $Z_n$ to some subspace $L=\rge(P,W)\in\M_{P,W}\partial\varphi(\xb,0)$, that $p_k$ converges to some $\bar p\not=0$ and that $W_kp_k$ converges to some $\bar w$. Obviously, $\bar p^T\bar w=\lim_{k\to\infty}p_k^TW_kp_k\leq 0$ and from $(p_k,W_kp_k)\in L_k$ we can conclude that $(\bar p,\bar w)\in L$, i.e., there is some $q$ with $\bar p=Pq\not=0$ and $\bar w=Wq$ showing that $\bar p^T\bar w=q^TPWq=q^TPWPq\leq 0$. Hence $W$ is not positive definite and we have verified that (ii) does not hold.

ad (iii)$\ \Rightarrow\ $(ii): This obviously holds.

Finally, equation \eqref{EqSCDReg_TiltStable} follows from Lemmata \ref{Lem_scdreg} and \ref{LemWposdef}.
\end{proof}
\if{
\begin{lemma}\label{LemStrongConv}
   Assume that the lsc function $q:\R^n\to\oR$ is strongly convex with constant $\beta>0$. Then for every $(x,x^*)\in\gph\partial q$ and every $(P,W)\in \M_{P,W}\partial q(x,x^*)$ the inequality \eqref{PosDefW} is fulfilled and $W$ is positive definite.
   Thus $\partial q$ is SCD regular with $\scdreg \partial q(x,x^*)\leq 1/\beta$.
\end{lemma}
\begin{proof}
    Consider $(x,x^*)\in\gph\partial q$ and $(P,W)\in \M_{P,W}\partial q(x,x^*)$.
  By \cite[Exercise 12.59]{RoWe98}, the function $\tilde q:=q-\frac12\beta\norm{\cdot}^2$ is convex and $\partial \tilde q(x) = \partial q(x)-\beta x$. Hence, by invoking \cite[Proposition 3.15]{GfrOut22a} we obtain $\tilde L:=\rge(P,W-\beta P)\in\Sp^*\partial\tilde q(x,x^*-\beta x)$. Since $(W-\beta P)(I-P)=I-P$, we conclude that $(P,W-\beta P)\in\M_{P,W}\partial q(x,x^*-\beta x)$. On the other hand, from \cite[Corollary 3.28]{GfrOut22a} we deduce that  there is a symmetric positive semidefinite matrix $B$ with $\norm{B}\leq 1$ such that $\tilde L=\rge (B,I-B)$. Utilizing the  arguments already employed in the proof of Proposition \ref{PropPW_Basis}, we obtain that $P=BB^\dag$ and $W-\beta P=(I-B)B^\dag +I-P$. Taking into account $\norm{B}\leq 1$ we obtain that $I-B$ is positive semidefinite and consequently $(I-B)B^\dag$ and $(I-B)B^\dag+(I-P)$ are positive semidefinite as well. Thus, for every $s\in\R^n$ we have $0\leq (Ps)^T(W-\beta P)Ps=s^TPWPs-\beta\norm{Ps}^2$ showing the first assertion. By taking into account \eqref{EqPW3}, for every $s\in\R^n$ we have
  \[s^TWs=s^TPWPs+s^T(I-P)s\geq \beta \norm{Ps}^2+\norm{(I-P)s}^2\geq\min\{\beta,1\}\norm{s}^2\]
  showing that $W$ is positive definite.
 The last statement on SCD regularity of $\partial q$ is now a consequence of Lemma \ref{Lem_scdreg} and \eqref{EqC_LStrongConv}.
  \end{proof}
}\fi
  At the end of this section we provide some sum rule.
\begin{lemma}\label{LemSumRul}Assume that $\varphi=f+g$, where $f$ is twice continuously differentiable at $\xb$ and $g:\R^n\to\oR$ is a proper lsc function. Given $\xba\in\partial g(\xb)$, for every $(P,W)\in\M_{P,W}\partial g(\xb,\xba)$ one has
\[(P,P\nabla^2 f(\xb)P +W)\in \M_{P,W}\partial \varphi(\xb,\nabla f(\xb)+\xba).\]
\end{lemma}
\begin{proof}
  By \cite[Proposition 3.15]{GfrOut22a} we have $L=\rge(P,\nabla^2f(\xb)P+W)\in\Sp\partial\varphi(\xb,\nabla f(\xb)+\xba)$. Let $(Pp, (\nabla^2f(\xb)P+W)p)\in L$ and set $\tilde p:=p+(I-P)\nabla^2f(\xb)Pp$. Then $P\tilde p=Pp$ and
  \begin{align*}(P\nabla^2f(\xb)P+W)\tilde p&= P\nabla^2f(\xb)Pp +Wp +W(I-P)\nabla^2f(\xb)Pp\\
  &=P\nabla^2f(\xb)Pp +Wp +(I-P)\nabla^2f(\xb)Pp=(\nabla^2f(\xb)P+W)p,\end{align*}
  verifying $L\subset \rge(P,P\nabla^2 f(\xb)P +W)$. On the other hand, consider
  \[(Pp,(P\nabla^2 f(\xb)P +W)p)\in\rge(P,P\nabla^2 f(\xb)P +W)\]
   and set $\tilde p:=p-(I-P)\nabla^2f(\xb)Pp$. Then $P\tilde p=Pp$ and
  \begin{align*}(\nabla ^2f(\xb)P+W)\tilde p&=\nabla^2f(\xb)Pp+Wp-W(I-P)\nabla^2f(\xb)Pp=\nabla^2f(\xb)Pp+Wp-(I-P)\nabla^2f(\xb)Pp\\
  &=(P\nabla^2f(\xb)P+W)p\end{align*}
  showing the  reverse inclusion $\rge(P,P\nabla^2 f(\xb)P +W)\subset L$. Hence equality holds and it is easy to see that $P\nabla^2 f(\xb)P +W$ is symmetric and satisfies $(P\nabla^2 f(\xb)P +W)(I-P)=W(I-P)=I-P$. Thus $(P,P\nabla^2 f(\xb)P +W)\in \M_{P,W}\partial \varphi(\xb,\nabla f(\xb)+\xba)$.
\end{proof}

\section{Globalizing the  SCD semismooth$^{*}$ Newton method}

The suggested globalization of the \SCD \ssstar Newton method is done by combining it with some globally convergent method and to this end we use  {\em forward-backward splitting} (also called {\em proximal gradient method}). One iteration of  forward-backward splitting can be described by
\[\ee x{k+1}\in T_{\ee \lambda k}(\ee xk),\]
where for convergence of the method it is important that the parameter $\ee\lambda k$ is  chosen adequately. If $\nabla f$ is globally Lipschitz with constant $L$,  a constant choice $\ee\lambda k\in(0,\frac 2L)$ ensures convergence. In case of unknown Lipschitz modulus $L$ some backtracking procedure for computing $\ee\lambda k$ is available, see, e.g., \cite[Section 10.3.3]{Be17}.

Very recently, De Marchi and Themelis \cite{MaTh22} presented the algorithm  PANOC${}^+$ based on forward-backward splitting, where global convergence is established under the weaker assumption of strict continuity of the gradient $\nabla f$. The following algorithm BasGSSN (Basic Globalized SemiSmooth$^{*}$ Newton method) is motivated by the algorithm PANOC${}^+$.

\begin{algorithm}[Algorithm BasGSSN]\label{AlgBasGSSN}\ \\
  {\bf Input: } Starting point $\ee x0\in \R^n$, $0<\lambda_0\leq\bar\lambda <\lambda_g$, parameters $\alpha, \beta\in(0,1)$, $\bar\rho>0$ and $\sigma\in(0,\frac 12]$.\\
  1. Set $k \leftarrow 0$, $\ee\lambda0\leftarrow\lambda_0$, compute $\ee z0\in T_{\ee \lambda 0}(\ee x0)$ and set $\ee \eta0\leftarrow\frac1{2\ee\lambda 0}\norm{\ee z0-\ee x0}^2$;\\
  2. \begin{minipage}[t]{\myAlgBox}{\tt while} $f(\ee z 0) > \ell_f(\ee x0,\ee z 0)+\alpha \ee \eta 0$\\
  \mbox{\qquad}$\ee\lambda 0\leftarrow\ee\lambda0/2$, compute $\ee z0\in T_{\ee \lambda 0}(\ee x0)$, $\ee \eta0\leftarrow\frac1{2\ee\lambda 0}\norm{\ee z0-\ee x0}^2$;\\
  {\tt end while}
  \end{minipage}\\[1ex]
  3. {\tt if} $\ee xk==\ee zk$ or another convergence criterion is fulfilled {\tt then stop};\\
  4. Compute  a search direction $\ee sk \in\R^n$ with $\norm{\ee sk}\leq \bar\rho$;\\
  5. \begin{minipage}[t]{\myAlgBox}Set $\ee\tau k\leftarrow 1$, $\ee x{k+1}\leftarrow\ee z k+\ee \tau k\ee sk$, $\ee \lambda{k+1}\leftarrow\ee\lambda k$,
   compute $\ee z{k+1}\in T_{\ee\lambda{k+1}}(\ee x{k+1})$ and set $\ee\eta{k+1}\leftarrow\frac 1{2\ee\lambda{k+1}}\norm{\ee z{k+1}-\ee x{k+1}}^2$;
  \end{minipage}\\[0.5ex]
  6. \begin{minipage}[t]{\myAlgBox}{\tt while} $\phiFB{k+1}>\phiFB{k}-\beta(1-\alpha)\ee \eta k\ \vee\ f(\ee z{k+1})> \ell_f(\ee x{k+1},\ee z{k+1})+\alpha \ee \eta{k+1}$\\
  \mbox{\qquad}{\tt if} $\phiFB{k+1}>\phiFB{k}-\beta(1-\alpha)\ee \eta k$ {\tt then}\\
  \mbox{\qquad\qquad}$\ee\tau k\leftarrow\ee\tau k/2$, $\ee x{k+1}\leftarrow\ee z k+\ee \tau k\ee sk$;\\
  \mbox{\qquad}{\tt else}\\
  \mbox{\qquad\qquad}$\ee\lambda{k+1}\leftarrow\ee\lambda{k+1}/2$;\\
  \mbox{\qquad}compute $\ee z{k+1}\in T_{\ee\lambda{k+1}}(\ee x{k+1})$, $\ee\eta{k+1}\leftarrow\frac 1{2\ee\lambda{k+1}}\norm{\ee z{k+1}-\ee x{k+1}}^2$;\\
  {\tt end while}
  \end{minipage}\\[1ex]
  7. \begin{minipage}[t]{\myAlgBox}{\tt while} $f(\ee z{k+1})\leq \ell_f(\ee x{k+1},\ee z{k+1})+ \sigma\alpha \ee \eta{k+1}\  \wedge\  2\ee\lambda{k+1}\leq\bar\lambda$ {\tt do}\\
    \mbox{\qquad} compute $\tilde z\in T_{2\ee\lambda{k+1}}(\ee x{k+1})$ and $\tilde \eta\leftarrow\frac1{4\ee\lambda{k+1}}\norm{\tilde z-\ee x{k+1}}^2$.\\
    \mbox{\qquad}{\tt if} $f(\tilde z)> \ell_f(\ee x{k+1},\tilde z) + \alpha\tilde\eta$ {\tt then }\\
    \mbox{\qquad\qquad} proceed with step 8.;\\
    \mbox{\qquad}{\tt else}\\
    \mbox{\qquad\qquad} set $\ee\lambda{k+1}\leftarrow 2\ee\lambda{k+1}$, $\ee z{k+1}\leftarrow\tilde z$, $\ee \eta{k+1}\leftarrow\tilde \eta$;\\
    {\tt end while}\end{minipage}\\[1ex]
 8. Set $k\leftarrow k+1$ and go to step 3.;
\end{algorithm}

The search direction $\ee sk$ will be an (inexact) SCD \ssstar Newton direction and we will discuss later how to compute it in practice.

\begin{remark}Apart from using a different notation there are the following significant differences to PANOC${}^+$ from \cite{MaTh22}:
\begin{enumerate}
  \item In PANOC${}^+$ the search direction is of the form $\ee sk = \ee dk + \ee xk-\ee zk$ with $\norm{\ee dk}\leq D\norm{\ee zk-\ee xk}$. We do not use such a restriction on our search direction $\ee sk$ because we do not know how to choose the constant $D$ in practice without destroying the superlinear convergence properties of the SCD \ssstar Newton method. Further, in our numerical experiments for the nonconvex regression problem  (Subsection \ref{SubSecRegr}) we observed that our method can escape from non-optimal stationary points and in this case the ratio $\norm{\ee sk}/\norm{\ee zk-\ee xk}$ can be arbitrarily large. We refer to Subsection \ref{SubSecRegr} for more details on this phenomenon.
  \item If in the {\tt while}-loop in step 6. both conditions are fulfilled, we decrease $\ee\tau k$ whereas in PANOC${}^+$ the parameter $\ee\lambda{k+1}$ is reduced.  This exchange does not affect the convergence analysis and is mainly motivated by performance reasons. We observed that for poor search direction $\ee sk$, which sometimes occur in the first iterations of the algorithm, the trial points $\ee x{k+1}$ can be far away from a solution and therefore it seems reasonable to favor the reduction of the step size $\ee\tau k$. Therefore we possibly  decrease $\ee \lambda{k+1}$ only at those points which are candidates for the next iterate, i.e., a sufficient reduction in $\phiFB{k+1}$ has been achieved.
  \item We also add  the possibility to increase $\lambda$ in step 7 in contrary to PANOC$^+$, where the produced sequence $\ee \lambda k$ is monotonically  decreasing. Again, this modification is mainly motivated  by performance reasons since  it is well-known that the method of proximal gradients works better with larger values of $\lambda$.
\end{enumerate}
\end{remark}

We now summarize properties of the  iterates $\ee xk$, $\ee zk$ and $\ee \lambda k$ which follow from the termination criteria of the {\tt while} loops in steps 2.,6. and 7.
\begin{lemma}\label{LemOutIterate}Consider the iterates generated by Algorithm \ref{AlgBasGSSN}. The following hold.
 \begin{enumerate}
  \item[(i)]For each $k\geq 0$ there holds $\ee zk\in T_{\ee \lambda k}(\ee xk)$ and
  \begin{align}\label{EqLambda}&f(\ee zk)\leq \ell_f(\ee xk,\ee zk)+\frac \alpha{2\ee\lambda k}\norm{\ee zk-\ee xk}^2,\\
   \label{EqDecrFB1}&\phiFB{k+1}\leq \phiFB{k}-\frac{\beta(1-\alpha)}{2\ee \lambda k}\norm{\ee zk-\ee x k}^2,\end{align}
   where the value of $\phiFB{k+1}$ is taken at step 8.
  \item[(ii)]For each $k\geq 1$ at least one of the following 3 conditions hold:
  \begin{enumerate}
  \item[(a)]$\ee \lambda k>\lb/2$.
  \item[(b)]$f(\ee zk)> \ell_f(\ee xk,\ee zk)+\frac{\sigma\alpha}{2\ee\lambda k}\norm{\ee zk-\ee xk}^2$.
  \item[(c)] There is some $\tilde z\in T_{2\ee \lambda k}(\ee xk)$ such that $f(\tilde z)>\ell_f(\ee xk,\tilde z)+\frac\alpha{4\ee \lambda k}\norm{\tilde z-\ee xk}^2$.
  \end{enumerate}
\end{enumerate}
\end{lemma}
\begin{proof}
  Only condition  \eqref{EqDecrFB1} needs some explanation. The bound \eqref{EqDecrFB1} surely holds with the value of $\phiFB{k+1}$ when the {\tt while} loop at step 6. is terminated. In the {\tt while} loop at step 7. the value $\ee\lambda{k+1}$ can only increase and therefore the bound \eqref{EqDecrFB1} is retained by Lemma \ref{LemBasicFBE}(iii).
\end{proof}

Before stating convergence properties of Algorithm \ref{AlgBasGSSN} we show that the {\tt while} loops are finite together with some other useful properties. In what follows, we index by
$\ee\lambda{0,j},\ee z{0,j}, \ee \eta{0,j}$, $j=0,\ldots,j_0$, the different values of $\ee\lambda 0,\ee z0,\ee\eta0$ encountered in the {\tt while} loop st step 2. Analogously, for each $k$ we denote by $\ee x{k+1,j},\ee\lambda{k+1,j},\ee z{k+1,j}, \ee \eta{k+1,j},\ee \tau{k,j}$, $j=0,\ldots,j_k$, the consecutive values of $\ee x{k+1},\ee\lambda{k+1}, \ee z{k+1},\ee\eta{k+1}$ and $\ee\tau k$ when the test expression in the {\tt while} loop at step 6 is evaluated. Finally, $\ee\lambda{k+1,j},\ee z{k+1,j}, \ee \eta{k+1,j}$, $j=j_k,\ldots,\bar j_k$, are the values of $\ee\lambda{k+1}, \ee z{k+1}$ and $\ee\eta{k+1}$ considered in the course of the {\tt while} loop at step 7.
\begin{lemma}\label{LemConvBasGSSN}
  Consider the iterates generated by Algorithm \ref{AlgBasGSSN}. The following hold.
  \begin{enumerate}
    \item[(i)] The {\tt while} loop at step 2. is finite.
    \item[(ii)] The {\tt while} loop at  step 6. is finite for each $k$.
    \item[(iii)] The {\tt while} loop at  step 7. is finite for each $k$.
    \item[(iv)] If the algorithm stops in step 3. with $\ee xk=\ee zk$ then $\ee xk$ is stationary for $\varphi$, i.e., $0\in\widehat\partial \varphi(\ee xk)$.
    \item[(v)] For each $k\geq 0$ there holds
    \begin{equation}\label{EqDecrFB2}
      \varphi(\ee z{k})+(1-\alpha)\ee\eta{k}\leq \phiFB{k}.
    \end{equation}
    \item[(vi)] For each bounded set $M\subset\R^n$ there is some positive constant $\underline{\lambda}$ such that for every iterate $k\geq 1$ with $\ee xk \in M$ there holds $\ee \lambda k\geq \underline{\lambda}$.
  \end{enumerate}
\end{lemma}
\begin{proof} ad (i): Assume at the contrary that by backtracking at step 2 we produce  infinite sequences $\ee\lambda{0,j}=2^{-j}\lambda_0$ and $\ee z{0,j}$. Applying Lemma \ref{LemBasicFBE}(ii) with $\bar M=\{\ee x0\}$ and $\bar\lambda=\ee\lambda 0$ we obtain that the set $S:=\{\ee z{0,j}\mv j\in\N\}$ is bounded.  Taking $M=S\cup\{\ee x0\}$, consider $\hat \lambda_{M,\alpha}$ according to Lemma \ref{LemBasicFBE}(i) in order to conclude that the {\tt while} loop  at step 2. is terminated  as soon as $\ee \lambda{0,j}\leq\hat\lambda_{M,\alpha}$.

 ad (ii): Assume that the assertion does not hold and assume that $k\geq 0$ is the first iteration index such that the {\tt while} loop at step 6. is not finite and we compute infinite sequences $\ee\tau{k,j}$, $\ee\lambda{k+1,j}$, $\ee x{k+1,j}$ and $\ee z{k+1,j}$. Since the sequence $\ee x{k+1,j}$, $j\in\N$ is bounded and $\ee\lambda{k+1,j}\leq \ee\lambda k$, by Lemma \ref{LemBasicFBE}(ii), we obtain that the sequence $\ee z{k+1,j}$ is bounded as well. Taking $M=\{\ee x{k+1,j}\mv j\in\N\}\cup \{\ee z{k+1,j}\mv j\in\N\}$ and $\hat\lambda_{M,\alpha}$ according to Lemma \ref{LemBasicFBE}(i), we conclude that $\ee \lambda{k+1,j}\geq\min\{\hat\lambda_{M,\alpha}/2,\ee\lambda k\}$ $\forall j$. Hence  there is some $\lambda\in[\min\{\hat\lambda_{M,\alpha}/2,\ee\lambda k\},\ee\lambda k]$ such that $\ee \lambda{k+1,j}=\lambda$ for all $j$ sufficiently large.  Thus $\ee\tau{k,j}\to 0$ and
 $\lim_{j\to\infty}\phiFB{{k+1,j}}=\lim_{j\to\infty}\varphi^{\rm FB}_\lambda(\ee x{k+1,j})=\varphi_\lambda^{\rm FB}(\ee zk)$ by continuity of the forward-backward envelope $\varphi_\lambda^{\rm FB}$. Since $f(\ee zk)\leq\ell_f(\ee xk,\ee zk)+\frac\alpha{2\ee\lambda k}\norm{\ee zk-\ee xk}^2$, we obtain from Lemma \ref{LemBasFBE1} that
 \begin{equation}\label{EqAuxBnd1}\varphi_\lambda^{\rm FB}(\ee zk)\leq \varphi(\ee zk)\leq \phiFB{k}-\frac{1-\ee\lambda k(\alpha/\ee \lambda k)}{2\ee \lambda k}\norm{\ee zk-\ee xk}^2=\phiFB{k}-(1-\alpha)\ee \eta k.\end{equation}
 From $\beta<1$ we deduce that for $j$ sufficiently large we have
 \begin{equation*}\phiFB{{k+1,j}}=\varphi^{\rm FB}_\lambda(\ee x{k+1,j})<\varphi_\lambda^{\rm FB}(\ee zk)+(1-\beta)(1-\alpha)\ee \eta k
 \leq \phiFB{k}-\beta(1-\alpha)\ee \eta k,
 \end{equation*}
 implying that the {\tt while} loop at step 6. is terminated.

 ad (iii) This clearly holds, since in the {\tt while} loop  we set $\ee \lambda{k+1,j+1}=2\ee\lambda{k+1,j}$ so that after finitely many cycles there holds $2\ee \lambda{k+1,j}>\bar\lambda$.

 ad (iv): By first-order optimality condition at $\ee z k=\ee xk$ we have
 \[0\in \widehat\partial_z\psi^{\rm FB}_{\ee \lambda k}(\ee xk,\ee zk)=\nabla f(\ee xk)+\frac 1{\ee \lambda k}(\ee zk-\ee xk)+\widehat\partial g(\ee zk)=\widehat\partial \varphi(\ee zk).\]

ad (v): The inequality in \eqref{EqDecrFB2} was already established in \eqref{EqAuxBnd1}.

ad (vi): Let $K:=\{k\geq 1\mv \ee xk\in M\}$. Applying Lemma \ref{LemBasicFBE}(ii) with $\bar M=M$ and $\lb$ from Algorithm \ref{AlgBasGSSN}, we obtain that all points $\ee zk$, $k\in K$ belong to the bounded set $S$ given by \eqref{EqS}. Let $L$ denote the Lipschitz constant of $\nabla f$ on the compact convex set $\cl\co(M\cup S)$. We claim that the assertion holds with $\underline\lambda=\min\{\lb/2,\sigma\alpha/L\}$. Consider the three possibilities according to Lemma \ref{LemOutIterate}(ii). If $\ee \lambda k>\lb/2$ then our claim certainly holds true. Hence assume that $\ee \lambda k\leq \lb/2$. If condition (b) of Lemma \ref{LemOutIterate}(ii) holds, together with the descent lemma we obtain
\[\ell_f(\ee xk,\ee zk)+\frac L2\norm{\ee zk-\ee xk}^2\geq f(\ee zk) > \ell_f(\ee xk,\ee zk)+\frac {\sigma\alpha}{2\ee \lambda k}\norm{\ee zk-\ee xk}^2\]
and $\ee \lambda k>\alpha\sigma/L\geq \underline{\lambda}$ follows. Finally, in case of Lemma \ref{LemOutIterate}(ii)(c) there holds $\tilde z\in S$ due to $2\ee \lambda k\leq\lb$ and therefore as before
\[\ell_f(\ee xk,\tilde z)+\frac L2\norm{\tilde z-\ee xk}^2\geq f(\tilde z) > \ell_f(\ee xk,\tilde z)+\frac {\alpha}{4\ee \lambda k}\norm{\tilde z-\ee xk}^2,\]
showing $\ee \lambda k> \frac 12 \alpha/L\geq \sigma\alpha/L\geq \underline\lambda$.
\end{proof}
In the following theorem we summarize the convergence properties of Algorithm \ref{AlgBasGSSN}. To this aim we assume that Algorithm \ref{AlgBasGSSN} does not prematurely stop at a stationary point and produces an infinite sequence of iterates.
\begin{theorem}\label{ThConvBasGSSN}
  Consider the iterates generated by Algorithm \ref{AlgBasGSSN}. Then either $\varphi(\ee zk)\to-\infty$ or the following hold.
  \begin{enumerate}
    \item[(i)] The sequence $\phiFB k$ is strictly  decreasing and converges to some finite value $\bar\varphi \geq \inf \varphi$.
    \item[(ii)] $\sum_{k\in\N}\ee\eta k=\sum_{k\in \N}\frac 1{2\ee \lambda k}\norm{\ee xk-\ee zk}^2<\infty$ and therefore $\lim_{k\to\infty}\ee \eta k=0$.
    \item[(iii)] $\lim_{k\to\infty}\norm{\ee zk-\ee xk}=0$.
    \item[(iv)] Every accumulation point of the sequence $\ee xk$ is also an accumulation point of the sequence $\ee zk$ and vice versa.
    \item[(v)] Every accumulation point $\xb$ of the sequence $\ee xk$ satisfies $\xb\in T_{\underline \lambda}(\xb)$ for some $\underline\lambda\in (0,\lb)$. In particular, $\xb$ is a stationary point for $\varphi$ fulfilling $0\in\hat \partial\varphi(\xb)$.
  \end{enumerate}
\end{theorem}
\begin{proof}
ad (i): By \eqref{EqDecrFB1}, the sequence $\phiFB k$ is strictly decreasing and $\limsup_{k\to\infty}\phiFB k\geq \limsup_{k\to\infty}\varphi(\ee zk)$ by \eqref{EqDecrFB2}. Hence, if $\limsup_{k\to\infty}\varphi(\ee zk)>-\infty$ then $\lim_{k\to\infty}\phiFB k=\limsup_{k\to\infty}\phiFB k$ is finite and satisfies $\lim_{k\to\infty}\phiFB k\geq \inf \varphi$.

ad (ii): From \eqref{EqDecrFB1} we deduce
\[\phiFB{k+1}\leq \phiFB 0-\sum_{i=0}^k(1-\alpha)\beta\ee \eta i\ \forall k\]
from which  the assertion follows.

ad (iii): Follows from (ii), since $\ee \lambda k\leq \bar\lambda$ $\forall k$.

ad (iv): Follows from (iii).

ad (v): Let $\xb$ be an accumulation point of $\ee xk$ and let $r>0$ be arbitrarily fixed. Considering a subsequence $\ee x{k_i}$ converging  to $\xb$, we may assume that $\ee x{k_i}\in\B_r(\xb)$, $\forall k_i$. Then, by Lemma \ref{LemConvBasGSSN}(vi) there is some $\underline\lambda$ such that $\ee\lambda{k_i}\geq \underline\lambda$ $\forall k_i$.
Now assume on the contrary that $\xb\not\in T_{\underline \lambda}(\xb)$. Then there is some $\tilde z\in T_{\underline \lambda}(\xb)$ and some $\epsilon>0$ such that
\[\psi_{\underline\lambda}^{\rm FB}(\xb,\tilde z)<\psi_{\underline\lambda}^{\rm FB}(\xb,\xb)-2\epsilon=\varphi(\xb)-2\epsilon,\]
implying that for all $k_i$ sufficiently large we have
\[\varphi_{\underline\lambda}^{\rm FB}(\ee x{k_i})\leq \psi_{\underline\lambda}^{\rm FB}(\ee x{k_i},\tilde z)\leq\varphi(\xb)-\epsilon.\]

Since $\varphi$ is lsc,  $\ee z{k_i}\to\xb$ and taking into account  \eqref{EqDecrFB2} and Lemma \ref{LemBasicFBE}(iii), we obtain the contradiction
\begin{equation*} \liminf_{i\to\infty}\phiFB{k_i}\leq \varphi_{\underline\lambda}^{\rm FB}(\ee x{k_i})\leq\varphi(\xb)-\epsilon < \varphi(\xb) \leq \liminf_{i\to\infty}\varphi(\ee z{k_i}) \leq \liminf_{i\to\infty}\phiFB{k_i}.\end{equation*}
Hence $\xb\in T_{\underline\lambda}(\xb)$ and $0\in\widehat\partial\varphi(\xb)$ follows from \eqref{EqFO}.
\end{proof}

At the end of this section we want to discuss some termination criteria for Algorithm \ref{AlgBasGSSN}. Given the iterate $\ee xk$, $\ee zk\in T_{\ee\lambda k}(\ee xk)$, by the first-order optimality condition \eqref{EqFO} we have
\begin{equation}\label{Eqz*g}\ee{z^*_g}k :=-\nabla f(\ee x{k})-\frac1{\ee \lambda k}(\ee zk-\ee xk)\in\widehat\partial g(\ee zk)\end{equation}
and therefore
\begin{equation}\label{Eqz*k}
  \ee{z^*}k:=\nabla f(\ee zk)+\ee {z^*_g}k\in\widehat\partial \varphi(\ee zk).
\end{equation}
Hence, $\dist{0,\partial\varphi(\ee zk)}\leq\norm{\ee {z^*}k}$ and thus the standard relative termination criterion
\[\norm{\ee{z^*}k}\leq\epsilon\norm{\ee{z^*}0}\]
is suitable. However, this termination criterion might be too stringent if the starting point $\ee x0$ is close to a solution of the problem and therefore the termination criterion should be modified to
\begin{equation}
  \norm{\ee{z^*}k}\leq\epsilon\max\{{\rm typ\_subgr},\norm{\ee{z^*}0}\},
\end{equation}
where ${\rm typ\_subgr}$ denotes the magnitude of a ``typical'' subgradient. If such a value is not at hand it can be estimated by the value $\norm{\ee {z^*}0}$ obtained by running steps 1. and 2. of Algorithm BasGSSN with a randomly generated starting point $\ee x0$.

In our implementation we use an alternative  termination criterion. As we are looking for a solution of the inclusion
\[0\in F(x,z):=\begin{pmatrix}
  \nabla f(x)+\partial g(z)\\x-z
\end{pmatrix}\]
and $(\frac 1{\ee\lambda k}(\ee xk-\ee zk), \ee xk-\ee zk)^T \in F(\ee xk,\ee zk)$ by \eqref{Eqz*g}, a stopping criterion of the form
\begin{equation}\label{EqTermRes}
 \ee rk := \left(1+\frac 1{\ee\lambda k}\right)\norm{\ee xk-\ee zk}\leq\epsilon\max\left\{{\rm typ\_val},\ee r0\right\}
\end{equation}
can be used, where ${\rm typ\_val}$ is the magnitude of a ``typical'' value of $(1+\frac 1\lambda)\dist{x,T_\lambda(x)}$.

Note that by Theorem \ref{ThConvBasGSSN} we can only ensure that the termination criterion \eqref{EqTermRes} is achieved in finite time if the sequence $\ee \lambda k$ has an accumulation point which is bounded away from $0$. Apart from the case when $\nabla f$ is Lipschitzian on $\R^n$, by Lemma \ref{LemConvBasGSSN}(vi) this is also the case when the produced sequence $\ee xk$ (or $\ee zk$) has at least one accumulation point. From \eqref{EqDecrFB1}, \eqref{EqDecrFB2} we immediately obtain that $\varphi(\ee zk)\leq\phiFB0$ $\forall k$. Hence, if $\varphi$ is level bounded, the produced sequence $\ee zk$ remains bounded and finite termination is guaranteed.
In case that $\varphi$ is not level-bounded or is even not bounded below, we can enforce finiteness of Algorithm \ref{AlgBasGSSN} by adding the stopping conditions
\begin{equation}\label{EqStoppCond}\ee\eta k \leq\epsilon\big(\phiFB 0-\phiFB k\big)\quad\vee\quad \varphi(\ee zk)\leq\varphi_{\min}.\end{equation}

\begin{remark}
  Since Algorithm \ref{AlgBasGSSN} is a modification of PANOC$^+$, many of the assertions of Lemma \ref{LemConvBasGSSN} and Theorem \ref{ThConvBasGSSN} follow also from \cite{MaTh22}. The only difference is that we can prove stationarity of accumulation points of the sequence $\ee xk$ (or $\ee zk$) whereas in \cite{MaTh22} this is only shown under the additional assumption that $\ee xk$ remains bounded. On the other hand, since we do not impose some bound of the form $\norm{\ee sk}\leq D\norm{\ee zk-\ee xk}$ on the search directions $\ee sk$ like in PANOC$^+$, we cannot prove that $\liminf_{k\to\infty} \frac 1{\ee \lambda k}\norm{\ee zk-\ee xk}=0$.
\end{remark}
\section{Combining BasGSSN with the SCD semismooth$^{*}$ Newton method (GSSN)}

We will refer to GSSN (Globalized SCD Semismooth$^{*}$ Newton method) when we use a realization of the basic Algorithm \ref{AlgBasGSSN} with search direction $\ee sk$ related to the SCD semismooth$^{*}$ Newton direction given by \eqref{EqSCDNewtonDir}. To this aim we assume throughout this section that Assumption \ref{AssProx} is fulfilled.

Given the iterate $\ee xk$, $\ee zk\in T_{\ee\lambda k}(\ee xk)$ we will take  the point $(\ee zk,\ee {z^*}k)\in\gph \partial \varphi$, with $\ee {z^*}k$ given by \eqref{Eqz*k},
as approximation step in the SCD \ssstar Newton method applied to the inclusion $0\in\partial\varphi(x)$.
Selecting two matrices $(\ee Pk,\ee Wk)\in \M_{P,W}\partial\varphi(\ee zk, \ee{z^*}k)$, it is near at hand to take as search direction $\ee sk$ in Algorithm \ref{AlgBasGSSN} the \SCD \ssstar Newton direction given as the solution of the linear equation
\begin{equation}\label{EqLinSystNewtonStep}\ee Wk s=-\ee Pk\ee{z^*}k.\end{equation}
However, from the viewpoint of global convergence this is not feasible. We do not know in general, whether this system has a solution and, if it has a solution, its norm is less or equal than the given radius $\bar\rho$. Further it is also possible that the problem dimension $n$ is so large that we cannot solve the linear system \eqref{EqLinSystNewtonStep} exactly and we must confine ourselves with some approximate solution produced by an iterative method.

Hence, for the moment we assume that our search direction $\ee sk$ satisfies $\ee sk\in\rge \ee Pk$, $\norm{\ee sk}\leq\bar\rho$ and is some approximate solution of \eqref{EqLinSystNewtonStep}. We denote the relative residual by
\begin{equation}\label{EqRelRes}\ee\xi k:=\frac{\norm{\ee Wk \ee sk+\ee Pk\ee{z^*}k}}{\norm{\ee {z^*}k}}.
\end{equation}

We now prove a very general result on superlinear convergence of our method. It only depends on the relative residuals $\ee \xi k$ but not on the way how we actually compute $\ee sk$.

\begin{theorem}\label{ThSuperLinConv} Consider the iterates generated by Algorithm \ref{AlgBasGSSN} and suppose that $\ee sk\in \rge\ee Pk$ $\forall k$ as described above. Assume that $\varphi(\ee zk)$ is bounded below and assume that the sequence $\ee xk$ has a limit point $\xb$ which is a local minimizer for $\varphi$. Further suppose that $\partial\varphi$ is both \SCD regular and \SCD \ssstar at $(\xb,0)$.
If  for every subsequence $\ee x{k_i}\to  \xb$ the corresponding subsequence of relative residuals $\ee \xi {k_i}$ converges to $0$,
then $\ee xk$ converges superlinearly to $\xb$ and for all $k$ sufficiently large there holds $\ee x{k+1}=\ee zk+\ee sk$.
\end{theorem}
\begin{proof}
    By Proposition \ref{Prop_sSR} and \cite[Corollary 3.5]{DruMoNhg14} it follows that the quadratic growth condition holds at the local minimizer $\xb$, i.e., there exists some reals $C,\delta>0$ such that
    \begin{equation}\varphi(x)\geq \varphi(\xb)+C\norm{x-\xb}^2\ \forall x\in \B_\delta(\xb).\label{EqQuadrGrowth}\end{equation}
    Consider the set
    \[S:=\bigcup_{\AT{x\in \B_{\delta/2}(\xb)}{\lambda\in (0,\lb]}}T_\lambda(x),\]
    which is bounded by Lemma \ref{LemBasicFBE}(ii), and let $L$ denote the Lipschitz constant of $\nabla f$ on the compact convex set $\cl\co(\B_{\delta/2}(\xb)\cup S)$.
    By Lemma \ref{LemConvBasGSSN}(vi) there is some $\underline \lambda>0$ such that $\ee \lambda k\geq \underline\lambda$ whenever $\ee xk\in\B_{\delta/2}(\xb)$ and a close look on the proof of Lemma \ref{LemConvBasGSSN}(vi) shows that we can take $\underline\lambda=\min\{\lb/2,\sigma\alpha/L\}$. By Theorem \ref{ThConvBasGSSN} we know that $\lim_{k\to\infty}\norm{z_k-x_k}=0$ and hence there is some index $k_\delta\geq 1$ such that $\norm{z_k-x_k}<\delta/2$ $\forall k\geq k_\delta$.
   Next consider the index set $K:=\{k\geq k_\delta\mv \ee xk\in \B_{\delta/2}(\xb)\}$ and we claim that
   \[\norm{\ee zk-\xb}\leq c_1\norm{\ee xk -\xb}\mbox{ with }c_1:=\sqrt{\frac{L+1/\underline{\lambda}}{2C}}\ \forall k\in K.\]
  Indeed, for arbitrarily fixed  $k \in K$ we have $\ee zk\in\B_\delta(\xb)$ and $\ee \lambda k> \underline{\lambda}$ implying
  \begin{align*}\phiFB k&\leq \psiFBp k{\xb}\\
  &=\varphi(\xb)+f(\ee xk)+\nabla f(\ee xk)(\xb-\ee xk)-f(\xb)+ \frac 1{2\ee\lambda k}\norm{\xb-\ee xk}^2\\
  &\leq \varphi(\xb)+\frac{L+1/\underline{\lambda}}2\norm{\xb-\ee xk}^2.
  \end{align*}
  Using the bound \eqref{EqDecrFB2} and the quadratic growth condition we conclude that
  \begin{align*}
    C\norm{\ee zk-\xb}^2\leq \varphi(\ee zk)-\varphi(\xb)\leq \phiFB k-\varphi(\xb)-(1-\alpha)\ee\eta k\leq \frac{L+1/\underline{\lambda}}2\norm{\xb-\ee xk}^2
  \end{align*}
  and the claimed inequality follows.

 Further, by \eqref{Eqz*k} we have
  \begin{align}\nonumber\norm{\ee {z^*}k}&=\norm{\nabla f(\ee zk)+\ee {z_g^*}k}=\norm{\nabla f(\ee zk)-\nabla f(\ee xk)-\frac1{\ee \lambda k}(\ee zk-\ee xk)}\\
  \nonumber&\leq \norm{\nabla f(\ee zk)-\nabla f(\ee xk)}+\frac1{\ee \lambda k}\norm{(\ee zk-\ee xk)}\\
  \label{EqNormSubgr1}&\leq c_2\norm{\ee zk-\ee xk}\\
  \label{EqNormSubgr2}&\leq c_2(\norm{\ee zk-\xb}+\norm{\ee xk-\xb})\leq c_3\norm{\ee xk-\xb}\end{align}
  with constants $c_2:=\left(L+ 1/\underline{\lambda}\right)$ and $c_3:=c_2(c_1+1)$,
  implying
  \[\norm{(\ee zk,\ee {z^*}k)-(\xb,0)}\leq (c_1+c_3)\norm{\ee xk-\xb}.\]
  Thus the bound \eqref{EqApprStepF} holds with $\eta:=c_1+c_3$ and we have $(\ee zk,\ee {z^*}k)\in \A_{\eta,\xb}(\ee xk)$ for all $k\in K$.
  Set $\kappa:=\scdreg\partial \varphi(\xb,0)$ and fix some $\epsilon>0$ with
  \begin{equation}\label{EqEpsilon}\epsilon(\eta+c_3)<\frac 12,\quad (L+1/\underline{\lambda})\epsilon^2\leq C,\quad 4(L+1/\lb)\epsilon^2c_2^2\leq \frac{(1-\alpha)(1-\beta)}{2\lb}. \end{equation}
  By Proposition \ref{PropConvNewton} and Lemma \ref{Lem_scdreg} we can find some radius $\delta_z\in (0,\delta/2)$ such that for each $(z,z^*)\in\gph\partial \varphi\cap \B_{\delta_z}(\xb,0)$ and each $L\in\Sp_{P,W}\partial\varphi(z,z^*)$ one has $L\in\Z_n^{\rm reg}$ and
  \begin{equation}\label{EqAuxC_L}\norm{C_{L^*}}=\norm{C_L}\leq \kappa+1\quad\mbox{and}\quad\norm{z-C_Lz^*-\xb}\leq \epsilon\norm{(z-\xb,z^*)}.\end{equation}
  Further we can choose $\delta_z$ so small that for every iteration index $k$ with $\ee xk\in\B_{\delta_z/c_1}(\xb)$ one has
  \begin{equation}\label{Eq_xi_k}(\kappa+1)\ee\xi k<\epsilon.\end{equation}
  Set $\delta_x:=\min\{\delta,\delta_z/\eta\}$ and consider an iteration $k\in K$ with $\ee xk\in\B_{\delta_x}(\xb)$. Then $(\ee zk,\ee{z^*}k)\in\gph\partial \varphi\cup\B_{\delta_z}(\xb,0)$ and therefore $\ee Lk:=\rge(\ee Pk,\ee Wk)\in\Z_n^{\rm reg}$. Thus the exact \SCD \ssstar Newton direction $\ee {\bar s}k:=-C_{\ee Lk}\ee{z^*}k=-\ee Pk{\ee Wk}^{-1}\ee {z^*}k$ is well-defined and satisfies
  \[\norm{\ee zk+\ee{\bar s}k-\xb}\leq \epsilon\norm{(\ee zk-\xb,\ee {z^*}k)}\]
  by \eqref{EqAuxC_L}. Further, since $\ee sk\in\rge \ee Pk$ and the symmetric matrices $\ee Pk$ and $\ee Wk$ commute, we obtain that
  \[\norm{\ee sk-\ee{\bar s}k}=\norm{\ee Pk \ee sk-\ee{\bar s}k}=\norm{{\ee Wk}^{-1}\ee Pk(\ee Wk \ee sk+\ee Pk\ee {z^*}k)}\leq \norm{C_{\ee Lk}}\ee \xi k\norm{\ee {z^*}k}\]
  and consequently
  \begin{align}\nonumber\norm{\ee zk+\ee sk-\xb}&\leq \norm{\ee zk+\ee{\bar s}k-\xb}+\norm{\ee sk-\ee{\bar s}k}\leq \epsilon\norm{(\ee zk-\xb,\ee {z^*}k)}+\norm{C_{\ee Lk}}\ee \xi k\norm{\ee {z^*}k}\\
  \label{EqAuxBndx_k+1(1)}&\leq \epsilon\norm{\ee zk-\xb}+ \big(\epsilon+(\kappa+1)\ee \xi k\big)c_2\norm{\ee zk-\ee xk}\\
  \nonumber&\leq \Big(\epsilon c_1+\big(\epsilon+(\kappa+1)\ee \xi k\big)c_3\Big)\norm{\ee xk-\xb}=\big(\epsilon\eta+(\kappa+1)\ee \xi k c_3)\norm{\ee xk-\xb}\\
  \label{EqAuxBndx_k+1(2)}&\leq\epsilon(\eta+c_3)\norm{\ee xk-\xb}<\frac 12\norm{\ee xk-\xb}
  \end{align}
  by \eqref{EqNormSubgr1}-\eqref{EqEpsilon} and \eqref{Eq_xi_k}. Using the notation introduced after Lemma \ref{LemOutIterate},
  we will now show that for every $j=0,\ldots,j_k$ we have $\ee \tau{k,j}=1$ resulting in $\ee x{k+1,j}=\ee zk+\ee sk$ and $\ee \lambda{k+1,j}=2^{-j}\ee \lambda{k+1,0}\geq 1/{\underline \lambda}$. Indeed, since $\ee zk\in \B_{\delta_z}(\xb)$ and $\ee x{k+1,0}=\ee zk+\ee sk\in\B_{\delta_x/2}(\xb)$, we certainly have $\ee x{k+1,j}\in \B_{\delta/2}(\xb)$ and therefore $\ee z{k+1,j}\in S$. Thus, if $\ee \lambda{k+1,j}\leq\alpha/L$, we have by Lemma \ref{LemDesc} that
  \[f(\ee z{k+1,j})\leq \ell_f(\ee x{k+1,j},\ee z{k+1,j})+\frac L2\norm{\ee z{k+1,j}-\ee x{k+1,j}}^2\leq \ell_f(\ee x{k+1,j},\ee z{k+1,j})+\alpha\ee \eta{k+1,j}\]
  and therefore $\ee \lambda{k+1,j}$ is not further reduced. It follows that for all $j=0,\ldots,j_k$ we have $\ee \lambda{k+1,j}\geq \min\{\ee \lambda{k+1,0},\frac \alpha{2L}\}=\min\{\ee \lambda k,\frac \alpha{2L}\}\geq \underline{\lambda}$ because of $\sigma\leq 1/2$. Hence
  \begin{align*}\lefteqn{\varphi^{\rm FB}_{\ee \lambda{k+1,j}}(\ee zk+\ee sk)}\\
  &\leq\psi^{\rm FB}_{\ee \lambda{k+1,j}}(\ee zk+\ee sk,\xb)\leq\ell_f(\ee zk+\ee sk,\xb)-f(\xb)+\frac 1{2\underline{\lambda}}\norm{\xb-(\ee zk+\ee sk)}^2+f(\xb)+g(\xb)\\
  &\leq\varphi(\xb)+\frac {L+1/\underline{\lambda}}2\norm{\ee zk+\ee sk-\xb}^2\\
  &\leq \varphi(\ee zk)-C\norm{\ee zk-\xb}^2\\
  &\qquad+\frac {L+1/\underline{\lambda}}2\Big( 2\epsilon^2\norm{\ee zk-\xb}^2+2\big(\epsilon+(\kappa+1)\ee \xi k\big)^2c_2^2\norm{\ee zk-\ee xk}^2\Big)\\
  &\leq \phiFB k+\big((L+1/\underline{\lambda})\epsilon^2-C\big)\norm{\ee zk-\xb}^2+   4(L+1/\underline{\lambda})\epsilon^2c_2^2\norm{\ee zk-\ee xk}^2-(1-\alpha)\ee\eta k\\
  &\leq \phiFB k +\frac{(1-\beta)(1-\alpha)}{2\lb}\norm{\ee zk-\ee xk}^2-(1-\alpha)\ee\eta k\leq \phiFB k -\beta(1-\alpha)\ee\eta k,\end{align*}
  where we have used \eqref{EqAuxBndx_k+1(1)} together with the inequality $(a+b)^2\leq 2a^2+2b^2$, the quadratic growth condition \eqref{EqQuadrGrowth}, \eqref{EqDecrFB2} and the bounds \eqref{EqEpsilon},\eqref{Eq_xi_k} together with $\ee \lambda k\leq \lb$ and the equality $\ee\eta k=1/(2\ee \lambda k)\norm{\ee zk-\ee xk}^2$. Thus, the condition in the {\tt if}-statement of step 6. of Algorithm \ref{AlgBasGSSN} is always violated and we never reduce $\ee\tau{k,j}$. This proves our claim that $\ee x{k+1}=\ee x{k+1,0}$. By \eqref{EqAuxBndx_k+1(2)}, $\norm{\ee x{k+1}-\xb}<\frac 12\norm{\ee xk-\xb}$ and an induction argument yields that $\ee xk$ converges to $\xb$ and that $\ee x{k+1}=\ee zk+\ee sk$ for all $k$ sufficiently large. The superlinear convergence of $\ee xk$ is a consequence of \eqref{EqAuxBndx_k+1(2)} since we can choose $\epsilon>0$ arbitrarily small.
\end{proof}
\begin{remark}\label{RemRho}
\begin{enumerate}
\item
Note that in the proof of Theorem \ref{ThSuperLinConv} we did not use the condition $\norm{\ee sk}\leq \bar\rho$, which is only required in Theorem \ref{ThConvBasGSSN} for showing global convergence. In fact, it follows from the proof of Theorem \ref{ThSuperLinConv} that
\[\norm{\ee sk}\leq \norm{C_{\ee Lk}}(1+\ee \xi k)\norm{\ee {z^*}k}\]
and $\norm{\ee {z^*}k}\leq c_2\norm{\ee zk-\ee xk}\to 0$ showing $\norm{\ee sk}\to 0$. Thus the condition $\norm{\ee sk}\leq \bar\rho$ is automatically fulfilled, if we are sufficiently close to a stationary point of $\varphi$ at which the subdifferential $\partial \varphi$ is SCD regular.
\item One of the reviewers asked whether the convergence properties of GSSN, as a Newton-type method, are invariant to the choice of $\ee \lambda k$. The calculation of the search direction $\ee sk$ is actually independent of the parameter $\ee \lambda k$. However, as one can see from the proof of Theorem \ref{ThSuperLinConv}, the parameter  $\eta$ appearing in \eqref{EqApprStepF} and Proposition \ref{PropSingleStep}, depends on the quantity $\inf\{\ee \lambda k\mv \ee xk\in\B_{\delta/2}(\xb)\}$, at least in theory.
\end{enumerate}
\end{remark}

We discuss now two possibilities for computing $\ee sk$. We also require $\norm{\ee sk}\leq \ee\rho k$, where the radius $\ee\rho k\in[\underline \rho,\bar\rho]$ with $\underline \rho>0$ is updated according to the trust region concept in order to keep the number of backtrackings in Step 6. of Algorithm \ref{AlgBasGSSN} small:
\begin{equation}\label{Eqrho}
  \begin{minipage}[t]{14cm}{\tt if $\ee \tau k<0.25$ then}\\
  \mbox{\quad} $\ee \rho{k+1}\leftarrow\max\{\underline\rho,\norm{\ee sk}/2\}$;\\
  {\tt else if $\norm{\ee sk}==\ee\rho k\ \wedge\ \ee\tau k==1$ then}\\
  \mbox{\quad} $\ee \rho{k+1}\leftarrow\min\{\bar\rho,1.5\norm{\ee sk}\}$;\\
  {\tt else}\\
  \mbox{\quad} $\ee \rho{k+1}\leftarrow \ee\rho k$;
  \end{minipage}
\end{equation}

The first method for computing $\ee sk$ is based on the exact solution of \eqref{EqLinSystNewtonStep}:
\begin{equation}\label{EqSk_Newton}
  \begin{minipage}[t]{14cm}{\tt if $\rge(\ee Pk,\ee Wk)\in \Z_n^{\rm reg}$ then}\par
   \qquad $\ee sk \leftarrow -{\ee Wk}^{-1}\ee Pk \ee {z^*}k$, {\tt if}  $\norm{\ee sk}>\ee \rho k$ {\tt then} $\ee sk\leftarrow \ee\rho k\ee sk/\norm{\ee sk}$\\
   {\tt else}\par
   \qquad $\ee sk\leftarrow -\ee\rho k\ee Pk\ee{z^*}k/\norm{\ee Pk\ee{z^*}k}$.
  \end{minipage}
\end{equation}
Since $\ee Wk$ and $\ee Pk$ commute, we automatically have $-{\ee Wk}^{-1}\ee Pk \ee {z^*}k\in\rge \ee Pk$.

\begin{corollary}\label{CorSuperlinConv1}Consider the iterates generated by Algorithm \ref{AlgBasGSSN} with $\ee sk$ computed by \eqref{EqSk_Newton}. Assume that $\varphi(\ee zk)$ is bounded below and assume that the sequence $\ee xk$ has a limit point $\xb$ which is a local minimizer for $\varphi$. Further suppose that $\partial\varphi$ is both \SCD regular and \SCD \ssstar at $(\xb,0)$. Then $\ee xk$ converges superlinearly to $\xb$.
\end{corollary}
\begin{proof}
  By the proof of Theorem \ref{ThSuperLinConv}, for all $\ee xk$ sufficiently close to $\xb$  we have $\rge(\ee Pk,\ee Wk)\in \Z_n^{\rm reg}$ and the Newton direction $\ee {\bar s}k$
  satisfies
  \[\norm{\ee{\bar s}k}\leq(\kappa+1)\norm{\ee {z^*}k}\leq (\kappa+1)c_3\norm{\ee xk-\xb}.\]
  Thus, for every $\ee xk$ sufficiently close to $\xb$ we have $\norm{\ee {\bar s}k}\leq\underline\rho\leq\ee \rho k$ and the Newton direction $\ee sk=-{\ee Wk}^{-1}\ee Pk \ee {z^*}k$ is returned by \eqref{EqSk_Newton} resulting in $\ee \xi k=0$. Then superlinear convergence follows from Theorem \ref{ThSuperLinConv}.
\end{proof}

For large $n$ the computation of the exact solution of the Newton system \eqref{EqLinSystNewtonStep} might be very time consuming and it seems to be advantageous  to compute only an approximate solution by means of an iterative method. This is our second approach. Since we are dealing with quadratic optimization problems, the use of conjugate gradients (CG) seems to be appropriate. Motivated by Lemma \ref{LemStatPoint}, we try to find an approximate solution of the subproblem
\begin{equation}\label{EqTRSubProbl_s}\min\frac 12 s^T\ee Wk s+ \skalp{\ee {z^*}k,s}\quad \mbox{subject to}\quad s\in\rge \ee Pk,\ \norm{s}\leq \ee\rho k\end{equation}
Consider an $n\times m$ matrix $\ee Zk$, where $m=\dim\rge \ee Pk$, whose columns form an  basis for $\rge \ee Pk$ and therefore $\ee Pk=\ee Zk\big({\ee Zk}^T\ee Zk\big)^{-1}{\ee Zk}^T$. Then the problem above can be equivalently rewritten as
\begin{equation}\label{EqTRSubProbl_p}
  \min_{u\in\R^m} \frac 12 u^T {\ee Zk}^T\ee Wk \ee Zk u+\skalp{{\ee Zk}^T\ee {z^*}k,u}\quad\mbox{subject to}\quad\norm{\ee Zk u}\leq \ee \rho k.
\end{equation}
Any solution $s$ to \eqref{EqTRSubProbl_s} induces a solution $u=\big({\ee Zk}^T\ee Zk\big)^{-1}{\ee Zk}^Ts$ to \eqref{EqTRSubProbl_p} and, vice versa, $s={\ee Zk}u$ is a solution to \eqref{EqTRSubProbl_s} for every solution $u$ to \eqref{EqTRSubProbl_p}. Consider the following algorithm:
\begin{algorithm}\label{AlgCG}{Preconditioned  Trust Region CG}\\
{\bf Input: }$n\times n$ matrix $\ee Wk$, $n\times m$ matrix $\ee Zk$, subgradient $\ee {z^*} k$, radius $\ee\rho k$, stopping tolerance $0<\ee{\bar \xi}k$.\\
1. \begin{minipage}[t]{\myAlgBox}Choose a symmetric positive definite $m\times m$ matrix $\ee Ck$ (preconditioner), set $\ee r0\leftarrow {\ee Zk}^T\ee{z^*}k$, $\ee u0\leftarrow 0$, $\ee p{0}\leftarrow -{\ee Ck}^{-1}\ee r0$, $j\leftarrow 0$;\end{minipage}\\
2. {\tt while} $\norm{\ee Z k\ee uj}\leq \ee \rho k\ \wedge \norm{\ee Wk\ee Zk\ee uj +\ee Pk\ee{z^*}k}\geq\ee{\bar\xi}k\norm{\ee{z^*}k}$ {\tt do}\\
3. \mbox{\qquad}Compute $\ee yj\leftarrow {\ee Z k}^T\ee W k\ee Z k\ee pj$;\\
4. \mbox{\qquad}{\tt if} $\skalp{\ee yj,\ee pj}\leq 0$ {\tt then}\\
 \mbox{\qquad\quad\quad}{\tt return} $\ee sk$ as a global solution of the problem
\begin{equation}\label{EqTwoDimTRProbl}\min \frac 12 s^T\ee W ks + \skalp{z^*,s}\mbox{ subject to }\norm{s}\leq \ee \rho k,\ s\in{\rm span\,}\{\ee Z k\ee p0, \ee Z k\ee pj\};\end{equation}
5. \mbox{\qquad}Compute
\begin{align*}&\ee\alpha j\leftarrow-\frac{\skalp{\ee rj,\ee pj}}{\skalp{\ee yj,\ee pj}},\quad \ee u{j+1}\leftarrow \ee uj+\ee\alpha j\ee pj,\quad \ee r{j+1}\leftarrow \ee rj+\ee\alpha j\ee yj,\\ &\ee p{j+1}\leftarrow -{\ee Ck}^{-1}\ee r{j+1}+\frac{\skalp{\ee r{j+1}, {\ee Ck}^{-1}\ee r{j+1}}}{\skalp{\ee rj, {\ee Ck}^{-1}\ee rj}}\ee pj;\end{align*}
6. \mbox{\qquad}$j\leftarrow j+1$;\\
7. {\tt end while}\\
8. {\tt if} $\norm{\ee Z k\ee uj}>\ee \rho k$ {\tt then}\\
\mbox{\quad\quad}{\tt return} $\ee sk\leftarrow \ee \rho k \ee Z k\ee uj/\norm{\ee Z k\ee uj}$;\\
\mbox{\quad}{\tt else}\\
\mbox{\quad\quad}{\tt return} $\ee sk\leftarrow \ee Z k\ee uj$;
\end{algorithm}
In Algorithm \ref{AlgCG} we apply the well-known preconditioned CG method, until we either encounter a direction of negative curvature in step 4., the search direction is longer than the given radius $\ee \rho k$ or we get a sufficient decrease in the gradient.

Note that \eqref{EqTwoDimTRProbl} is in fact a two-dimensional problem, which is computationally inexpensive to solve. After some
algebraic manipulation it can be reduced to finding the roots of a fourth degree polynomial. However, in our implementation we used an iterative procedure (see, e.g, \cite[Chapter 4.3]{NoWr06}) for solving this subproblem. The condition  $\skalp{\ee yj,\ee pj}=\skalp{{\ee Z k}^T\ee W k\ee Z k\ee pj,\ee pj}\leq 0$ implies that the matrix ${\ee Z k}^T\ee W k\ee Z k$ is not positive semidefinite. Further, if $\skalp{\ee yj,\ee pj}<0$ then the solution $\ee sk$ of \eqref{EqTwoDimTRProbl} satisfies $\norm{\ee sk}=\ee \rho k$.

As an alternative to Algorithm \ref{AlgCG} one could also use   the conjugate gradient (CG) algorithm due to Steihaug \cite{Stei83} for solving \eqref{EqTRSubProbl_p}.

In case when $\ee sk$ is computed by Algorithm \ref{AlgCG}, we need some stronger assumption on the minimizer $\xb$ to guarantee that the matrix $\ee Wk$ is positive definite for $\ee xk$ close to $\xb$ and therefore $\ee sk$ is not computed by \eqref{EqTwoDimTRProbl}.
\if{
The following definition is due to Rockafellar \cite{Ro19}.
\begin{definition}\label{DefVarConv}
 $\varphi$ is called {\em (strongly) variationally convex} at $\xb$ for $\xba\in\partial \varphi(\xb)$ if for some convex neighborhood $U\times V$ of $(\xb,\xba)$ there exist an lsc (strongly) convex function $q\leq\varphi$ on $U$  and a scalar $\epsilon>0$ such that
\[ (U_\epsilon\times V)\cap\gph\partial\varphi = (U\times V)\cap\gph\partial q\quad\mbox{and $\varphi(x)=q(x)$ at the common elements $(x, x^*)$,}\]
where $U_\epsilon:=\{x\in U\mv \varphi(x)<\varphi(\xb)+\epsilon\}$.
\end{definition}
By \cite[Lemma 5.7]{KhMoPh23}, a function $\varphi:\R^n\to\oR$ which is prox-regular and subdifferential continuous at a local minimizer $\xb$ for $0$, is strongly variational convex at $\xb$ for $0$ if and only if $\xb$ is a tilt-stable local minimizer.
}
\fi
\begin{corollary}
\label{CorSuperlinConv2}Consider the iterates generated by Algorithm \ref{AlgBasGSSN} with $\ee sk$ computed by Algorithm \ref{AlgCG} with tolerances $\ee{\bar \xi}k:=\chi(\norm{\ee {z^*}k})$, where $\chi:(0,\infty)\to(0,1)$ is a monotonically increasing function with $\lim_{t\downarrow 0}\chi(t)=0$. Assume that $\varphi(\ee zk)$ is bounded below and assume that the sequence $\ee xk$ has a limit point $\xb$ which is a tilt-stable local minimizer for $\varphi$. Further suppose that $\varphi$ is prox-regular and subdifferentially  continuous at $\xb$ for $0$ and that  $\partial\varphi$ is \SCD \ssstar at $(\xb,0)$. Then $\ee xk$ converges superlinearly to $\xb$.
\end{corollary}
\begin{proof}
   Since $\partial\varphi$ is SCD regular at $(\xb,0)$ by Theorem \ref{ThTiltStab}, the real $\beta:=2/\scdreg\partial\varphi(\xb,0)$ is positive, By Proposition \ref{PropTiltStab} together with Lemma \ref{Lem_scdreg} we can find a neighborhood $U\times V$ of $(\xb,0)$ such that for every $(x,x^*)\in\partial\varphi\cap(U\times V)$ and  every $(P,W)\in \M_{P,W}\partial\varphi(x,x^*)$ the matrix $W$ is positive definite and there holds
   \[s^TPWPs\geq\beta\norm{Ps}^2\ \forall s\in\R^n.\]
    Thus, whenever $(\ee zk,\ee {z^*}k)\in \inn(U\times V)$ then  the matrix $\ee Wk$ is positive definite and the obtained search  direction $\ee sk$ fulfills ${\ee sk}^T\ee Pk\ee Wk\ee Pk\ee sk\geq \beta\norm{\ee Pk\ee sk}^2$. Hence the condition in the {\tt if} statement in step 4. of Algorithm \ref{AlgCG} is never fulfilled. Since $\ee Pk\ee sk=\ee sk$ and the objective function values at the iterates produced by the CG-method are strictly decreasing, we obtain
  \begin{align*}0&>\frac 12 {\ee sk}^T\ee Wk\ee sk+\skalp{\ee {z^*}k,\ee sk}\geq \frac 12 \beta\norm{\ee sk}^2-\norm{\ee {z^*}k}\norm{\ee sk}.
  \end{align*}
  Thus $\norm{\ee sk}<2\norm{\ee {z^*}k}/\beta$ showing that $\norm{\ee sk}<\underline\rho<\ee\rho k$ whenever $\norm{\ee {z^*}k}<\underline\rho\beta/2$ and in this case the relative residual fulfills $\ee\xi k\leq \ee {\bar \xi}k$. Since for every subsequence $\ee x {k_i}\to\xb$ we have $\ee {z^*}{k_i}\to 0$ by \eqref{EqNormSubgr1}, we conclude $\ee\xi{k_i}\to0$ and the assertion follows from Theorem \ref{ThSuperLinConv}.
\end{proof}
Note that the assumption of  a tilt-stable local minimizer (together with the other assumptions) is only sufficient for superlinear convergence but not necessary. Without this assumption
it is possible that some of the matrices $\ee Wk$ are not positive definite, but as soon as we encounter a positive definite matrix $\ee Wk$ at some iterate $\ee zk$ close to $\xb$, this results in a big improvement of $\ee zk$. Let us illustrate this behavior in the following example.
\begin{example}
  Consider the function $\varphi(x_1,x_2)=\frac 12 x_1^2 -\frac 12 x_2^2 +\delta_C(x_1,x_2)$, where
  \[C:=\{(x_1,x_2)\mv  -x_1\leq  2x_2\leq x_1\}.\]
   Then $\xb=(0,0)$ is a strict local minimizer which is  not tilt-stable, cf.\cite[Example 7.10]{GfrOut22a}. Now consider an iterate $\ee zk$ satisfying $0<2\ee{z_2}k<\ee {z_1}k$. Then easy calculations show that $\ee {z^*}k=(\ee {z_1}k,-\ee{z_2}k)$, $\ee Pk=I$ and that $\ee Wk=\left(\begin{smallmatrix}1&0\\0&-1\end{smallmatrix}\right)$ is not positive definite. Note that the direction computed by \eqref{EqSk_Newton} points to the exact solution. If we use Algorithm \ref{AlgCG} it might be that the search direction $\ee sk$ is computed by \eqref{EqTwoDimTRProbl} yielding a search direction $\ee sk$ with $\norm{\ee sk}=\ee\rho k$. However it can be easily verified that in this case the next iterate $\ee z{k+1}$  either belongs to the set $T:=\{(x_1,x_2)\mv 0\leq 2x_2=x_1\}\subset\bd C$ or is at least closer to $T$ as $\ee zk$. Due to the use of the proximal mapping, one can show that one of the subsequent iterates $\ee zj$, $j\geq k+1$ must actually belong to $T$. If $\ee zj\not=\xb$, then $\ee Pj =\frac 15\left(\begin{smallmatrix}4&2\\2&1\end{smallmatrix}\right)$ and $\ee Wj=\ee Pj\left(\begin{smallmatrix}1&0\\0&-1\end{smallmatrix}\right)\ee Pj+(I-\ee Pj)$ is positive definite and the  search direction $\ee sj$ points towards the solution.
\end{example}

\section{Numerical experiments}

In all our numerical experiments the smooth part $f$ of the objective is quadratic and we used the parameters $\alpha=0.8$, $\beta=0.2$, $\sigma=0.1$ and  $\bar\rho=10^5$ in Algorithm \ref{AlgBasGSSN}. As a stopping criterion we used the condition
\begin{equation}\label{EqStopping}
  \ee rk \leq10^{-13}\max\left\{10^3,\ \ee r0\right\}
\end{equation}
according to \eqref{EqTermRes}. In all our test instances the objective $\varphi$ is level bounded and therefore we need not take care of \eqref{EqStoppCond}.
For computing the search direction $\ee sk$ we used Algorithm \ref{AlgCG} with stopping parameter $\ee{\bar\xi} k=\chi(\norm{\ee {z^*}k})$,  $\chi(t):=0.1/(1-\ln \frac t{t+1})$ and a preconditioning matrix $\ee C k$ depending on the problem. The trust region radius $\ee\rho k$ is updated by \eqref{Eqrho} with $\underline\rho=10^{-3}$.

We compared our method with ZeroFPR \cite{ThStPa18} and stopping tolerance
\[{\tt opt.tol}=10^{-13}\max\{10^3, L_f\dist{\ee x0, T_{\frac 1{L_f}}(\ee x0)}\},\]
where $L_f$ denotes the Lipschitz constant of $\nabla f$, i.e., $L_f$ is the largest eigenvalue of $\nabla^2f$. This tolerance produces results with an accuracy comparable to \eqref{EqStopping}.

\subsection{Signorini problem with Tresca friction}
We consider an elastic body represented by the domain
\[\Omega=\{(x_1,x_2,x_3)\mv (x_1,x_2)\in(0,2)\times(0,1), d(x_1,x_2)<x_3<1\}\subset\R^3,\]
where $d:(0,2)\times(0,1)\to(0,1)$ is a function describing the bottom surface of the body.
The body is made of elastic, homogeneous, and isotropic material.
 The boundary consists of three parts $\Gamma_u=\{0\}\times[0,1]^2$,
$\Gamma_c:=\{(x_1,x_2,d(x_1,x_2))\mv (x_1,x_2)\in (0,2)\times(0,1)\}$ and $\Gamma_p=\partial\Omega\setminus (\Gamma_c\cup\Gamma_u)$.
Zero displacements are prescribed on $\Gamma_u$, surface tractions act on $\Gamma_p$, and the body is subject to volume forces. We seek a displacement field
and a corresponding stress field satisfying the Lam\'e system of PDEs in $\Omega$, the homogeneous Dirichlet
boundary conditions on $\Gamma_u$, and the Neumann boundary conditions on $\Gamma_p$. The body is unilaterally
supported along $\Gamma_c$ by some flat rigid foundation given by the half space $\R^2\times\R_-$ and the initial gap
between the body and the rigid foundation is given  by $d(x_1,x_2)$, $(x_1,x_2)\in [0,2]\times[0,1]$. In the contact zone $\Gamma_c$, we consider
a Tresca friction condition.
This problem can be described by partial differential equations and boundary conditions for the
displacements, which we are looking for. We refer the reader to, e.g., \cite{EJ}, where also a weak formulation
can be found. We consider here only the discrete algebraic problem, which arises after some
suitable finite element approximation.
Let $n$ denote the number of degrees of freedom of the nodal displacement vector and let $p$ denote
the number of nodes $x_i$ lying in the contact zone $\Gamma_c$. After some suitable reordering of the variables, such that
the first $3p$ positions are occupied by the displacements of the nodes lying in the contact part of the
boundary, we arrive at the following nodal block structure for an arbitrary vector $y\in\R^n$:
\[y = (y^1,\ldots,y^p,y^R) \mbox{ with } y^i \in\R^3,\ i = 1,\ldots,p,\ y^R\in\R^{n-3p}.\]
In what follows, $A\in\R^{n\times n}$ and  $l\in\R^n$ are the stiffness matrix and the load vector, respectively. The matrix $A$ is sparse, symmetric and positive definite and has no more than 81 nonzero elements in each row. Given a vector $z = (z_1, z_2, z_3)^T\in\R^3$, we denote by $z_{12}:= (z_1,z_2)^T\in\R^2$ the vector formed by the first
two components. Using this notation, the displacement vector $u\in\R^n$ is a solution of the problem
\begin{equation}\label{EqTresca}\min_{v\in\R^n}\frac 12 v^TAv-\skalp{l,v}+\sum_{i=1}^p\big(\F_i\norm{v^i_{12}}+\delta_{\R_+}(v_3^i+d(x^i_{12}))\big),\end{equation}
where $\delta_{R_+}$ denotes the indicator function of $\R_+$ and $\F_i>0$ is the given friction coefficient.

This problem is of the form \eqref{EqOptProbl}  with the convex quadratic function $f(v)=\frac 12 v^TAv-\skalp{l,v}$ and the nonsmooth convex function $g(v)=\sum_{i=1}^p\big(\F_i\norm{v^i_{12}}+\delta_{\R_+}(v_3+d(x_i))\big)$. Clearly, the proximal mapping ${\rm prox}_{\lambda g}$ is single-valued and given by $\bar v={\rm prox}_{\lambda g}(v)$ with
\[\bar v_{12}^i=\begin{cases}0&\mbox{if $\norm{v^i_{12}}\leq  \lambda\F_i$,}\\
\left(1-\frac{\lambda\F_i}{\norm{v^i_{12}}}\right)v^i_{12}&\mbox{else,}\end{cases}\quad \bar v^i_3=\max\{v_3^i, -d(x^i_{1,2})\},\ i=1,\ldots,p,\  \bar v^R=v^R.\]
The corresponding subgradient $\bar v^*\in\partial g(\bar v)$ can be easily computed by $\bar v^*=(v-\bar v)/\lambda$.

Given $(v,v^*)\in\gph \partial g$, we can compute an element
\[(P,W)=\left(\begin{pmatrix}P^1&&&0\\&\ddots&&\\&&P^p&\\0&&&P^R\end{pmatrix},
\begin{pmatrix}W^1&&&0\\&\ddots&&\\&&W^p&\\0&&&W^R\end{pmatrix}\right)\in \M_{P,W}\partial g(v,v^*)\] as block diagonal matrices formed by the $3\times 3$ blocks
\[P^i=\begin{pmatrix}P^i_{12}&0\\0&P^i_3\end{pmatrix},\  W^i=\begin{pmatrix}W^i_{12}&0\\0&W^i_3\end{pmatrix},\ i=1,\ldots,p\]
and the $(n-3p)\times (n-3p)$ matrices $P^R=I$, $W^R=0$, where
\[(P^i_{12},W^i_{12})=\begin{cases}\left(I,\frac{\F_i}{\norm{v^i_{12}}}\left(I-\frac{v^i_{12}{v^i_{12}}^T}{\norm{v^i_{12}}^2}\right)\right)&\mbox{if $v^i_{12}\not=0$,}\\
(0,I)&\mbox{if $v^i_{12}=0$,}\end{cases}\quad\mbox{and}\quad
(P^i_3,W^i_3)=\begin{cases}\left(1,0\right)&\mbox{if $v^i_3+d(x^i_{12})>0$,}\\
(0,1)&\mbox{if $v^i_3+d(x^i_{12})=0$,}\end{cases}\]
cf. Example \ref{ExEuclNorm}. We see that $P$ is a diagonal matrix with diagonal elements belonging to $\{0,1\}$ and therefore the unit vectors given by the nonzero columns of $P$ form an orthogonal basis for $\rge P$. Writing these unit vectors into an $n\times m$ matrix $Z$ and using Lemma \ref{LemSumRul}, problem \eqref{EqTRSubProbl_p} is of the form
\[\min_{u\in\R^m}\frac 12 u^TZ^T(A+W)Zu+\skalp{u, Z^T(Av-l+v^*)}\quad\mbox{subject to}\quad\norm{u}\leq\rho.\]
The SCD \ssstar Newton direction is then given by $s=Zu$, where $u$ denotes an (approximate) solution of this quadratic auxiliary problem.
Of course, the $m\times m$ matrix $Z^T(A+W)Z$ is the submatrix of $A+W$ formed by the rows and columns corresponding to the unit vectors in $Z$.

In Algorithm \ref{AlgCG}, we used incomplete Cholesky factorization for preconditioning. In Table \ref{TabResTresca1} we report the numerical results for various discretization levels.
\begin{table}
\[\begin{tabular}{|c|c|c|c|c|c|c|c||c|c||c|c|}
\cline{4-12}
\multicolumn{3}{c}{ }&\multicolumn{7}{|c||}{$\min f(v)+g(v)$}&\multicolumn{2}{|c|}{$\min f(v)$}\\
\cline{4-12}
\multicolumn{3}{c}{ }&\multicolumn{5}{|c||}{GSSN}&\multicolumn{2}{|c||}{ZeroFPR}&\multicolumn{2}{|c|}{Precond. CG}\\
\hline
  lev&n&p&iter&$f$-eval&prox&cg&time&iter&time&cg&time\\
  \hline
  3&1\,764&84&12&27&15&63&0.090&176&{\bf0.080}&53&0.020\\
  4&3\,888&144&13&31&18&88&0.235&226&{\bf0.196}&69&0.075\\
  5&11\,661&299&19&49&30&145&{\bf0.798}&351&1.001&101&0.336\\
  6&27\,744&544&17&44&27&177&{\bf2.255}&476&3.391&133&0.975\\
  7&79\,488&1\,104&16&36&20&247&{\bf7.770}&783&13.14&189&3.821\\
  8&209\,088&2\,112&20&48&28&374&{\bf31.07}&1095&54.75&263&14.98\\
  9&603\,057&4\,277&25&65&40&543&{\bf120.5}&1563&234.2&375&56.11\\
  10&1\,622\,400&8\,320&37&103&66&808&{\bf494.0}&2160&941.5&519&205.0\\
  \hline
\end{tabular}\]
\caption{\label{TabResTresca1}Performance of the semismooth$^{*}$ Newton method for the Signorini problem with Tresca friction}
\end{table}
Besides the total number of unknowns $n$ and the number $p$ of nodes in the contact zone  we present the count of Newton iterations needed to reach the termination criteria \eqref{EqStopping}. Further we state the number $f-eval$ of evaluations of the quadratic function $f$, the number $prox$ of evaluations of the proximal mapping of $g$ and the overall number $cg$ of CG iterations. Finally, the needed CPU time is specified. For the purpose  of comparison we present also the number of iterates and the CPU time needed by ZeroFPR and for minimizing the quadratic function $f$ with the CG-method using again incomplete Cholesky factorization for preconditioning. We can see that the CPU time needed by GSSN for solving the nonsmooth problem is approximately 2-2.5 times the one for minimizing the quadratic function. Further, GSSN is approximately twice as fast as ZeroFPR for finer discretizations levels.

As suggested by one referee, we compare our method also with FISTA \cite{BeTe09}. We computed first with GSSN a quite accurate solution $\xb$ of the problem and then ran GSSN and FISTA to obtain the relative errors $\norm{\ee xk-\xb}_\infty/\norm{\xb}_\infty$. We also computed the relative residuals $\ee rk/\ee r0$ for both GSSN and FISTA, where in FISTA the residual $\ee rk$ was computed with $\ee\lambda k\equiv 1/L_f$. The results are depicted in Figure \ref{FigGSSN_FISTA} for the two discretization levels $lev=7$ and $lev=10$.
\begin{figure}\centering
\includegraphics[width=7cm]{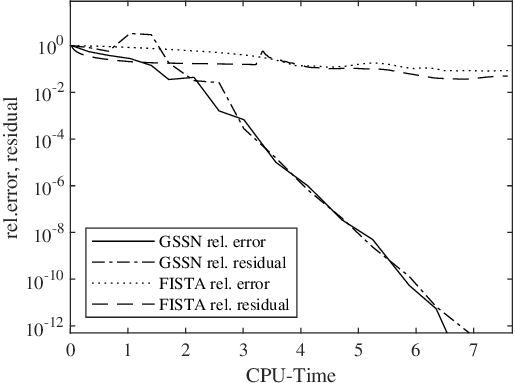}\quad
\includegraphics[width=7cm]{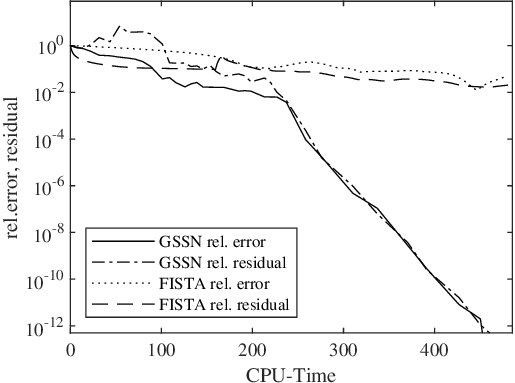}
\caption{\label{FigGSSN_FISTA}relative error $\frac{\norm{\ee xk-\xb}_\infty}{\norm{\xb}_\infty}$ and relative residuals $\frac{\ee rk}{\ee r0}$ for discretization levels $lev=7$ (left) and $lev=10$ (right)}
\end{figure}
We see that at the beginning the residuals of FISTA are up to a factor of 100 smaller than those of GSSN. However, this does not mean that the iterates approach the solution and, in fact, the iterates produced by GSSN are always closer to the solution than the ones computed by FISTA. For the two test instances, FISTA is not able to compute approximate solutions with more than one significant digit within reasonable time.

Let us mention that there exist a lot of approaches to solve contact problems by means of (classical) semismooth Newton methods, see, e.g., \cite{KuMoMaHa20}. These approaches have in common that they are applied either to a primal-dual or a dual formulation of the problem, whereas we are working exclusively with the primal problem. Compared with some recent numerical experiments \cite{KuMoMaHa20} for the dual problem we see that our numbers of CG-iterations are slightly higher than the ones which can be found in \cite{KuMoMaHa20}. This is due to the fact that we have produced solutions with a much higher accuracy. Further note that for our method, since $W$ is a block diagonal matrix and we are using incomplete Cholesky factorization for preconditioning,  the effort for performing one CG-iteration is approximately two times the one of forming the matrix vector product $Ax$ whereas  performing a CG-iteration for the dual problem requires a matrix vector product $(B^TA^{-1}B)x$ with some sparse $n\times 3p$ matrix $B$.

\subsection{(Nonconvex) sparse approximation\label{SubSecRegr}}

Now let us consider the problem of finding a sparse solution $x\in\R^n$ to a least-squares problem $Ax=b$, where $A$ is an $m\times n$ matrix and $b\in\R^m$. Sparsity can be induced by
considering the $\ell_q(q\in[0,1])$-regularized  problem
\begin{equation}\label{EqNonconvexLS}\min_{x\in\R^n}\frac 12\norm{Ax-b}^2+\mu \norm{x}_q^q,\end{equation}
where $\mu>0$ is the regularization parameter and
\[\norm{x}_q^q:=\sum_{i=1}^n\vert x_i\vert^q.\]
Thus the smooth part is given by $f(x)=\frac 12\norm{Ax-b}^2$ and the nonsmooth part by $g(x)=\mu \sum_{i=1}^n\vert x_i\vert^q$.
 In case $q=1$ we obtain the well-known lasso problem which is convex, whereas the resulting problem is nonconvex when $q\in[0,1)$. In our numerical experiments we restrict ourselves to the cases $q\in\{1,\frac 12\}$ for which the proximal mapping of $\vert \cdot\vert^q$ is explicitly known. We would like to point out that we can apply our method also when the proximal mapping can only be computed numerically, because usually only a few evaluations of the proximal mapping are required.

Due to the separable structure of $g$ only the proximal mapping of the one-dimensional functions $\vert\cdot\vert^q$ is required in order to compute ${\rm prox}_{\lambda g}$. Given an element $(x,x^*)\in\gph\partial g$, we can choose the matrices $(P, W)\in \M_{P,W}\partial g(x,x^*)$ as diagonal matrices with entries
\[( P_{ii}, W_{ii})=\begin{cases}(1, q(q-1)\mu\vert x_i\vert^{q-2})&\mbox{if $x_i\not=0$,}\\
(0,1)&\mbox{else,}\end{cases}\]
cf. Example \ref{ExEll_q}. Note that in case when $x_i\not=0$ we have  $W_{ii}=\mu\nabla^2\vert\cdot\vert^q(x_i)$. As for the Signorini problem, the nonzero columns of $P$, which correspond to the nonzero components of $x$, form an orthogonal basis for $\rge P$ and can be used to build the matrix $\ee Zk$ in \eqref{EqTRSubProbl_p}.

Although in our convergence analysis the \ssstar Newton direction $\ee sk$ is only used for ensuring superlinear convergence, in practice $\ee sk$ can be also used to accelerate convergence from poor starting points. To this aim  we observed that the trust region concept according to \eqref{Eqrho} is not adequate for the function $\norm{\cdot}_q^q$ if $q>0$. Instead of requiring that $\norm{\ee sk}^2$ is sufficiently small, it seems to be better to bound the  relative quantities
\[\ee {\nu_-}k:=\min\{\ee sk_i/\ee zk_i\mv \ee zk_i\not=0\},\quad \ee {\nu_+}k:=\max\{\ee sk_i/\ee zk_i\mv \ee zk_i\not=0\}.\]
Clearly, if for a nonzero component $\ee zk_i$ there holds $\ee sk_i/\ee zk_i\ll-1$ then the quantity
\[\vert \ee zk_i\vert^q+\nabla \vert\cdot\vert^q(\ee zk_i)\ee sk_i+\frac 12 \nabla^2 \vert\cdot\vert^q(\ee zk_i){\ee sk_i}^2\] cannot be a good approximation for $\vert \ee zk_i+\ee sk_i\vert^q$ and the same holds true when $\ee sk_i/\ee zk_i\gg1$ in case when $q=\frac 12$.
 We thus add to the matrix $\ee Wk$ in \eqref{EqTRSubProbl_p} a diagonal matrix $\ee Dk$ with entries $\ee Dk_{ii}=\ee\zeta k/{\ee zk_i}^2$ for the nonzero components of $\ee zk$, the remaining diagonal elements of $\ee Dk$ can be set arbitrarily, e.g. zero. The parameter $\ee\zeta k$ is updated by the rule
\[\ee \zeta{k+1}=\begin{cases}2\ee \zeta k&\mbox{if ${\ee sk}^T\ee Wk\ee sk\leq 0$}\\
\ee\zeta k\sqrt{\ee\nu k_\epsilon}&\mbox{if ${\ee sk}^T\ee Wk\ee sk> 0$ and $\ee{\nu_\epsilon} k\not\in[0.8,1.2]$,}\\
\ee\zeta k&\mbox{else,}\end{cases}\]
where
\begin{gather*}\ee  {\nu_\epsilon} k:=\begin{cases}-\ee \nu k_{\epsilon-}&\mbox{if $q=1$,}\\\max\{-\ee \nu k_{\epsilon-},\ee\nu k_{\epsilon+}\}&\mbox{if $q=\frac 12$,}\end{cases}\\
\ee {\nu}k_{\epsilon-}:=\min\{\ee sk_i/(\ee zk_i+\epsilon{\rm sign\,}\ee zk_i)\mv \ee zk_i\not=0\},\quad
\ee {\nu}k_{\epsilon+}:=\max\{\ee sk_i/(\ee zk_i+\epsilon{\rm sign\,}\ee zk_i)\mv \ee zk_i\not=0\}.\end{gather*}
For the considered test examples we used $\epsilon=10^{-4}$.
With this modification the superlinear convergence is not affected because of $\ee\zeta k\to 0$ whenever $\ee sk\to 0$, but for some test problems we observed a considerable reduction both in the number of iterations and the number of backtracking steps. The maximum speedup (factor 8) was achieved  for the instance {\tt triazines4}, $\lambda_c=10^{-4}$, cf. Table \ref{TabGSSNSSNAL} below.

In Algorithm \ref{AlgCG} we used  the diagonal for preconditioning in case when the matrix $A$ is available. In case when we have at our disposal only functions for evaluating the matrix-vector products $Ax$ and $A^Ty$ we replace the diagonal of $A^TA$ by $\sigma I$, where $\sigma$ is an estimate of the mean value of the diagonal elements of $A^TA$. We estimate $\sigma$ by choosing randomly 10 vectors $x_i$ and set $\sigma:=\frac 1{10}\sum_{i=1}^{10}\norm{Ax_i}^2/\norm{x_i}^2$.

We tested our method with  9 test instances  obtained from large scale regression problems in the LIBSVM data sets \cite{ChLi11}. As suggested in
[23], for the data sets {\tt pyrim}, {\tt triazines}, {\tt abalone}, {\tt bodyfat}, {\tt housing}, {\tt mpg}, and
{\tt space\_ga}, we expand their original features by using polynomial basis functions over those features. For example, the last digit in {\tt pyrim5} indicates that an order 5 polynomial is used to generate the basis functions. In addition, we also considered 8 real-valued sparse reconstruction problems from the Sparco collection \cite{BeFrHeHeSaYi09}.  One main difference between the two problem classes is that for the former the matrix $A$ is explicitly given, whereas for the problems from the Sparco collection only operator evaluations $Ax$ and $A^Ty$ are available.

In Table \ref{TabGSSNSSNAL} we present the numerical results for these test problems  for the convex case $q=1$. We always used the origin as starting point. The first nine test problems are from the LIBSVM data sets and the remaining ones are from the Sparco collection. For each problem, the number of samples and features is denoted by $m$ and $n$, respectively, and $\lambda_{\max}$is the largest eigenvalue of $A^TA$. The regularization parameter is chosen as $\mu=\lambda_c \norm{A^Tb}_\infty$, where for each test problem we considered two values $\lambda_c\in\{10^{-3},10^{-4}\}$. In the block entitled by GSSN we report the results of our globalized SCD semismooth$^{*}$ Newton method. The column $\norm{A\bar x-b}$ gives the Euclidean norm of the residual, $nnz$ is the number of nonzero elements of the obtained solution and $\ee\lambda k$ is the final parameter in Algorithm \ref{AlgBasGSSN}. Further we report the number of iterations $iter$, the total number of CG-iterations $cg$ and the CPU-time $time$ needed. For comparison we also present the CPU-time for the semismooth Newton augmented Lagrangian method  SSNAL \cite{LiSuTo18}. We observed that the default stopping tolerance $\epsilon=10^{-6}$ for SSNAL yields rather inaccurate results and therefore we ran SSNAL also with the increased accuracy $\epsilon=10^{-10}$ which corresponds approximately with the accuracy obtained by our stopping criterion \eqref{EqStopping}. In the last column we report the CPU-time needed by ZeroFPR. All CPU times are measured in seconds.

\begin{table}[h]
{\scriptsize
\begin{tabular}{|c|c||c|c|c|c|c|c||c|c||c|}
\hline
\multicolumn{2}{|c||}{convex case: $q=1$}&\multicolumn{6}{c||}{GSSN}&\multicolumn{2}{c||}{SSNAL}&\multicolumn{1}{c|}{ZeroFPR}\\
\hline
Problem name&&&&&&&&time&time&\\
(m,n); $\lambda_{\max}$&\raisebox{1.5ex}[-1.5ex]{$\lambda_c$}&\raisebox{1.5ex}[-1.5ex]{$\norm{Ax-b}$}&\raisebox{1.5ex}[-1.5ex]{nnz}&\raisebox{1.5ex}[-1.5ex]{$\ee\lambda k$}&\raisebox{1.5ex}[-1.5ex]{iter}&\raisebox{1.5ex}[-1.5ex]{cg}&\raisebox{1.5ex}[-1.5ex]{time}&{\tiny$\epsilon=10^{-6}$}&{\tiny$\epsilon=10^{-10}$}&\raisebox{1.5ex}[-1.5ex]{time}\\
\hline
log1p.E2006.test&$10^{-3}$&21&9&$9.4\cdot10^{-7}$&43&172&28&10.8&{\bf12.3}&449\\
(3\,308, 1\,771\,946); $1.5\cdot10^7$&$10^{-4}$&13.6&1138&$7.5\cdot10^{-6}$&121&2916&92.9&53.2&{\bf63.8}&3256\\
\hline
log1p.E2006.train&$10^{-3}$&47.2&6&$1.9\cdot10^{-6}$&41&147&87.9&15.4&{\bf17.2}&998\\
(16\,087, 4\,265\,669); $5.9\cdot10^7$&$10^{-4}$&43&628&$9.4\cdot10^{-7}$&99&1906&281&76.8&{\bf102}&10437\\
\hline
space\_ga9&$10^{-3}$&6.95&14&$2.0\cdot10^{-4}$&36&283&0.88&0.62&{\bf0.75}&12.2\\
(3\,107, 5\,005 ); $4.0\cdot10^3$&$10^{-4}$&5.83&40&$8.0\cdot10^{-4}$&63&1189&4.4&1.16&{\bf1.19}&82.6\\
\hline
bodyfat7&$10^{-3}$&0.063&2&$3.2\cdot10^{-3}$&40&138&7.2&1.42&{\bf1.42}&14.3\\
(252, 116\,280); $5.3\cdot10^4$&$10^{-4}$&0.046&4&$6.4\cdot10^{-3}$&55&221&15.4&1.86&{\bf2}&31.9\\
\hline
housing7&$10^{-3}$&52&164&$5.1\cdot10^{-4}$&144&2664&33.2&2.41&{\bf2.74}&1555$^{1)}$\\
(506, 77\,520); $3.3\cdot10^5$&$10^{-4}$&30.7&290&$1.0\cdot10^{-3}$&302&17295&162&3.46&{\bf4.09}&3784$^{1)}$\\
\hline
abalone7&$10^{-3}$&139&27&$9.6\cdot10^{-4}$&73&814&8.91&1.64&{\bf1.99}&449$^{1)}$\\
(4\,177, 6\,435); $5.2\cdot10^5$&$10^{-4}$&133&63&$2.4\cdot10^{-4}$&126&3390&40.8&3.24&{\bf3.32}&1295$^{1)}$\\
\hline
mpg7&$10^{-3}$&46.4&50&$1.6\cdot10^{-3}$&66&1047&0.5&0.18&{\bf0.17}&8.15\\
(392, 3\,432 ); $1.3\cdot10^4$&$10^{-4}$&36.6&132&$1.6\cdot10^{-3}$&114&4791&2.13&0.23&{\bf0.24}&51.3\\
\hline
triazines4&$10^{-3}$&0.587&2224&$6.4\cdot10^{-6}$&394&9478&422&12.4&{\bf14.4}&12086$^{1)}$\\
(186, 557\,845); $2.1\cdot10^7$&$10^{-4}$&0.304&446&$2.5\cdot10^{-5}$&916&108291&2506&30.9&{\bf34}&45726$^{1)}$\\
\hline
pyrim5&$10^{-3}$&0.143&97&$1.3\cdot10^{-3}$&230&8146&38.5&2.65&{\bf2.67}&2075$^{1)}$\\
(74, 169\,911); $1.2\cdot10^6$&$10^{-4}$&0.070&93&$2.6\cdot10^{-3}$&393&20536&152&3.38&{\bf3.51}&3517$^{1)}$\\
\hline
\hline
blknheavi&$10^{-3}$&1.34&17&$1.4\cdot10^{0}$&111&1202&{\bf0.14}&0.19&0.19&14.0\\
(1\,024, 1\,024); $7.1\cdot10^2$&$10^{-4}$&0.155&20&$1.4\cdot10^{0}$&111&1318&{\bf0.13}&0.18&0.2&19.7\\
\hline
srcsep1&$10^{-3}$&0.459&15578&$1.3\cdot10^{-1}$&113&4416&116&245&339&{\bf82.5}\\
(29\ 166, 57\ 344 ); $3.6$&$10^{-4}$&0.064&22920&$1.3\cdot10^{-1}$&209&18176&469&476&1490&{\bf367}\\
\hline
srcsep2&$10^{-3}$&0.401&18377&$1.3\cdot10^{-1}$&126&4124&148&374&508&{\bf128}\\
(29\ 166, 86\ 016); $5.0$&$10^{-4}$&0.046&26270&$2.5\cdot10^{-1}$&273&24321&844&703&2050&{\bf709}\\
\hline
srcsep3&$10^{-3}$&4.57&30001&$1.0\cdot10^{0}$&80&961&{\bf28.4}&15.8&56&83.8\\
(196\,608, 196\,608); $5.7$&$10^{-4}$&0.753&93944&$1.0\cdot10^{0}$&148&5210&{\bf141}&83.5&254&241\\
\hline
soccer1&$10^{-3}$&690&4&$2.0\cdot10^{-7}$&7&18&{\bf0.41}&0.51&0.72&1.77\\
(3\,200, 4\,096); $3.3\cdot10^6$&$10^{-4}$&632&8&$2.0\cdot10^{-7}$&18&53&1.19&0.68&{\bf0.95}&3.55\\
\hline
soccer2&$10^{-3}$&884&4&$3.9\cdot10^{-7}$&6&14&{\bf0.36}&0.48&0.64&2.20\\
(3\,200, 4\,096); $3.3\cdot10^6$&$10^{-4}$&856&8&$1.2\cdot10^{-8}$&16&45&1.06&0.67&{\bf0.94}&4.15\\
\hline
blurrycam&$10^{-3}$&3.47&1868&$5.0\cdot10^{-1}$&31&231&{\bf2.01}&1.95&3.54&2.61\\
(65\,536, 65\,536 ); $1.0$&$10^{-4}$&1.72&6260&$5.0\cdot10^{-1}$&51&758&{\bf6.1}&4.17&9.69&8.98\\
\hline
blurspike&$10^{-3}$&0.057&2488&$4.0\cdot10^{0}$&60&544&{\bf1.23}&1.9&3.46&4.49\\
(16\,384, 16\,384 ); $1.0$&$10^{-4}$&0.021&12432&$1.0\cdot10^{0}$&92&2893&{\bf5.72}&4.91&9.4&8.19\\
\hline
\multicolumn{10}{l}{{}$^{1)}$ exceeded maximum iterations}
\end{tabular}
}
\caption{\label{TabGSSNSSNAL}Performance of GSSN in the convex case $q=1$}
\end{table}

We can see that for the problems from the LIBSVM data sets SSNAL outperforms GSSN. Whereas for the first 4 problems the difference is not so tremendous, for the remaining 5 test problems SSNAL is up to 70 times faster than GSSN. A possible explanation for this phenomena could be that for these problems the primal problem is degenerated but not the dual one, which is solved by SSNAL. Indeed, consider the matrix $\bar A$ formed by the columns of $A$ corresponding to the nonzero elements of the obtained solution $\xb$. Then for SCD regularity of $\partial \varphi$ at $(\xb,0)$ it is necessary that the matrix $\bar A$ has full column rank and it seems that this condition is violated for the problems, where GSSN does not perform well. Indeed, let us  define the numerical rank of $\bar A$ as the number
\[r(\bar A):=\max\{k\mv \sigma_k\geq 10^{-14}\sigma_1\},\]
where $\sigma_1\geq\sigma_2\geq\ldots$ denote the singular values of $\bar A$ in descending order. Then for the 5 problems {\tt housing7}--{\tt pyrim5} and for each $\lambda_c\in\{10^{-3},10^{-4}\}$ the numerical rank $r(\bar A)$, as reported  in Table \ref{TabRank}, is much smaller than the number of nonzeros in $\xb$. Note that for the problems {\tt triazines4} and {\tt pyrim5} the numerical rank is also smaller than the number of rows. This justifies the poor performance of GSSN as compared with SSNAL:

\begin{table}[h]
\begin{center}
\begin{tabular}{|c|c|c|c|c|c|c|c|c|c|c|}
\hline
&\multicolumn{2}{|c}{housing7}&\multicolumn{2}{|c}{abalone7}&\multicolumn{2}{|c}{mpg7}&\multicolumn{2}{|c}{triazines4}&\multicolumn{2}{|c|}{pyrim5}\\
\cline{2-11}
&$10^{-3}$&$10^{-4}$&$10^{-3}$&$10^{-4}$&$10^{-3}$&$10^{-4}$&$10^{-3}$&$10^{-4}$&$10^{-3}$&$10^{-4}$\\
\hline
$r(\bar A)$&91&220&18&41&34&96&120&163&55&70\\
$nnz$&164&290&27&63&50&132&2224&446&97&93\\
\hline
\end{tabular}
\caption{\label{TabRank}Rank deficiency for several degenerate test problems in the convex case}
\end{center}
\end{table}

Finally note that the final parameter $\ee \lambda k$ is much smaller than $1/\lambda_{\max}$. In the very beginning of the algorithm, $\lambda_k$ usually has the same magnitude as $1/\lambda_{\max}$, but when the algorithm approaches the (sparse) solution $\xb$, $\ee \lambda k$ is approximately the reciprocal of the largest eigenvalue of the matrix $\bar A^T\bar A$. For the problems from the LIBSVM data set, ZeroFPR performs very poor. For eight test instances, ZeroFPR failed to converge and for the remaining test instances the computation time (with exception of {\tt bodyfat7}) is very high, even when compared with GSSN.

The picture changes when we consider the test problems from the Sparco collection. Now GSSN is throughout faster than SSNAL except for the two problems {\tt soccer1} and {\tt soccer2} with $\lambda_c=10^{-4}$, where SSNAl needed slightly less CPU time. Among the three considered methods, ZeroFPR performs best for the problems {\tt srcsep1} and {\tt srcsep2}. Note that SSNAL now takes significantly more time to solve the problems with increased accuracy. This is on the one  hand due to the fact that the matrix $A$ is not available and therefore sparsity cannot be exploited, on the other hand SSNAL is an augmented Lagrangian method which converges only linearly when the penalization parameter remains bounded.

\begin{table}[h]
\begin{center}
{\scriptsize
\begin{tabular}{|c|c|c|c|c|c|c|c|c|c|c|}
\hline
Problem name, $(m,n)$&$\lambda_c$&Method&$\varphi(\ee x1)$&$t_1$&$\varphi_{opt}$&$t_{tot}$&it&$t_{avg}$&$\norm{Ax-b}$&nnz\\
\hline
&&GSSN&224&10.5&223.6&22.6&3&{\bf7.52}&20.1&8\\log1p.E2006.test&\raisebox{1.5ex}[-1.5ex]{$10^{-4}$}&HpgSRN&235.4&8.98&235.4&39.5&4&9.88&21.1&3\\
\cline{2-11}
(3\,308, 1\,771\,946)&&GSSN&135.5&42.1&130.4&354&11&32.2&9.396&637\\&\raisebox{1.5ex}[-1.5ex]{$10^{-5}$}&HpgSRN&136.4&47.2&135.5&221&11&{\bf20.1}&12.53&434\\
\hline
&&GSSN&1158&33.3&1158&50.3&2&{\bf25.2}&45.94&9\\log1p.E2006.train&\raisebox{1.5ex}[-1.5ex]{$10^{-4}$}&HpgSRN&1158&41.9&1155&684&7&97.7&46.77&3\\
\cline{2-11}
(16\,087, 4\,265\,669)&&GSSN&1004&82.3&1003&747&9&{\bf83}&41.54&314\\&\raisebox{1.5ex}[-1.5ex]{$10^{-5}$}&HpgSRN&1003&188&1003&1.27e+03&4&316&42.69&221\\
\hline
&&GSSN&34.21&0.3&34.06&0.79&3&{\bf0.26}&6.63&8\\space\_ga9&\raisebox{1.5ex}[-1.5ex]{$10^{-3}$}&HpgSRN&34.29&0.53&34.29&8.59&4&2.15&6.771&7\\
\cline{2-11}
(3\,107, 5\,005 )&&GSSN&19.95&0.3&19.95&0.59&2&{\bf0.29}&5.769&24\\&\raisebox{1.5ex}[-1.5ex]{$10^{-4}$}&HpgSRN&19.93&0.4&19.93&53.5&4&13.4&5.78&23\\
\hline
&&GSSN&0.3282&0.39&0.3282&0.64&2&{\bf0.32}&0.1432&2\\bodyfat7&\raisebox{1.5ex}[-1.5ex]{$10^{-3}$}&HpgSRN&0.3187&0.09&0.3187&0.73&2&0.37&0.3016&1\\
\cline{2-11}
(252, 116\,280)&&GSSN&0.0344&0.31&0.0344&0.5&2&{\bf0.25}&0.0481&2\\&\raisebox{1.5ex}[-1.5ex]{$10^{-4}$}&HpgSRN&0.0344&0.12&0.0344&1.39&2&0.7&0.0481&2\\
\hline
&&GSSN&2102&1.68&2102&3.22&2&{\bf1.61}&44.83&51\\housing7&\raisebox{1.5ex}[-1.5ex]{$10^{-3}$}&HpgSRN&2033&13.1&2007&26.1&3&8.69&43.15&44\\
\cline{2-11}
(506, 77\,520)&&GSSN&680.4&1.93&651.9&10.8&7&{\bf1.54}&26.24&140\\&\raisebox{1.5ex}[-1.5ex]{$10^{-4}$}&HpgSRN&658.8&33.1&648.1&62&6&10.3&28.32&103\\
\hline
&&GSSN&1.029e+04&0.78&1.024e+04&3.67&5&{\bf0.73}&135.5&13\\abalone7&\raisebox{1.5ex}[-1.5ex]{$10^{-3}$}&HpgSRN&1.023e+04&1.44&1.023e+04&58.7&4&14.7&135.9&9\\
\cline{2-11}
(4\,177, 6\,435)&&GSSN&9025&0.69&9015&2.32&3&{\bf0.77}&132.5&37\\&\raisebox{1.5ex}[-1.5ex]{$10^{-4}$}&HpgSRN&9002&4.18&8990&180&4&45.1&132.2&28\\
\hline
&&GSSN&1401&0.11&1384&0.85&7&{\bf0.12}&43.41&23\\mpg7&\raisebox{1.5ex}[-1.5ex]{$10^{-3}$}&HpgSRN&1392&0.76&1351&2.32&4&0.58&43.69&20\\
\cline{2-11}
(392, 3\,432 )&&GSSN&759.2&0.14&697.9&0.92&8&{\bf0.12}&31.15&96\\&\raisebox{1.5ex}[-1.5ex]{$10^{-4}$}&HpgSRN&746.1&6.16&668.3&70.2&9&7.8&29.52&78\\
\hline
&&GSSN&0.2808&7.13&0.2808&14.5&2&{\bf7.23}&0.3555&73\\triazines4&\raisebox{1.5ex}[-1.5ex]{$10^{-4}$}&HpgSRN&0.2738&34&0.2738&68.6&3&22.9&0.3578&67\\
\cline{2-11}
(186, 557\,845)&&GSSN&0.0706&8.66&0.0682&33.2&4&{\bf8.29}&0.2764&108\\&\raisebox{1.5ex}[-1.5ex]{$10^{-5}$}&HpgSRN&0.0699&91.4&0.069&169&4&42.2&0.2902&94\\
\hline
&&GSSN&0.0391&0.95&0.0352&2.77&3&{\bf0.92}&0.1011&37\\pyrim5&\raisebox{1.5ex}[-1.5ex]{$10^{-4}$}&HpgSRN&0.0377&3.22&0.0338&15.4&7&2.2&0.1046&29\\
\cline{2-11}
(74, 169\,911)&&GSSN&6.644e-03&1.49&6.044e-03&12.3&5&{\bf1.46}&0.0682&54\\&\raisebox{1.5ex}[-1.5ex]{$10^{-5}$}&HpgSRN&6.728e-03&1.99&6.085e-03&41.3&11&3.75&0.0689&47\\
\hline
\hline
\end{tabular}
}
\caption{\label{Tab_q12_LIBSVM_Sol}Numerical results for problems from the LIBSVM dataset in the nonconvex case $q=\frac 12$  with starting point equal to the solution of the convex case}
\end{center}
\end{table}
\begin{table}[h]
\begin{center}
{\scriptsize
\begin{tabular}{|c|c|c|c|c|c|c|c|c|c|c|}
\hline
Problem name, $(m,n)$&$\lambda_c$&Method&$\varphi(\ee x1)$&$t_1$&$\varphi_{opt}$&$t_{tot}$&it&$t_{avg}$&$\norm{Ax-b}$&nnz\\
\hline
&&GSSN&5.581&0.01&5.581&0.02&2&{\bf0.01}&0.108&12\\blknheavi&\raisebox{1.5ex}[-1.5ex]{$10^{-3}$}&ZeroFPR&5.581&2.18&5.581&2.32&2&1.16&0.108&12\\
\cline{2-11}
(1\,024, 1\,024)&&GSSN&0.5612&0.01&0.5612&0.02&2&{\bf0.009}&0.0799&12\\&\raisebox{1.5ex}[-1.5ex]{$10^{-4}$}&ZeroFPR&0.5612&1.91&0.5612&2.04&2&1.02&0.0799&12\\
\hline
&&GSSN&1.047&34.7&1.042&168&11&{\bf15.3}&0.3157&10790\\srcsep1&\raisebox{1.5ex}[-1.5ex]{$10^{-4}$}&ZeroFPR&1.049&38.5&1.047&125&6&20.8&0.3366&10685\\
\cline{2-11}
(29\ 166, 57\ 344) &&GSSN&0.0124&119&0.0122&387&11&{\bf35.2}&0.0164&20586\\&\raisebox{1.5ex}[-1.5ex]{$10^{-6}$}&ZeroFPR&0.0123&143&0.0122&640&11&58.2&0.0183&20455\\
\hline
&&GSSN&1.109&50.3&1.107&179&9&{\bf19.9}&0.325&11582\\srcsep2&\raisebox{1.5ex}[-1.5ex]{$10^{-4}$}&ZeroFPR&1.111&50.6&1.11&121&4&30.3&0.3376&11485\\
\cline{2-11}
(29\ 166, 86\ 016)&&GSSN&0.013&118&0.0128&719&11&{\bf65.3}&0.016&21958\\&\raisebox{1.5ex}[-1.5ex]{$10^{-6}$}&ZeroFPR&0.013&161&0.0129&732&11&66.5&0.0185&21739\\
\hline
&&GSSN&19.07&19.2&19.07&49.8&5&{\bf9.95}&2.248&32641\\srcsep3&\raisebox{1.5ex}[-1.5ex]{$10^{-4}$}&ZeroFPR&19.07&47.6&19.06&118&4&29.4&2.25&32590\\
\cline{2-11}
(196\,608, 196\,608)&&GSSN&2.559&27.8&2.559&61.5&4&{\bf15.4}&0.5488&75887\\&\raisebox{1.5ex}[-1.5ex]{$10^{-5}$}&ZeroFPR&2.558&79.3&2.558&206&4&51.5&0.5536&75538\\
\hline
&&GSSN&9.495e+05&0.64&9.495e+05&0.86&2&{\bf0.43}&740.9&3\\soccer1&\raisebox{1.5ex}[-1.5ex]{$10^{-3}$}&ZeroFPR&9.495e+05&2.69&9.495e+05&3.64&2&1.82&740.9&3\\
\cline{2-11}
(3\,200, 4\,096)&&GSSN&2.734e+05&0.54&2.701e+05&1.7&3&{\bf0.57}&517.1&9\\&\raisebox{1.5ex}[-1.5ex]{$10^{-4}$}&ZeroFPR&2.694e+05&6.04&2.694e+05&11.5&2&5.77&532.4&8\\
\hline
&&GSSN&9.931e+05&0.34&9.931e+05&1.62&2&{\bf0.81}&885.4&4\\soccer2&\raisebox{1.5ex}[-1.5ex]{$10^{-3}$}&ZeroFPR&9.931e+05&3.08&9.931e+05&11.7&2&5.84&885.4&4\\
\cline{2-11}
(3\,200, 4\,096)&&GSSN&4.236e+05&0.79&4.236e+05&4.4&2&{\bf2.2}&804.7&8\\&\raisebox{1.5ex}[-1.5ex]{$10^{-4}$}&ZeroFPR&4.236e+05&31.4&4.236e+05&51.5&2&25.8&804.7&8\\
\hline
&&GSSN&6.2&9.89&6.197&56.7&8&{\bf7.09}&1.965&2683\\blurrycam&\raisebox{1.5ex}[-1.5ex]{$10^{-4}$}&ZeroFPR&6.203&11&6.203&26.3&3&8.76&1.957&2728\\
\cline{2-11}
(65\,536, 65\,536 )&&GSSN&1.791&11.8&1.702&106&5&{\bf21.3}&1.278&8457\\&\raisebox{1.5ex}[-1.5ex]{$10^{-5}$}&ZeroFPR&1.791&7.92&1.701&109&5&21.8&1.279&8417\\
\hline
&&GSSN&2.924e-03&0.91&2.924e-03&2.97&4&{\bf0.74}&0.0569&965\\blurspike&\raisebox{1.5ex}[-1.5ex]{$10^{-4}$}&ZeroFPR&2.924e-03&4.53&2.924e-03&14.8&4&3.71&0.0577&867\\
\cline{2-11}
(16\,384, 16\,384 )&&GSSN&8.029e-04&7.37&8.028e-04&23.1&5&{\bf4.61}&0.0157&9917\\&\raisebox{1.5ex}[-1.5ex]{$10^{-5}$}&ZeroFPR&8.032e-04&13.1&8.024e-04&51.1&6&8.52&0.0159&9838\\
\hline
\end{tabular}
}
\caption{\label{Tab_q12_Sparco_Sol}Numerical results for test problems from the Sparco collection in the nonconvex case $q=\frac 12$  with starting point equal to the solution of the convex case}
\end{center}
\end{table}

Let us now turn to the nonconvex case $q=\frac 12$. We observed that the considered test problems have many local minima and, dependent on the chosen starting point, our method may converge to any of these minima. Therefore we implemented the following heuristic to obtain a better solution:
\begin{algorithm}[Heuristic for the nonconvex case $q=\frac12$]\label{AlgHeur}\ \\
  {\bf Input: } Starting point $\ee x0\in \R^n$\\
  {\bf Initialization: } $j \leftarrow 0$, $\varphi_{\rm opt}\leftarrow\infty$, $j_{\rm opt}\leftarrow0$;\\
  1. \begin{minipage}[t]{\myAlgBox} Compute a (local) solution $\ee {\bar x}j$ using $\ee xj$ as starting point. Let $\ee\lambda j$ denote the final parameter $\ee\lambda k$ when the method terminates.\end{minipage}\\
  2. \begin{minipage}[t]{\myAlgBox}
  {\tt if $\varphi(\ee{\bar x}j)<\varphi_{\rm opt}$ then $\varphi_{\rm opt}\leftarrow\varphi(\ee{\bar x}j)$, $j_{\rm opt}\leftarrow j$.}
  \end{minipage}\\
  3. \begin{minipage}[t]{\myAlgBox}
  {\tt if $j>\max\{j_{\rm opt}+2,10\} \vee \big(j> 1 \wedge \varphi(\ee{\bar x}j)==\varphi(\ee {\bar x}{j-1})\big)$ then stop.}
  \end{minipage}\\
  4. \begin{minipage}[t]{\myAlgBox}
  Compute $\ee zj\in T_{1000\ee\lambda j}(\ee{\bar x}j)$, denote $\ee Ij:=\{i\mv \ee zj_i\not=0\}$ and compute
  an (approximate) solution $\ee{\triangle x}j$ of the system
  \begin{equation}\label{EqNormalEq}\min_{\triangle x\in\R^n} \frac12\norm{A(\ee{\bar x}j+\triangle x)-b}^2\quad\mbox{ s.t. }\quad\triangle x_i=0,\ i\not\in \ee Ij.\end{equation}
  \end{minipage}\\
  5. \begin{minipage}[t]{\myAlgBox}
  Set $\ee x{j+1}\leftarrow\ee {\bar x}j+\frac 14\ee{\triangle x}j$, $j\leftarrow j+1$ and  go to Step 1.
  \end{minipage}\\
  \end{algorithm}
  In fact, we solve the system \eqref{EqNormalEq} rather imprecisely by means of a CG method stopping when we decreased the initial residual by a factor of $0.1$. The idea behind Algorithm \ref{AlgHeur} is to use a starting point $\ee x{j+1}$ having a smaller residual than the previously obtained local solution $\ee{\bar x}j$, i.e., $\norm{A\ee x{j+1}-b}<\norm{A\ee{\bar x}j-b}$.

  In most cases, the first iteration of the heuristic is the most time consuming one. In Table \ref{Tab_q12_LIBSVM_Sol} we present some results of Algorithm \ref{AlgHeur} when applied to the nine  test problems from the LIBSVM data set  with starting point as the corresponding solution of the convex case with $\lambda_c=10^{-4}$. For each problem we report our findings for two values of $\lambda_c$. We compare our method with the very recent method HpgSRN \cite{WuPaYa23}, a regularized Newton method for $\ell_q$-norm composite optimization problems. In \cite{WuPaYa23}, the reader can find a comparison of HpgSRN with other methods, including ZeroFPR, showing that HpgSRN is superior to ZeroFPR on this problem class. Note that HpgSRN uses also a proximal gradient approach for ensuring global convergence and therefore the parameter $\ee \lambda j$ in Algorithm \ref{AlgHeur} is defined. We state for the first iteration the objective function value $\varphi(\ee x1)$ and the time $t_1$ needed to compute the first local minimizer $\ee x1$. Finally  we give the best obtained objective function value $\varphi_{\rm opt}$ and the total time $t_{tot}$ for running Algorithm \ref{AlgHeur}. Further we report the number $it$ of loop passes   and the average time $t_{avg}=t_{tot}/it$  needed for computing one local minimizer. In the columns entitled with $\norm{Ax-b}$ and $nnz$ we present the residual and the number of nonzero elements for the best local solution.

  As we can see in Table \ref{Tab_q12_LIBSVM_Sol}, in most  cases the heuristic could not improve the first local minimizer very much. In seven test instances, which can be identified by $it=2$, the initial local solution could not be improved at all. This indicates that the solution of the convex problem ($q=1$) is a good starting point. However, for the two test problems {\tt mpg7} ($\lambda_c=10^{-4}$) and {\tt pyrim5} ($\lambda_c=10^{-4},10^{-5}$) we obtained a decrease in the objective function of approximately 10 percent. Concerning the computation time we see that, with the exception of the instance {\tt log1p.E2006.test} ($\lambda_c=10^{-5}$), GSSN performs better than HpgSRN. When comparing the average computation time $t_{avg}$, the difference can be up to a factor 65 as in {\tt mpg7} ($\lambda_c=10^{-4}$). There is also another difference in the run-time behaviour between GSSN and HpgSRN: For GSSN, in all cases the average time  $t_{avg}$ is smaller or at most slightly larger than $t_1$, indicating that the first pass of the loop in Algorithm \ref{AlgHeur} tends to be the most time consuming one. However, for HpgSRN the average time $t_{avg}$ is considerably larger that the time $t_1$ for the problems {\tt log1p.E2006.train, space\_ga9, bodyfat7, abalone7.}

Finally, comparing the results from Table \ref{Tab_q12_LIBSVM_Sol}  with the ones of Table \ref{TabGSSNSSNAL} we can see that nonconvex regularization $q=\frac 12$ tends to produce a smaller residual $\norm{Ax-b}$ with less nonzero elements $nnz$. This is in accordance with other papers in this field \cite{XuZhWaChLi10}.

In Table \ref{Tab_q12_Sparco_Sol} we present our numerical results for the seven test problems from the Sparco collection.  Since HpgSRN requires that the matrix $A$ is available, we could not apply HpgSRN to these problems and therefore we compare GSSN with ZeroFPR. The average computation time per iteration is throughout smaller for GSSN, but for the three  test instances
{\tt srcsep1, srcsep2, blurrycam} with $\lambda_c=10^{-4}$ the overall computation time for GSSN was larger due to the higher number of iterations required by Algorithm \ref{AlgHeur}.
\begin{table}
\begin{center}
{\tiny
\begin{tabular}{|c|c|c|c|c|c|c|c|c|c|c|}
\hline
Problem name, $(m,n)$&$\lambda_c$&Method&$\varphi(\ee x1)$&$t_1$&$\varphi_{opt}$&$t_{tot}$&it&$t_{avg}$&$\norm{Ax-b}$&nnz\\
\hline
&&GSSN&234.9&35.4&223.1&49.3&3&{\bf16.4}&20.17&6\\log1p.E2006.test&\raisebox{1.5ex}[-1.5ex]{$10^{-4}$}&HpgSRN&234.6&109&234.6&139&4&34.8&20.93&5\\
\cline{2-11}
(3\,308, 1\,771\,946)&&GSSN&156.3&98.8&135.4&426&11&{\bf38.7}&9.693&659\\&\raisebox{1.5ex}[-1.5ex]{$10^{-5}$}&HpgSRN&149.6&727&146.2&923&11&83.9&13.25&418\\
\hline
&&GSSN&1161&129&1156&171&3&{\bf57}&46.3&7\\log1p.E2006.train&\raisebox{1.5ex}[-1.5ex]{$10^{-4}$}&HpgSRN&1159&399&1158&772&5&154&46.84&5\\
\cline{2-11}
(16\,087, 4\,265\,669)&&GSSN&1012&293&999.3&681&5&{\bf136}&41.59&280\\&\raisebox{1.5ex}[-1.5ex]{$10^{-5}$}&HpgSRN&-$^{1)}$&-&-&-&-&-&-&-\\
\hline
&&GSSN&38.35&0.85&33.71&2.08&5&{\bf0.42}&6.662&9\\space\_ga9&\raisebox{1.5ex}[-1.5ex]{$10^{-3}$}&HpgSRN&36.44&0.3&33.64&8.62&5&1.72&6.767&7\\
\cline{2-11}
(3\,107, 5\,005 )&&GSSN&21.73&3.01&20.91&3.86&4&{\bf0.96}&6.138&18\\&\raisebox{1.5ex}[-1.5ex]{$10^{-4}$}&HpgSRN&20.5&4.36&20.19&105&8&13.1&5.719&22\\
\hline
&&GSSN&0.3187&2.14&0.3187&2.36&2&1.18&0.3016&1\\bodyfat7&\raisebox{1.5ex}[-1.5ex]{$10^{-3}$}&HpgSRN&0.3187&1.18&0.3187&1.78&2&{\bf0.89}&0.3016&1\\
\cline{2-11}
(252, 116\,280)&&GSSN&0.0344&6.88&0.0344&7.07&2&3.53&0.0481&2\\&\raisebox{1.5ex}[-1.5ex]{$10^{-4}$}&HpgSRN&0.0344&4.96&0.0344&5.64&2&{\bf2.82}&0.0481&2\\
\hline
&&GSSN&2229&7.53&2165&13.5&4&{\bf3.38}&54.47&29\\housing7&\raisebox{1.5ex}[-1.5ex]{$10^{-3}$}&HpgSRN&2113&35.6&2113&40.4&2&20.2&52.6&33\\
\cline{2-11}
(506, 77\,520)&&GSSN&750.8&16.9&657.5&40.2&11&{\bf3.65}&25.25&128\\&\raisebox{1.5ex}[-1.5ex]{$10^{-4}$}&HpgSRN&729.5&261&728.6&270&3&89.8&30.79&108\\
\hline
&&GSSN&1.089e+04&6.21&1.036e+04&14.6&8&{\bf1.82}&136.1&11\\abalone7&\raisebox{1.5ex}[-1.5ex]{$10^{-3}$}&HpgSRN&1.043e+04&11.3&1.037e+04&40.4&3&13.5&137&11\\
\cline{2-11}
(4\,177, 6\,435)&&GSSN&9056&14.5&9020&16.4&3&{\bf5.47}&132.1&35\\&\raisebox{1.5ex}[-1.5ex]{$10^{-4}$}&HpgSRN&9039&90.9&9001&358&6&59.7&132&31\\
\hline
&&GSSN&1400&0.31&1339&1&9&{\bf0.11}&43.48&20\\mpg7&\raisebox{1.5ex}[-1.5ex]{$10^{-3}$}&HpgSRN&1395&0.98&1368&4.47&6&0.75&45.2&19\\
\cline{2-11}
(392, 3\,432 )&&GSSN&791.3&0.92&692.2&1.81&8&{\bf0.23}&30.81&83\\&\raisebox{1.5ex}[-1.5ex]{$10^{-4}$}&HpgSRN&760.2&29.7&679.1&77.2&7&11&29.99&81\\
\hline
&&GSSN&0.4481&75.3&0.4481&85.6&2&{\bf42.8}&0.797&41\\triazines4&\raisebox{1.5ex}[-1.5ex]{$10^{-4}$}&HpgSRN&0.3458&573&0.3458&589&2&295&0.5631&64\\
\cline{2-11}
(186, 557\,845)&&GSSN&0.0806&174&0.0806&180&2&{\bf89.9}&0.3068&111\\&\raisebox{1.5ex}[-1.5ex]{$10^{-5}$}&HpgSRN&-$^{1)}$&-&-&-&-&-&-&-\\
\hline
&&GSSN&0.0525&7.14&0.0525&7.96&2&{\bf3.98}&0.1923&30\\pyrim5&\raisebox{1.5ex}[-1.5ex]{$10^{-4}$}&HpgSRN&0.0405&46&0.0405&49.3&3&16.4&0.152&26\\
\cline{2-11}
(74, 169\,911)&&GSSN&7.284e-03&16.7&7.266e-03&21.9&3&{\bf7.3}&0.0756&45\\&\raisebox{1.5ex}[-1.5ex]{$10^{-5}$}&HpgSRN&7.167e-03&281&7.093e-03&287&3&95.8&0.072&47\\
\hline
\hline
&&GSSN&38.85&0.1&5.581&0.13&3&{\bf0.04}&0.108&12\\blknheavi&\raisebox{1.5ex}[-1.5ex]{$10^{-3}$}&ZeroFPR&5.581&7.57&5.581&7.71&2&3.85&0.108&12\\
\cline{2-11}
(1\,024, 1\,024)&&GSSN&3.603&0.23&0.5612&0.26&3&{\bf0.09}&0.0799&12\\&\raisebox{1.5ex}[-1.5ex]{$10^{-4}$}&ZeroFPR&0.5612&8.17&0.5612&8.3&2&4.15&0.0799&12\\
\hline
&&GSSN&1.939&83.8&1.443&181&11&{\bf16.5}&0.3931&12867\\scrsep1&\raisebox{1.5ex}[-1.5ex]{$10^{-4}$}&ZeroFPR&1.544&103&1.421&297&11&27&0.3992&12219\\
\cline{2-11}
(29\ 166, 57\ 344 )&&GSSN&0.0229&697&0.0219&1.6e+03&11&{\bf146}&0.0174&23223\\&\raisebox{1.5ex}[-1.5ex]{$10^{-6}$}&ZeroFPR&0.023&1.2e+03&0.0213&3e+03&11&273&0.0199&22648\\
\hline
&&GSSN&1.746&182&1.599&471&11&{\bf42.8}&0.4032&15313\\scrsep2&\raisebox{1.5ex}[-1.5ex]{$10^{-4}$}&ZeroFPR&1.922&292&1.764&621&11&56.5&0.4805&15253\\
\cline{2-11}
(29\ 166, 86\ 016)&&GSSN&0.0197&1.08e+03&0.0196&2.03e+03&11&{\bf185}&0.0155&26232\\&\raisebox{1.5ex}[-1.5ex]{$10^{-6}$}&ZeroFPR&0.0277&5.86e+03&0.0267&8.53e+03&11&775&0.0169&27452\\
\hline
&&GSSN&23.11&140&19.37&260&11&{\bf23.7}&2.467&30242\\scrsep3&\raisebox{1.5ex}[-1.5ex]{$10^{-4}$}&ZeroFPR&21.38&687&19.34&1.17e+03&11&107&2.286&32197\\
\cline{2-11}
(196\,608, 196\,608)&&GSSN&3.02&470&2.671&917&11&{\bf83.3}&0.5712&75142\\&\raisebox{1.5ex}[-1.5ex]{$10^{-5}$}&ZeroFPR&2.739&2.12e+03&2.684&2.44e+03&10&244&0.5507&75951\\
\hline
&&GSSN&1.406e+06&17.6&9.495e+05&18.6&3&{\bf6.21}&740.9&3\\soccer1&\raisebox{1.5ex}[-1.5ex]{$10^{-3}$}&ZeroFPR&5.126e+10$^{2)}$&2.56e+03&-&-&-&-&-&- \\
\cline{2-11}
(3\,200, 4\,096)&&GSSN&3.158e+05&32.6&2.701e+05&34.3&3&{\bf11.4}&517.1&9\\&\raisebox{1.5ex}[-1.5ex]{$10^{-4}$}&ZeroFPR&4.584e+09$^{2)}$&2.55e+03 &-&-&-&-&-&- \\
\hline
&&GSSN&9.945e+05&15.9&9.931e+05&19.2&3&{\bf6.4}&885.4&4\\soccer2&\raisebox{1.5ex}[-1.5ex]{$10^{-3}$}&ZeroFPR&3.011e+10$^{2)}$&2.12e+03  &-&-&-&-&-&- \\
\cline{2-11}
(3\,200, 4\,096)&&GSSN&4.330e+05&28.8&4.189e+05&30.6&3&{\bf10.2}&691.4&19\\&\raisebox{1.5ex}[-1.5ex]{$10^{-4}$}&ZeroFPR&3.885e+09$^{2)}$&2.13e+03  &-&-&-&-&-&- \\
\hline
&&GSSN&6.721&17.1&6.194&75.4&8&9.42&1.984&2604\\blurrycam&\raisebox{1.5ex}[-1.5ex]{$10^{-4}$}&ZeroFPR&6.527&13.4&6.199&56&6&{\bf9.33}&1.999&2573\\
\cline{2-11}
(65\,536, 65\,536 )&&GSSN&1.865&34.3&1.708&175&9&{\bf19.5}&1.25&8791\\&\raisebox{1.5ex}[-1.5ex]{$10^{-5}$}&ZeroFPR&1.744&40.2&1.712&205&9&22.8&1.251&8797\\
\hline
&&GSSN&2.937e-03&2.23&2.919e-03&4.96&5&{\bf0.99}&0.0572&935\\blurspike&\raisebox{1.5ex}[-1.5ex]{$10^{-4}$}&ZeroFPR&2.930e-03&4.69&2.918e-03&17.2&5&3.45&0.0573&913\\
\cline{2-11}
(16\,384, 16\,384 )&&GSSN&1.320e-03&5.03&8.022e-04&33.2&8&{\bf4.15}&0.0162&9713\\&\raisebox{1.5ex}[-1.5ex]{$10^{-5}$}&ZeroFPR&8.066e-04&13.5&8.034e-04&83.7&11&7.61&0.0156&9919\\
\hline
\multicolumn{11}{l}{{}$^{1)}$ HpgSRN exceeded time limit}\quad{}$^{2)}$ ZeroFPR exceeded maximum iterations
\end{tabular}
}
\caption{\label{Tab_q12_LIBSVM_Rand}Numerical results in the nonconvex case $q=\frac 12$  with random starting point}
\end{center}
\end{table}

We also performed tests with randomly chosen starting points, the results are given in Table \ref{Tab_q12_LIBSVM_Rand}. Due to the worse starting point, the first iteration of Algorithm \ref{AlgHeur} took significantly longer and for the  test instances {\tt log1p.E2006.train ($\lambda_c=10^{-5}$), triazines4 ($\lambda_c=10^{-5}$), soccer1 ($\lambda_c=10^{-4},10^{-5}$), soccer2 ($\lambda_c=10^{-4},10^{-5}$)},  either HpgSRN or ZeroFPR was not able to solve the first subproblem whereas  GSSN worked well. For GSSN, the times $t_1$ for solving the first subproblem are often of the same magnitude as the corresponding times in Table \ref{TabGSSNSSNAL} and for the problems {\tt housing7, triazines4, pyrim5}, where GSSN converged in the convex case $q=1$ relatively slowly due to rank deficiency, the first problem could be very efficiently solved in the nonconvex case $q=\frac 12$. The objective function value $\varphi(\ee x1)$ of the first local minimizer is  sometimes significantly larger than in Tables \ref{Tab_q12_LIBSVM_Sol},\ref{Tab_q12_Sparco_Sol} when a good starting point is available, but the heuristic is able to compensate this in many cases. For the test instances {\tt log1p.E2006.test, log1p.E2006.train, space\_ga9 ($\lambda_c=10^{-3}$), bodyfat7 ($\lambda_c=10^{-3}$), mpg7 ($\lambda_c=10^{-3}$), soccer2 ($\lambda_c=10^{-4}$),  blurrycam ($\lambda_c=10^{-4}$), blurspike} we obtained even a better objective function value $\varphi_{\rm opt}$ with the randomly chosen starting point. This shows that the heuristic of Algorithm \ref{AlgHeur}, although very simple, works well for our test problems. It is beyond the scope of this paper to find better methods approximating the global minimizer in the nonconvex case.

Finally we want to mention a remarkable property of our method in the nonconvex case $q=\frac 12$. We observed that the function not only has a lot of local minima but also many stationary points which are locally non-optimal. Let us therefore have a closer look on second-order necessary conditions for a local minimizer. If $\xb$ is a local minimizer, then it is clearly also a local minimizer for the problem
\[\min \frac12 \norm{Ax-b}^2+\mu\sum_{i\in I(\xb)}\sqrt{\vert x_i\vert}\quad\mbox{subject to}\quad x_i=0,\ i\not\in I(\xb),\]
where $I(\xb)=\{i\mv\xb_i\not=0\}$. The objective for this problem is twice continuously differentiable near $\xb$ and therefore the second-order necessary KKT-conditions hold at $\xb$. Thus the matrix $A_{I(\xb)}^TA_{I(\xb)}+\bar W_{I(\xb)}$ must be positive semidefinite, where $A_{I(\xb)}$ denotes the matrix with columns $A_i$, $i\in I(\xb)$ and $\bar W_{I(\xb)}$ is a diagonal matrix with entries $-\frac{\mu}4\vert \bar x_i\vert^{-3/2}$, $i\in I(\xb)$. Whenever the number of nonzeros at the returned solution was small enough, i.e., less than 5000, we tested this condition.  We observed that it was always fulfilled for the solutions returned by GSSN, whereas it failed to hold in rare cases for HpgSRN and ZeroFPR. However, we want to point out that this behaviour of GSSN is due to the high accuracy in our termination criterion \eqref{EqStopping}. We observed that sometimes the iterates produced by GSSN get very close to a non-optimal stationary point, until the CG-method of Algorithm \ref{AlgCG} terminated with a direction of negative curvature and the method could escape from the stationary point along the search direction computed by \eqref{EqTwoDimTRProbl}.

\section{Conclusion}

We presented GSSN, a globalized SCD semismooth$^{*}$ Newton method, which is a hybrid of some variant of a proximal gradient method ensuring global convergence and the SCD semismooth$^{*}$ Newton method yielding fast convergence near the solution. The method can be applied to nonconvex nonsmooth optimization problem, where the nonsmooth part can be rather general. Compared with other methods, GSSN has the advantage that superlinear convergence can be assured under weak assumptions.

\section*{Declarations}

\noindent{\bf Conflict of interest.} The author has no competing interests to declare that are relevant to the content of this article.

\noindent{\bf Data availability. }
The LIBSVM data sets analysed during the current study are publicly available from \url{https://www.csie.ntu.edu.tw/~cjlin/libsvmtools/datasets}. These data sets are collected
from 10-K Corpus \url{http://www.cs.cmu.edu/~ark/10K} and the UCI data repository \url{http://archive.ics.uci.edu/datasets}. The code package of ZeroFPR can be downloaded from \url{https://github.com/kul-optec/ForBES}. The code for SSNAL is online available from \url{https://github.com/MatOpt/SuiteLasso} and also includes  the code for generating the test examples from the Sparco collection.  HpgSRN is available from \url{https://github.com/YuqiaWU/HpgSRN}.

\end{document}